\documentclass[11pt]{amsart}

\usepackage[utf8]{inputenc}
\usepackage[T1]{fontenc}

\usepackage[english]{babel}

\usepackage{amsmath,amsfonts,amssymb}
\usepackage{mathrsfs}
\usepackage{dsfont} 
\usepackage{mathtools} 
\usepackage{esint} 
\usepackage{xfrac} 
\usepackage{defs} 
\usepackage{fonts} 

\usepackage{amsthm}
\usepackage[shortlabels]{enumitem} 
\usepackage[colorlinks=true,urlcolor=blue, citecolor=red,linkcolor=blue,linktocpage,pdfpagelabels, bookmarksnumbered,bookmarksopen]{hyperref} 
\usepackage[paper=a4paper,hmargin=1in,vmargin=1in, marginparwidth=0.8in]{geometry} 

\usepackage[colorinlistoftodos,prependcaption,linecolor=blue,backgroundcolor=blue!25,bordercolor=blue,textsize=tiny]{todonotes}
\usepackage{cite} 
\usepackage{cleveref}
\usepackage{comment}

plus 6pt minus 12pt
	\numberwithin{equation}{section}

	\newtheorem{theorem}{Theorem}[section]
	\newtheorem{proposition}[theorem]{Proposition}
	\theoremstyle{definition}
	\newtheorem{definition}[theorem]{Definition}
	\theoremstyle{plain}
	\newtheorem{lemma}[theorem]{Lemma}
	\newtheorem{corollary}[theorem]{Corollary}
	\theoremstyle{remark}
	\newtheorem{remark}[theorem]{Remark}
	
	\title[Weighted nonlocal operators]{Weighted nonlocal operators and their applications in semi-supervised learning}
	\author{Qiang Du, James M. Scott}
	\address[Qiang Du]{Department of Applied Physics and Applied Mathematics, and Data Science Institute, Columbia University, 500 W. 120th St., New York, NY 10027, USA.}
    \email{qd2125@columbia.edu}

\address[James M. Scott]{Department of Applied Physics and Applied Mathematics, Columbia University, 500 W. 120th St., New York, NY 10027, USA.}
\email{jms2555@columbia.edu}

        	
\begin{document}

    \begin{abstract}
    Motivated by problems in machine learning, we study a class of variational problems characterized by nonlocal operators.
These operators are characterized by power-type weights, which are singular at a portion of the boundary. 
We identify a range of exponents on these weights for which the variational Dirichlet problem is well-posed. This range is determined by the ambient dimension of the problem, the growth rate of the nonlocal functional, and the dimension of the boundary portion on which the Dirichlet data is prescribed.
We show the variational convergence of solutions to solutions of local weighted Sobolev functionals in the event of vanishing nonlocality.
		\end{abstract}

		\maketitle

\section{Introduction}

        In data analytics and machine learning, one popular paradigm is to estimate a function using a set of labeled data. However, it is often challenging to obtain samples of labeled data, experimentally or computationally. Meanwhile, in many applications, 
        unlabeled data can be relatively more readily accessible.
        With mixed samples of labeled and unlabeled data, \textit{semi-supervised learning} seeks to use labeled data to assign labels to unlabeled data 
       \cite{van2020survey}.
		Among various strategies for extending labeled data for semi-supervised learning, graph-Laplacian-based learning algorithms and their variants have been proposed and used successfully \cite{zhu2003semi,nadler2009semi,belkin2008towards,streicher2023graph,weihs2024consistency}. Our present work is motivated by these studies and their close connections to a wide range of research on the various forms of graph Laplacians and diffusion maps, as well as their continuum representations by partial differential operators and nonlocal operators, in many applications such as supervised and unsupervised learning, image processing, the modeling of point clouds, and discretizations of manifolds \cite{berry2016variable,coifman2006diffusion,Gilboa2008Nonlocal,he2005laplacian,antil2021fractional,elmoataz2015p,burago2015graph,maskey2024fractional,roith2023continuum,Shi2017Convergence}.
     
		When the set of labeled data is sparse, i.e., the data are labeled at a low rate, poor approximations may be produced in the large graph limit \cite{Calder2020Properly,dong2020cure,streicher2023graph}.
         To improve performance, various approaches have been proposed. For example, in \cite{dong2020cure}, it was shown that the second-order Laplacian can be replaced by high-order elliptic operators to yield well-defined problems due to improved solution regularity.
          In \cite{shi2024continuum}, a formulation of the discrete $p$-biharmonic operators and their continuum limits were studied.
         Meanwhile, in \cite{Calder2020Properly}, it was shown that introducing a proper reweighting leads to the recovery of a well-posed weighted Laplacian,
     while maintaining the approximation of the labeled data set. In this case,
     the variational problem 
     is set
     to minimize over all $u : \cX \to \bbR$
		\begin{equation*}
		  \sum_{\bx,\by \in \cX } \dist(\bx,\Gamma)^{-\beta} \rho_{\bx,\by} |u(\bx)-u(\by)|^2 \text{ subject to } u(\bx) = g(\bx) \text{ on } \Gamma,
		\end{equation*}
        where $\bx$, $\by$ are points in the data set $\cX$ -- which is a graph with edge weights $\rho_{\bx,\by}$ -- and where $\Gamma \subset \cX$ is the set of labeled data points with label function $g : \Gamma \to \bbR$. 
		Among their other main results, the authors of \cite{Calder2020Properly} show that the minimization problem recovers the energy 
        \begin{equation*}
            \int_{\Omega} \dist(\bx,\Gamma)^{\beta} |\grad u(\bx)|^2 \, \rmd \bx
        \end{equation*}
        in an appropriate
        large graph limit, i.e., the discrete-to-continuum limit.
		This weighted graph Laplacian-based learning problem can thus be viewed as a harmonic extension problem for an elliptic operator with singular/degenerate coefficients, with ``boundary conditions'' specified by the Dirichlet conditions on the labeled set $\Gamma$. In \cite{Calder2020Properly}, it is shown that the well-posedness of this boundary value problem relies on choosing $\beta \in \bbR$ in a suitable range of exponents, which in turn is characterized by the ambient dimension of the problem.

        Another approach to address the challenge of sparse labeled data in the semi-supervised learning problem is to replace the graph Laplacian with the graph $p$-Laplacian, for some $p > 2$.
        In the aforementioned large graph limit, minimizers of the $p$-Dirichlet energy $\int_\Omega |\grad u(\bx)|^p \, \rmd \bx$ 
        enjoy greater regularity thanks to the Sobolev embedding theorem.
        This feature has been used to obtain consistent semi-supervised learning problems in the large graph limit via convergence of minimizers, and additionally to obtain desirable convergence rates of these minimizers; see for instance \cite{Flores2022Analysis,Calder2018game,Slepcev2019Analysis}.
		
		A key strategy in the discrete-to-continuum convergence results of \cite{Calder2020Properly,Flores2022Analysis} is the use of a nonlocal Laplacian as an intermediary, i.e. the discrete graph-based functional is cast in a continuum setting via the functional
		\begin{equation*}
			\iint \dist(\bx,\Gamma)^{-\beta} \bar{\rho}(\bx,\by) |u(\bx)-u(\by)|^2 \, \rmd \by \, \rmd \bx.
		\end{equation*}
		Precisely, by studying optimal transport between measures and their induced metrics, the discrete Laplacian for general graphs can be placed into a nonlocal continuum formulation; for details of this strategy see \cite{Calder2020Properly,Trillos2015rate,GarciaTrillos2016Continuum}.

    The main objective of this paper is to study general nonlocal functionals and the associated nonlocal operators for semi-supervised learning, with the aim of  developing a rigorous mathematical theory that can offer practical guidance to more effective parameter tuning in applications.
    To this end, we seek to provide a unified theory of general weighted problems for general nonlocal $p$-Laplacian operators. 
    It is our hope that the analysis developed in this work can be readily applied to the discrete graph-based learning problems via well-understood tools; for instance, the optimal transport theory \cite{Trillos2015rate}, the interaction between the  nonlocal horizon and dense graph length scales \cite{GarciaTrillos2016Continuum}, and the regularity theory for solutions to graph Laplacian equations \cite{Calder2022Lipschitz}. 
    
    At the same time, we consider a class of nonlocal kernel $\bar{\rho}$ which, 
    to the extent of the authors' knowledge, has not been considered in the literature on semi-supervised learning. A parameter in $\bar{\rho}$ controls the maximum extent of the edge weights between points, which is often chosen to be a fixed constant. The models considered here allow the length scale to be position-dependent, i.e. vary depending on the grid point.
    Such a spatial-dependent length scale has been discussed in the literature, see for example, \cite{berry2016variable}. Innovative in the context of our work is the use of heterogeneous localization \cite{Tian2017Trace,Scott2023Nonlocal,Scott2023Nonlocala}. This latter choice in the model parameter allows us to consider problems posed on a bounded domain
    and gives further insight into learning problems with boundary, e.g. 
    the learning of a manifold with boundary \cite{GarciaTrillos2020Error}.

   In the first part of the paper we treat the labeled set $\Gamma$ as a finite set of points. However, one may anticipate cases where some of the labeled data may be clustered and can form a set of higher dimension, embedded within the graph.
    This scenario is to depict cases that the data set, while sparse, could be concentrated in the form of disjoint clusters.
   Thus, $\Gamma$ can be taken to be of general dimension, say having a dimension $\ell \in \{1,\ldots,d-1\}$. We treat this more general case in the latter part of the paper.
  
		\subsection{The variational problem}\label{subsec:Intro:VarProb}
		
		Let $\Omega^* \subset \bbR^d$ be a bounded Lipschitz domain, and let $\Gamma \subset \overline{\Omega}^*$ be a given finite set of points.
		We define $\Omega := \Omega^* \setminus \Gamma$. Thus $\p \Omega = \p \Omega^* \cup (\Gamma \cap \Omega^*)$ and $\overline{\Omega} = \overline{\Omega}^*$.
        The set $\Omega$ can be interpreted as the unlabeled data set in the following sense:
        Let $\mu$ be the uniform probability distribution on $\Omega$.
        Given a set $X_n$ of $n$ randomly sampled unlabeled data points, i.e. independent and identically distributed random variables $X_n = \{ \bx_1,\ldots,\bx_n\} \subset \Omega$ each with probability distribution $\mu$, there exists a \textit{transportation map} from $\mu$ to the empirical measure $\mu_n$ on the set $\cX_n := X_n \cup \Gamma$.
        That is, there exists a Borel measurable function $T_n : \Omega \to \Omega$ such that $\mu(T_n^{-1}(U)) = \mu_n(U)$ for all open sets $U \subset \overline{\Omega}$.
        Moreover, it was shown in \cite{Trillos2015rate}
        that $T_n$ is comparable in $L^\infty$ norm to the identity map on $\Omega$, with a constant of comparison vanishing as $n \to \infty$; see \Cref{thm:TransportMapConvergence} below and also \cite[Theorem A.3]{Calder2020Properly} for the precise rate of convergence.
        The discrete graph functions/functionals can therefore be treated as nonlocal continuum functions/functionals using this transportation map. Specifically, to compare functions $u_n : X_n \to \bbR$ that are defined on the discrete unlabeled data set in a consistent way with continuum functions $u : \Omega \to \bbR$ in the $L^p$-topology, 
        one can compare $\wt{u}_n := u_n \circ T_n : \Omega \to \bbR^d$ with $u$.
       In this context, as illustrated in
 \cite{Calder2020Properly,GarciaTrillos2016Continuum}, the discrete-to-continuum variational convergence is characterized by taking the number of data points $n \to \infty$. The asymptotic properties of the transportation maps $T_n$ as $n \to \infty$ are studied in \cite{Calder2020Properly,GarciaTrillos2016Continuum} and the variational convergence results therein make use of the $TL^p$ metric, which is defined in terms of transportation maps. 
The edge weights of the functionals are described by a radial function and a length scale, or \textit{horizon}, $\delta > 0$. The horizon is chosen to depend on the size of the sampled data, so that $\delta \to 0$ as $n \to \infty$. The precise scaling laws between $\delta$ and $n$ determine the variational convergence properties of the discrete-to-continuum limit.

        In our setting, we take advantage of the heterogeneous localization to formulate a well-defined nonlocal variational problem without being confined to functions that result from the embedding provided by the transportation map.
        This allows us to retain the horizon $\delta$ and consider the scaling regime $\delta \to 0$, without reference to the size of the data set $n$. We can regard all functions as maps defined on $\Omega$, not merely on discrete subsets. Our topologies for variational convergence are then described by nonlocal function spaces. That is, we can
        focus the work on the
        nonlocal-to-local continuum regime. Meanwhile, the discrete-to-continuum regime can be discussed in the nonlocal continuum setting without resorting to the local continuum limit.
        In this sense, the nonlocal formulation proposed and analyzed here provides a bridge linking the discrete problem and the local continuum formulation and offers an alternative path to model large point clouds.  Moreover, we hope the nonlocal-to-local convergence results of this work will also better inform the choice of parameters in solving practical semi-supervised learning problems.
        
        The reference probability measure $\mu$ used in this work is the rescaled Lebesgue measure. However, the results of this work remain true if Lebesgue measure is replaced with any probability measure $\wt{\mu}$
        that is absolutely continuous with respect to Lebesgue measure, with its density function bounded from above and away from zero. This is the more general setting considered in \cite{Calder2020Properly}, and we consider only Lebesgue measure in this article for simplicity.
        
The continuum variational problem formulated for semi-supervised learning is as follows: Given exponents $p \in (1,\infty)$ and $\beta \in \bbR$, a finite set of points $\Gamma \subset \overline{\Omega^*}$ corresponding to the labeled data set, and a label function $g : \Gamma \to \bbR$, we search for an extension of the given labels $g$ to the entire data set $\Omega$, i.e. we search for a label function $u:\Omega \to \bbR$ that solves the problem:
	\begin{equation}\label{eq:Intro:MinProb}
			\begin{gathered}
		 \text{ Minimize } 
				\cE_{\delta}(u) := \frac{1}{p} \int_{\Omega} \int_{\Omega} \frac{1}{\gamma(\bx)^\beta} \rho \left( \frac{|\by-\bx|}{\delta \eta(\bx)} \right) \frac{|u(\by)-u(\bx)|^p}{(\delta \eta(\bx))^{d+p}} \, \rmd \by \, \rmd \bx, \\
       \text{ over } u\in \mathfrak{W}^{p}[\delta](\Omega;\beta)
				\text{ subject to } u(\bx) = g(\bx) \text{ for } \bx \in \Gamma.
			\end{gathered}
		\end{equation}

Here, we adopt notation introduced in \cite{Scott2023Nonlocal}, namely,
the constant scalar parameter $\delta>0$
controls the maximum range of interactions, the function
 $\eta(\bx) = \dist(\bx,\p \Omega)$ denotes the distance function, and the nonlocal kernel $\rho$ satisfies \eqref{assump:VarProb:Kernel} in \Cref{sec:assump-ker} below.
We introduce the function $\gamma(\bx) \approx \dist(\bx,\Gamma)$, defined precisely in \eqref{assump:weight} below, that controls the singularity/degeneracy at $\Gamma$;
the exponent $\beta \in \bbR$ controls the size of this singularity/degeneracy.

  To define the nonlocal function space $\mathfrak{W}^{p}[\delta](\Omega;\beta)$ associated with \eqref{eq:Intro:MinProb}, we first
  define the weighted Lebesgue space
$L^p(\Omega;\beta)$ to be the class of all Lebesgue-measurable functions with $\vnorm{u}_{L^p(\Omega;\beta)} <\infty$, where
\begin{equation*}
    \Vnorm{u}_{L^p(\Omega;\beta)}^p
		:= \int_{\Omega} \frac{|u(\bx)|^p}{ (\gamma(\bx))^\beta } \, \rmd \bx .
\end{equation*}
We then introduce a nonlocal seminorm: 
    \begin{equation}\label{eq:Intro:NonlocalSeminorm}		
        [u]_{\mathfrak{W}^{p}[\delta](\Omega;\beta)}^p :=   \frac{ \overline{C}_{d,p} (d+p) }{ \sigma(\bbS^{d-1}) } \int_{\Omega} \int_{B(\bx,\delta d_{\p \Omega}(\bx))} \frac{1}{\gamma(\bx)^\beta} \frac{|u(\by)-u(\bx)|^p}{ (\delta \eta(\bx))^{d+p} } \, \rmd \by \, \rmd \bx,
    \end{equation}
		where, in the normalizing constant, $\sigma$ denotes the surface measure and $\bbS^{d-1} \subset \bbR^d$ is the unit sphere.
		These constants are defined so that the nonlocal seminorm is consistent with a weighted Sobolev seminorm in a precise way, as will be discussed later.
The corresponding nonlocal weighted function space is then given by
		\begin{equation*}
			\mathfrak{W}^{p}[\delta](\Omega;\beta) :=
			\{ u \in L^p(\Omega;\beta) \, :\, [u]_{ \mathfrak{W}^{p}[\delta](\Omega;\beta) } < \infty \}, 
            \qquad \text{ for } \beta < d,
		\end{equation*}
which is a reflexive Banach space with the norm
    \begin{equation*}
        \Vnorm{u}_{\mathfrak{W}^{p}[\delta](\Omega;\beta)}^p := \Vnorm{u}_{L^p(\Omega;\beta)}^p + [u]_{\mathfrak{W}^{p}[\delta](\Omega;\beta)}^p.
    \end{equation*}

The behavior of a function $u$ with the above norm finite will compensate the weight accordingly;  in fact, we will show that for kernel and weight functions
satisfying desirable properties (see \Cref{sec:assump-wei} and \Cref{sec:assump-ker})
and for a range of the parameter $\beta$, the values of $u$ on $\Gamma$ can be prescribed.

The reason for us to consider the heterogeneous localization $\eta(\bx) = \dist(\bx,\p \Omega)$, despite the fact that $\p \Omega \setminus \Gamma$ is an unlabeled set, is two-fold. First, the formulation provides a well-posed theory for the nonlocal problem that is also consistent with a local classical boundary-value problem. 
    Incorporating nonlocal versions of a flux condition on $\p \Omega \setminus \Gamma$, or on the complement of such a set, would introduce additional technical complications that would nevertheless resolve in the local limit. 
        For instance, we could repeat the arguments of this work for a problem posed on a torus (i.e. with periodic boundary conditions), with $\eta(\bx) \approx \dist(\bx,\Gamma)$, with no appreciable technical differences. Second, we hope to demonstrate that heterogeneous localization can be used to analyze the discrete semi-supervised learning problem on manifolds with boundary.

We next present some assumptions to clarify the kernel and weight functions used in the above problem formulations. Although these assumptions are analogous to those discussed in \cite{Scott2023Nonlocal, Scott2023Nonlocala}, a notable difference is the introduction of $\Gamma$-dependent weights for the present study.
\subsection{Assumptions on the weight function}\label{sec:assump-wei}
First, in the case of a discrete labeled data set $\Gamma$, we assume that there exists $R > 0$ depending only on $\Gamma$ such that $B(\bx_0,4R) \cap \Gamma = \{\bx_0\}$ for all $\bx_0 \in \Gamma$, where $B(\bx_0,R)$ denotes the Euclidean ball centered at $\bx_0 \in \mathbb{R}^d$ of radius $R$.
		We then define the weight function $\gamma(\bx)$ as a function in $C^{\infty}(\overline{\Omega} \setminus \Gamma) \cap C^0(\overline{\Omega})$ such that
\begin{equation}\label{assump:weight}
		\begin{gathered}
			\gamma(\bx) = \dist(\bx,\Gamma) = |\bx-\bx_0| \text{ whenever } \exists \bx_0 \in \Gamma \text{ such that } \bx \in B(\bx_0,R), \\
				\gamma(\bx) \equiv 1 \; \forall \bx \in \Omega \setminus \left( \cup_{\bx_0 \in \Gamma} B(\bx_0,2R) \right), \\
                \text{ and for each multi-index } \alpha \in \bbN^d_0, \\
				\exists \kappa_\alpha > 0 \text{ such that } |D^\alpha \gamma(\bx)| \leq \kappa_\alpha |\dist(\bx,\Gamma)|^{1-|\alpha|}, \: \forall \bx \in \Omega.
\end{gathered}\tag{\ensuremath{\rmA_{\gamma}}}
		\end{equation}
		Such a function can be constructed via 
        mollification and cutoff functions.
    Our assumptions on $\gamma$ are similar to those on the weight function used in the discrete graph Laplacian in \cite{Calder2020Properly}, with some distinctions. First, the continuum energy functionals considered in the majority of this work are finite for smooth functions, with no truncation of the weights required. We treat nonlocal functionals with truncated weights only during comparison with discrete functionals in \Cref{sec:Graph}. Second, our analysis does not make use of the exact transition strategy adopted on the sets $B(\bx_0,2R) \setminus B(\bx_0,R)$ for $\bx_0 \in \Gamma$.

\subsection{Assumptions on the scaling parameter}\label{sec:assump-lam}

We refer to the function $\eta(\bx) = \dist(\bx,\p \Omega)$ as a \textit{heterogeneous localization} function.
We define the rescaled function
		$\eta_\delta(\bx)$ which
		is given by	\begin{equation}\label{eq:localizationfunction}
			\eta_\delta(\bx) := \delta \eta(\bx), \quad\forall \bx \in \overline{\Omega}.
		\end{equation}
The scaling parameter $\delta > 0$ measures the maximum range of nonlocal interactions.
For the study of the variational problems, the maximum admissible value of the bulk horizon parameter $\delta$ is chosen to depend on $\lambda(\bx)$ as follows:
		\begin{equation}\label{assump:Horizon}
			\begin{gathered}
				\delta \in (0,\underline{\delta}_0)
                \;\text{ where }\;
				\underline{\delta}_0 := \frac{1}{3 \kappa } 
                \; \text{ and } \;
                \kappa := \kappa_0^2 \kappa_1.
				\tag{\ensuremath{\rmA_{\delta}}}
			\end{gathered}
		\end{equation}
        We note that 
        thanks to \eqref{assump:weight}, 
        we have for all $\delta < \underline{\delta}_0$,
        \begin{equation*}
            |\gamma(\bx) - \gamma(\by)| < \frac{1}{3} \gamma(\bx),\; \text{ if }\; |\bx-\by| < \eta_\delta(\bx).
        \end{equation*}
		This guarantees that desired coordinate changes in the integrals defining the nonlocal seminorm can be carried out.

\subsection{Assumptions on the nonlocal kernel}\label{sec:assump-ker}

Following the discussions in \cite{Scott2023Nonlocal,Scott2023Nonlocala},
in the nonlocal functional specified in \eqref{eq:Intro:MinProb},
the nonlocal kernel $\rho : \bbR \to [0,\infty)$ is assumed to satisfy
		\begin{equation}\label{assump:VarProb:Kernel}
			\begin{gathered}
				\rho \in L^{\infty}(\bbR), \; \rho(x)\geq 0, \;
                \text{ and } 
				[-c_\rho,c_{\rho}] \subset \supp \rho \Subset (-1,1) \text{ for fixed } c_{\rho} > 0.\\
				\;\text{ Moreover, } \rho(x) \text{ is nonincreasing } \forall x \geq 0, \;  \rho(-x) = \rho(x) \; \forall x\in  \bbR, \text{ and } \\
                \int_{B(0,1)} |\bz|^{p} \rho(|\bz|) \, \rmd \bz =
                \frac{ \sqrt{\pi} \, \Gamma ( \frac{d+p}{2} ) }{ \Gamma(\frac{p+1}{2} ) \Gamma(\frac{d}{2}) } := \overline{C}_{d,p},
			\end{gathered}
			\tag{\ensuremath{\rmA_{\rho}}}
		\end{equation}
  		with $\Gamma(z)$ denoting the Euler gamma function and $B(0,1)$ denoting the unit ball centered at the origin.
		Note that $\int_{\bbS^{d-1}} |\bsomega \cdot \be|^p \, \rmd \sigma(\bsomega)=
    \frac{\overline{C}_{d,p}}{\sigma(\bbS^{d-1})} 
        $, where $\be$ is any fixed unit vector.

        \subsection{Main results}\label{subsec:MainResults}

        The well-posedness of \eqref{eq:Intro:MinProb} is contingent on functions in the energy space $\mathfrak{W}^p[\delta](\Omega;\beta)$ having well-defined values on $\Gamma$. 
        Specifically, functions in $\mathfrak{W}^p[\delta](\Omega;\beta)$ have a well-defined trace on $\Gamma$ if and only if $d - p < \beta$. Other structural properties of the function space for this range of $\beta$ will allow us to conclude the following theorem, proved as part of \Cref{thm:equi:wellposed} in \Cref{sec:VarProb}.
        \begin{theorem}\label{thm:Intro:WellPosed}
            Assume $d \geq 1$, $p \in (1,\infty)$, and $d-p<\beta < d$. Then there exists a unique solution $u \in \mathfrak{W}^p[\delta](\Omega;\beta)$ to \eqref{eq:Intro:MinProb}.
        \end{theorem}
The well-defined problem \eqref{eq:Intro:MinProb} provides a  nonlocal continuum formulation of the discrete learning problem in the infinite data limit. One may draw further connection to the local limit by letting $\delta\to 0$. In such a limit, we show the convergence of solutions to \eqref{eq:Intro:MinProb} to the following local variational problem:
\begin{equation}\label{eq:Intro:MinProb:Limit}
            \begin{gathered}
                \text{Minimize } \cE_0(u) = \frac{1}{p} \int_{\Omega} \frac{|\grad u(\bx)|^p}{\gamma(\bx)^\beta}
                \text{over } u \in W^{1,p}(\Omega;\beta), \\
                \text{ subject to } u(\bx) = g(\bx) \text{ for } \bx \in \Gamma.
            \end{gathered}
		\end{equation}
  Here, the local weighted Sobolev space is
        \begin{equation*}
			W^{1,p}(\Omega;\beta) := \{ u \in L^p(\Omega;\beta) \, : \, \vnorm{\grad u}_{L^p(\Omega;\beta)} < \infty \}, \qquad  \beta < d, \quad 1 < p < \infty,
		\end{equation*}
		with norm
		\begin{equation*}
		  \vnorm{u}_{W^{1,p}(\Omega;\beta)}^p := \vnorm{u}_{L^p(\Omega;\beta)}^p + \vnorm{\grad u}_{L^p(\Omega;\beta)}^p.
		\end{equation*}
        
     The local energy $ \cE_0(u)$ can be derived formally from $\cE_\delta(u)$ via a Taylor expansion of $u$. 
        The analysis on the well-posedness of the minimization of $\cE_0$ was carried out in \cite{Calder2020Properly} for the case $p = 2$, and with the local weighted Sobolev space defined in a slightly different way. 
        In the next theorem, proved as part of \Cref{thm:LocLimit:Dirichlet}, we extend their analysis to the case of general $p$, and use spaces like $W^{1,p}(\Omega;\beta)$ whose weights mimic those in $\mathfrak{W}^p[\delta](\Omega;\beta)$.

        \begin{theorem}\label{thm:Intro:LocLimit:Dirichlet}
            Assume $d \geq 1$, $p \in (1,\infty)$, and $d-p<\beta < d$. Then, as $\delta\to 0$, the sequence of unique solutions $\{u_\delta\}_\delta \subset \mathfrak{W}^p[\delta](\Omega;\beta)$ to \eqref{eq:Intro:MinProb} converge strongly in $L^p(\Omega;\beta)$ to the unique solution $u \in W^{1,p}(\Omega;\beta)$ of \eqref{eq:Intro:MinProb:Limit}.
        \end{theorem}

        \begin{remark}
            In the case of highly singular weight corresponding to $\beta > d$, it turns out that functions in a weighted nonlocal space can only have trace zero, as shown in \Cref{subsec:NonlocalSpaces} later. Therefore, this range of parameters results in a function space too inflexible to use in the semi-supervised learning problem. Additionally, the analysis of the minimization problems is trivialized; the unique solution to \eqref{eq:Intro:MinProb} with $g = 0$ is $u(\bx) \equiv 0$, and so all the results of \Cref{subsec:MainResults} clearly hold.
                \end{remark}

    \begin{remark}
    We note that the framework developed here offers various extensions that could be of interest in applications. For example, the case of a labeled set with connected components of distinct dimensions has a similar program of analysis; this can be found in \Cref{sec:GeneralCodim}.
    \end{remark}        

    With the necessary mathematical framework in place, we demonstrate that the nonlocal energy can be recovered from an appropriately-defined graph energy in the infinite data limit $n \to \infty$. Recalling the definitions of the data set $\cX_n$ and transportation maps $T_n$ above, we consider the following discrete functional with truncated parameters
    \begin{equation*}
    \cE_{n,\delta,\tau}(u) := \frac{1}{n^2} \sum_{\bx,\by \in \cX_n} \rho \left( \frac{|\by-\bx|}{\eta_\delta^\tau(\bx)} \right) \frac{|u(\bx)-u(\by)|^p}{(\gamma^\tau(\bx))^{\max\{\beta,0\}} (\gamma(\bx))^{\min\{\beta,0\}} \eta_\delta^\tau(\bx)^{d+p}},
    \end{equation*}
    where for $\tau > 0$ the functions $\eta_\delta^\tau$ and $\gamma^\tau$ are defined as
    \begin{equation}\label{eq:TruncatedParameters}
    \begin{gathered}
        \eta_\delta^\tau(\bx) := \delta \max \{ \eta(\bx), \tau \}, \qquad \qquad 
        \gamma^\tau(\bx) := \max \{ \gamma(\bx), \tau \}.
    \end{gathered}
    \end{equation}
    Since $\eta(\bx) = \gamma(\bx) = 0$ on $\Gamma$, using the full weight parameters in the discrete energy would require some additional smoothness of $u$ on $\Gamma$, so the weights are replaced by the truncations in order to ensure that $\cE_{n,\delta,\tau}(u)$ is well-defined. The truncation $\tau$ will be chosen to depend on $n$, and will satisfy $\tau_n \to 0$ in the continuum limit $n \to \infty$. 
    {The precise scaling law is described in the following theorem, in which the nonlocal continuum energy is obtained in the $n \to \infty$ limit of the discrete energies $\cE_{n,\delta,\tau_n}(u \circ T_n^{-1})$ for $u$ in the nonlocal function space.}

\begin{theorem}\label{thm:GraphToNonlocalConv}
    Let $d-p < \beta < d$ and $\beta \geq 0$. Let $\delta \in (0,\underline{\delta}_0)$, and for the sequence of transportation maps $T_n$ defined above that satisfy \Cref{thm:TransportMapConvergence} set $\zeta_n := 2 \vnorm{T_n-Id}_{L^\infty(\Omega)}$. 
    Define a sequence $\tau_n$ such that $\tau_n \to 0$ as $n \to \infty$ and such that the sequence $c_n :=  \frac{\zeta_n}{\tau_n}$ satisfies $\lim\limits_{n \to \infty} c_n = 0$.
    Then there exists $n_0$ depending only on $\underline{\delta}_0$ and $\Gamma$ such that for all $n \geq n_0$
    \begin{equation}
    \label{eq:Discrete:Finite}
        \cE_{n,\delta,\tau_n}(u \circ T_n^{-1}) \leq \frac {C(d,p,\beta,\rho,\kappa,\Omega)}{ \delta^p }
        \vnorm{u}_{\mathfrak{W}^p[\delta](\Omega;\beta)}^p,
        \qquad \forall u \in \mathfrak{W}^p[\delta](\Omega;\beta),
    \end{equation}
    and moreover 
    \begin{equation}\label{eq:Discrete:conv}
    \lim\limits_{n \to \infty} \cE_{n,\delta,\tau_n}(u \circ T_n^{-1}) = |\Omega|^{-2} \cE_{\delta}(u), \qquad \quad \forall\, u \in \mathfrak{W}^p[\delta](\Omega;\beta).
    \end{equation}
\end{theorem}

\begin{remark}
Although this theorem can be shown for arbitrary $\beta \in \bbR$, for the sake of illustration we prove it only for nonnegative $\beta$ 
that are also in $(d-p, d)$, i.e. those
for which \Cref{thm:Intro:WellPosed} holds.
\end{remark}

        \subsection{Organization of the paper}
		
        The paper is organized as follows: \Cref{sec:WeightedLocalSpaces} contains the necessary background on local weighted Sobolev spaces. In \Cref{sec:WeightedNonlocalSpaces} we establish analogous results for the nonlocal weighted Sobolev spaces. The well-posedness and the variational convergence for the problem \eqref{eq:Intro:MinProb} are proved in \Cref{sec:VarProb} and \Cref{sec:Conv}, respectively.
        We prove generalizations of all of these results for the case of higher-dimensional sets $\Gamma$ in \Cref{sec:GeneralCodim}. Finally, we connect the nonlocal energies considered here to weighted graph $p$-Laplacian-type energies via a discrete-to-nonlocal convergence result in \Cref{sec:Graph}.
		
		\section{Local weighted Sobolev spaces}\label{sec:WeightedLocalSpaces}
		
		We recall and establish some results on weighted Sobolev spaces which are viewed as local spaces, as their characterizations are based on norms of weak derivatives. These spaces have been studied extensively; see for instance the monographs \cite{Kozlov2001Spectral,Kufner1980Weighted}.
		Most of the results we need for these spaces are already known; we state the theorems in a way that provides a complete story for our purposes. We also include proofs whenever the existing literature has not covered the specific situations we are concerned with.
		Most of these proofs rely on the following inequalities adapted from the more general Hardy's inequality (see e.g. \cite[Theorem 330]{Hardy1952Inequalities}):
		Let $d \geq 1$, $1 < p < \infty$, and $\beta \in \bbR$.
		For a function $v :[0,\infty) \to [0,\infty)$ that is absolutely continuous and compactly supported on $[0,\infty)$ with $r^{(d-1-\beta)/p} v'(r) \in L^p((0,\infty))$, 
		\begin{align}
			\label{eq:Hardy1D:inf}
			\int_0^\infty \frac{|v(r)|^p}{r^{\beta-d+1+p}} \, \rmd r &\leq \left( \frac{p}{ d-p-\beta } \right)^p
			\int_0^\infty \frac{|v'(r)|^p}{r^{\beta-d+1}} \, \rmd r, \quad \text{ for } \beta \in (- \infty,d-p), \\
			\label{eq:Hardy1D:0}
			\int_0^\infty \frac{|v(r)-v(0)|^p}{r^{\beta-d+1+p}} \, \rmd r &\leq \left( \frac{p}{ \beta+p-d } \right)^p
			\int_0^\infty \frac{|v'(r)|^p}{r^{\beta-d+1}} \, \rmd r, \quad \text{ for } \beta \in (d-p,\infty).
		\end{align}
		
	\subsection{Local weighted spaces}
		  
        In this subsection we show density, trace and extension results for the reflexive Banach space $W^{1,p}(\Omega;\beta)$. 
		
		\begin{theorem}\label{thm:Density}
			For all $p \in (1,\infty)$ and all $\beta \in (-\infty,d)$, the class $C^{\infty}(\overline{\Omega})$ is dense in $W^{1,p}(\Omega;\beta)$. 
		\end{theorem}
		
		\begin{proof}
			It suffices to prove the result for the case $\Gamma = \{\bx_0\}$ and $\gamma(\bx) = |\bx-\bx_0|$, since then the general case will follow by a localization and partition of unity argument.
			
			First, if $d(1-p) < \beta < d$, then $|\bx-\bx_0|^{-\beta}$ is a Muckenhoupt $A_p$-weight, see \Cref{lma:ApWeight}.
   Thus, as a special case of a similar density result for $A_p$-weighted Sobolev spaces for functions defined on Jones domains (see for instance \cite[Theorem 6.1]{Chua1992Extension}), there exists $\{ u_n \} \in C^{\infty}(\bbR^d)$, hence in $C^{\infty}(\overline{\Omega})$, such that $\vnorm{ u_n - u }_{W^{1,p}(\Omega;\beta)} \to 0$ as $n \to \infty$.
			
			Next, if $\beta \leq d(1-p)$, then $\beta < 0$, and so the weight is degenerate instead of singular. In the event that $\bx_0 \in \p \Omega^*$, the density result in this setting is proved in \cite[Theorem 7.4]{Kufner1980Weighted}.
			
			Now assume that $\beta \leq d(1-p)$ and that $\bx_0 \in \Omega^*$. By a localization argument (as well as a translation and dilation) we can assume that $\Omega^* = B(0,1)$, that $\bx_ 0 = 0$, and that $\supp u \Subset B(0,1)$.
			Let $\zeta \in C^\infty(\bbR^d)$ be a function satisfying $0 \leq \zeta \leq 1$, $\zeta \equiv 1$ on $B(0,1)$ and $\zeta \equiv 0$ on $\bbR^d \setminus B(0,2)$. Then define $u_k(\bx) := u(\bx)(1-\zeta(k\bx))$. Clearly $\vnorm{u_k-u}_{L^p(\Omega;\beta)} \to 0$ as $k \to \infty$ by continuity of the Lebesgue integral. Now,
			\begin{equation*}
				\begin{split}
					\int_{B(0,1)} \frac{|\grad u_k - \grad u|^p}{|\bx|^\beta} \, \rmd \bx &\leq C \int_{B(0,2/k)} \frac{|\grad u|^p}{|\bx|^\beta} \, \rmd \bx + C k^p \int_{B(0,2/k) \setminus B(0,1/k)} \frac{|\grad \zeta(k\bx) u(\bx)|^p}{|\bx|^\beta} \, \rmd \bx \\
					&\leq C \vnorm{ \grad u }_{L^p(B(0,2/k);\beta)}^p + C \int_{B(0,2/k) \setminus B(0,1/k)} \frac{|\grad \zeta(k\bx) u(\bx)|^p}{|\bx|^{\beta+p}} \, \rmd \bx
				\end{split}
			\end{equation*}
			Thanks to the support of $\zeta$, the function $\grad \zeta(k\bx) u(\bx)$ belongs to $W^{1,p}(B(0,1))$ and is compactly supported on $B(0,1) \setminus \{0\}$, hence the function $v(r) := \grad \zeta(kr \bsomega) u(r\bsomega)$ is absolutely continuous on $[0,1]$ for $\scH^{d-1}$-almost every $\bsomega \in \bbS^{d-1}$. Hence we can apply the Hardy inequality \eqref{eq:Hardy1D:inf} to get
			\begin{equation*}
				\int_{1/k}^{2/k} \frac{|v(r)|^p}{ r^{\beta-d+1+p} } \, \rmd r \leq C(\beta,d,p) \int_{1/k}^{2/k} \frac{|v'(r)|^p}{ r^{\beta-d+1} } \, \rmd r.
			\end{equation*}
			(This can be applied, since by assumption $\beta \leq d(1-p) < d - p$).
			Integrating over $\bsomega \in \bbS^{d-1}$, we obtain that
			\begin{equation*}
				\int_{B(0,2/k) \setminus B(0,1/k)} \frac{|\grad \zeta(k\bx) u(\bx)|^p}{|\bx|^{\beta+p}} \, \rmd \bx \leq C \int_{B(0,2/k)} \frac{|\grad u|^p}{|\bx|^\beta} \, \rmd \bx,
			\end{equation*}
            where we also used that $|\p_r v| \leq |\grad_{\bx} u|$.
			Therefore by continuity of the integral we get that $\vnorm{u_k - u}_{W^{1,p}(\Omega;\beta)} \to 0$ as $k \to \infty$.
			
			Now, for $m \in \bbN$ with $m \gg k$ define $u_{k,m}(\bx) := \int_{\bbR^d} m^d \varphi(m(\bx-\by)) u_k(\by) \, \rmd \by$, where $\varphi$ is a standard mollifier. Then $u_{k,m} \in C^\infty(\overline{B(0,1)})$,  $\supp u_{k,m} \Subset B(0,1) \setminus \{0\}$, and since $\beta < 0$
			\begin{equation*}
			\vnorm{u_{k,m}-u_k}_{L^p(\Omega;\beta)} \leq \vnorm{u_{k,m}-u_k}_{L^p(\Omega)} \to 0 \text{ as } m \to \infty.
			\end{equation*}
			Thus, for any $\veps > 0$ we can choose $k$ large so that $
			\vnorm{u_k-u}_{L^p(\Omega;\beta)} < \veps/2$, and then choose $m \ll k$ so that $\vnorm{u_{k,m} - u_k}_{L^p(\Omega;\beta)} < \veps/2$. 
		\end{proof}
		
		Now, for a certain range of $\beta$, one can show the existence of traces on $\Gamma$ in this space, see discussions in \cite{Calder2020Properly} for the case $p = 2$.
  We state the relevant results in the theorem below, together with a more complete proof than that given in the literature. In particular, we consider the cases when $\beta\geq 0$ and when $\beta<0$, in the event that $p > d$.
		\begin{theorem}\label{thm:Trace}
			Let $d - p < \beta < d$. The operator $T_\Gamma : W^{1,p}(\Omega;\beta) \to \bbR^{|\Gamma|}$ is a bounded linear operator that satisfies $T_\Gamma u = u \big|_\Gamma$ for all $u$ in $C^\infty(\overline{\Omega})$.
			Moreover, the limit
			\begin{equation*}
				\begin{gathered}
				u(\bx_0) := 
				\begin{cases}
					\lim\limits_{\veps \to 0} (u)_{B(\bx_0,\veps) \cap \Omega}, &\quad \beta \geq 0, \\
					\lim\limits_{\veps \to 0} (u)_{B(\bx_0,\veps) \cap \Omega,\beta}, &\quad \beta < 0,
				\end{cases}\\
				\text{where }  (u)_{B(\bx_0,\veps) \cap \Omega,\beta} := \frac{ \displaystyle\int_{B(\bx_0,\veps) \cap \Omega} u(\bx) \gamma(\bx)^{-\beta} \, \rmd \bx }{ \displaystyle\int_{B(\bx_0,\veps) \cap \Omega}  \gamma(\bx)^{-\beta} \, \rmd \bx}, \quad 0 < \veps < R,
				\end{gathered}
			\end{equation*}
			is well-defined for any $\bx_0 \in \Gamma$ and $u \in W^{1,p}(\Omega;\beta)$, with
			\begin{equation*}
				\begin{split}
					|u(\bx_0) - (u)_{B(\bx_0,\veps) \cap \Omega }|^p &\leq C \veps^{\beta-(d-p)} \vnorm{\grad u}_{L^{p}(\Omega;\beta)}^p, \qquad \beta \geq 0, \: 0 < \veps < R \\
					|u(\bx_0) - (u)_{B(\bx_0,\veps) \cap \Omega,\beta}|^p &\leq C \veps^{\beta-(d-p)} \vnorm{\grad u}_{L^{p}(\Omega;\beta)}^p, \qquad \beta < 0, \: 0 < \veps < R.
				\end{split}
			\end{equation*}
		\end{theorem}
		
		\begin{proof}
			We proceed in a manner inspired by \cite{Calder2020Properly}, with additional details presented.
			Fix $\bx_0 \in \Gamma$. In this proof, we set the notation $B_r = B(r) = B(\bx_0,r) \cap \Omega$ for any $ r > 0$. If $u \in W^{1,p}(\Omega;\beta)$, then because $d-p < \beta < d$ the following Poincar\'e inequalities with $A_p$-weights hold, i.e.
			\begin{equation}\label{eq:TraceLeb:Poincare}
				\begin{split}
					\int_{B_\veps} |u(\bx) - (u)_{B_\veps}|^p \gamma(\bx)^{-\beta} \, \rmd \bx &\leq C \veps^p \int_{B_\veps} |\grad u(\bx)|^p \, \gamma(\bx)^{-\beta} \rmd \bx, \qquad \forall \veps \in (0,R), \text{ and } \\
					\int_{B_\veps} |u(\bx) - (u)_{B_\veps,\beta}|^p \gamma(\bx)^{-\beta} \, \rmd \bx &\leq C \veps^p \int_{B_\veps} |\grad u(\bx)|^p \, \gamma(\bx)^{-\beta} \rmd \bx, \qquad \forall \veps \in (0,R);
				\end{split}
			\end{equation}
			see e.g. \cite[Theorem 1.5]{Fabes1982local}. 
			
		We first consider the case when $\beta \geq 0$; the first Poincar\'e inequality in \eqref{eq:TraceLeb:Poincare} implies
			\begin{equation}\label{eq:PoincareConsequence}
				\int_{B_\veps} |u(\bx) - (u)_{B_\veps}|^p \, \rmd \bx \leq C \veps^{\beta+p} \int_{B_\veps} |\grad u(\bx)|^p \, \gamma(\bx)^{-\beta} \rmd \bx, \qquad \forall \veps \in (0,R).
			\end{equation}
			
			Now, for $M$, $N \in \bbN$ with $M \leq N$, we write the difference as a telescoping sum and use H\"older's inquality to get
			\begin{equation*}
				\begin{split}
					|(u)_{B(2^{-M})} - (u)_{B(2^{-N})} | 
					&\leq \sum_{k=M+1}^{N} |(u)_{B(2^{-k-1})} - (u)_{B(2^{-k})} | \\
					&\leq \sum_{k=M+1}^{N} \fint_{B(2^{-k-1})} |u(\by) - (u)_{B(2^{-k})}| \, \rmd \by \\
					&\leq C^{1/p} \sum_{k=M+1}^{N} \left( \fint_{B(2^{-k})} |u(\by) - (u)_{B(2^{-k})}|^p \, \rmd \by \right)^{1/p},
				\end{split}
			\end{equation*}
			where $C$ is a number depending only on $d$ and $\Omega$ that satisfies $\frac{|B(\bx_0,2^{-k}) \cap \Omega|}{ |B(\bx_0,2^{-k-1}) \cap \Omega| } \leq C$. Since $\Omega$ is a Lipschitz domain with ``punctures,'' $C$ is independent of $k$ and $\bx_0$.
			Applying \eqref{eq:PoincareConsequence},
			\begin{equation}\label{eq:TraceLebPf1}
				\begin{split}
					|(u)_{B(2^{-M})} - (u)_{B(2^{-N})} |
					&\leq \sum_{k=M+1}^{N} \left( C 2^{-k(p+\beta-d)} \int_{B(2^{-k})} \gamma(\by)^{-\beta} |\grad u(\by)|^p \, \rmd \by \right)^{1/p} \\
					&\leq C [u]_{W^{1,p}(\Omega;\beta)} \sum_{k = M+1}^N (2^{(d-p-\beta)/p})^k.
				\end{split}
			\end{equation}
			
			Note that for any $0 < r'/2 < r \leq r' < R$ we have by \eqref{eq:PoincareConsequence}
			\begin{equation*}
				\begin{split}
					r^d \big| (u)_{B_r} - (u)_{B_{r'}} \big|^p 
					&= C \int_{B_r} \big| (u)_{B_r} - (u)_{B_{r'}} \big|^p \, \rmd \bx \\
					&\leq C \int_{B_r} \big| u(\bx) - (u)_{B_{r'}} \big|^p \, \rmd \bx 
					+ C \int_{B_{r'}} \big| u(\bx) - (u)_{B_r} \big|^p \, \rmd \bx \\
					&\leq C ( r^{p+\beta} + (r')^{p+\beta} ) \int_{B_{r'}} \gamma(\bx)^{-\beta} |\grad u(\bx)|^p \, \rmd \bx,
				\end{split}		
			\end{equation*}
			so that
			\begin{equation}\label{eq:TraceLebPf2}
				\big| (u)_{B_r} - (u)_{B_{r'}} \big| \leq C (r')^{\frac{\beta+p-d}{p}} [u]_{W^{1,p}(\Omega;\beta)}, \qquad \forall r, r' \text{ with } \frac{r}{r'} \in \left( \frac{1}{2}, 1 \right).
			\end{equation}
			
			Now, let $\veps_j \to 0$ as $j \to \infty$ be an arbitrary sequence. We need to show that $\{(u)_{B_{\veps_j} }\}_j$ is Cauchy. For each $j$, $j' \in \bbN$, let $N_j$ and $M_{j'}$ be the two integers satisfying $2^{-N_j-1} \leq \veps_j \leq 2^{-N_j}$ and $2^{-M_{j'}-1} \leq \veps_{j'} \leq 2^{-M_{j'}}$. 
			Therefore, by \eqref{eq:TraceLebPf1} and \eqref{eq:TraceLebPf2}
			\begin{equation*}
				\begin{split}
					\big| (u)_{B_{\veps_j}} - (u)_{B_{\veps_j}} \big|
					&\leq \big| (u)_{B_{\veps_j}} - (u)_{B(2^{-N_j-1})} \big| + \big| (u)_{B(2^{-N_j-1})} - (u)_{B(2^{-M_{j'}})} \big| \\
					&\qquad + \big| (u)_{B(2^{-M_{j'}})} - (u)_{B(\veps_{j'})} \big| \\
					&\leq C [u]_{W^{1,p}(\Omega;\beta)} \sum_{k = M_{j'} }^{ N_{j} } (2^{(d-p-\beta)/p})^k.
				\end{split}
			\end{equation*}
			Since $M_{j'}$, $N_j \to \infty$ as $\min\{j,j'\} \to \infty$, the right-hand side forms the tail of a geometric series, and so we see that the sequence $\{(u)_{B(\veps_j)}\}_j$ is Cauchy. Therefore the sequence converges to a number, which we call $u(\bx_0)$ (consistent with Lebesgue differentiation), and the convergence rate follows from taking $N_j \to \infty$ in the above estimate.
			
 Next, we consider the case when
 $\beta < 0$; the proof proceeds similarly. Since $\Omega$ satisfies the interior cone condition, we get
			\begin{equation}\label{eq:IntConeCond:Conseq}
				C(d) \veps^{d-\beta} \leq \int_{B(\bx_0,\veps) \cap \Omega} |\bx-\bx_0|^{-\beta} \, \rmd \bx \leq C'(d) \veps^{d-\beta}, \qquad \forall \veps > 0.
			\end{equation}
			Therefore we have 
			\begin{equation*}
				\begin{split}
					&|(u)_{B(2^{-k-1}),\beta} - (u)_{B(2^{-k}),\beta}| \\
					\leq& C 2^{-k(\beta-d)} \int_{B(2^{-k-1})} |u(\by) - (u)_{B(2^{-k}),\beta}| \gamma(\bx)^{-\beta} \, \rmd \bx \\
					\leq& C 2^{-k(\beta-d)/p} \left( \int_{B(2^{-k-1})} |u(\by) - (u)_{B(2^{-k}),\beta}|^p \gamma(\bx)^{-\beta} \, \rmd \bx \right)^{1/p},
				\end{split}
			\end{equation*}
			and similarly to \eqref{eq:TraceLebPf1} we use the second Poincar\'e inequality in \eqref{eq:TraceLeb:Poincare} to obtain
			\begin{equation}\label{eq:TraceLebPf1:Case2}
				\begin{split}
					|(u)_{B(2^{-M}),\beta} - (u)_{B(2^{-N}),\beta} |
					&\leq C [u]_{W^{1,p}(\Omega;\beta)} \sum_{k = M+1}^N (2^{(d-p-\beta)/p})^k.
				\end{split}
			\end{equation}
			In the same way as \eqref{eq:TraceLebPf2} we have
			\begin{equation}\label{eq:TraceLebPf4}
				\big| (u)_{B_r,\beta} - (u)_{B_{r'},\beta} \big| \leq C (r')^{\frac{\beta+p-d}{p}} [u]_{W^{1,p}(\Omega;\beta)}, \qquad \forall\,  \frac{r}{r'} \in \left( \frac{1}{2}, 1 \right).
			\end{equation}
			The rest of the proof proceeds identically to the first case; we obtain that the sequence \\$\{(u)_{ B(\veps_j),\beta } \}_j$ is Cauchy. 
			Finally, we note that since $\beta < 0$, we have for $u \in C^\infty_c(\bbR^d)$
			\begin{equation*}
				\begin{split}
				|(u)_{B(\bx_0,\veps),\beta} - u(\bx_0)| 
				&\leq C \int_{B(0,1)} \frac{ |u(\bx_0+\veps \bz) - u(\bx_0)| }{ |\bz|^{\beta} } \, \rmd \bz \\
				&\leq \int_{B(0,1)} |u(\bx_0+\veps \bz) - u(\bx_0)| \, \rmd \bz \to 0 \text{ as } \veps \to 0,
				\end{split}
			\end{equation*}
			hence the limit of $(u)_{ B(\veps_j),\beta }$ can be identified with $u(\bx_0)$.

            Finally, the fact that $T_\Gamma$ is a bounded operator in the case $\beta \geq 0$ follows from the estimate
            \begin{equation*}
                \begin{split}
                    |u(\bx_0)| &\leq C \veps^{ \frac{\beta-d+p}{p} } \vnorm{\grad u}_{L^p(\Omega;\beta)} + |(u)_{B_\veps}| \\
                    &\leq C_\veps (\vnorm{\grad u}_{L^{p}(\Omega;\beta)} + \vnorm{u}_{L^p(\Omega)}) \leq C_\veps (\vnorm{\grad u}_{L^{p}(\Omega;\beta)} + \vnorm{u}_{L^p(\Omega;\beta)}),
                \end{split}
            \end{equation*}
            where Jensen's inequality was applied to $(u)_{B_\veps}$. The same result in the case $\beta < 0$ follows
        similarly, using $(u)_{B_\veps,\beta}$ in place of $(u)_{B_\veps}$.
		\end{proof}
		
		The following theorem on extensions follows easily from the relevant definitions and straightforward calculations.
		
		\begin{theorem}\label{thm:Extension}
			Given a function $g : \Gamma \to \bbR$, let $\psi \in C^{\infty}_c(\bbR^d)$ satisfy $0 \leq \psi \leq 1$, with $\psi(\bx) = 1$ for $|\bx|<1$ and $\psi(\bx) = 0$ for $|\bx| \geq 2$, and define $E_\Gamma g : \bbR^d \to \bbR$ by
			\begin{equation*}
				E_\Gamma g (\bx) := \sum_{\bx_0 \in \Gamma} g(\bx_0) \psi \left( \frac{\bx-\bx_0}{R} \right).
			\end{equation*}
			Then for $\beta < d$, $E_\Gamma g$ is a linear extension operator on $W^{1,p}(\Omega;\beta)$, i.e. $E_\Gamma g = g$ on $\Gamma$ and $\vnorm{E_\Gamma g}_{W^{1,p}(\Omega;\beta)} \leq C(\beta,\Omega) \vnorm{g}_{L^{\infty}(\Gamma)}$.
		\end{theorem}

		\subsection{Homogeneous weighted spaces}

		We now define another Banach space 
		\begin{equation}
			V^{1,p}(\Omega;\beta) := \{ u \in L^p(\Omega;\beta+p) \, :\, \vnorm{u}_{ V^{1,p}(\Omega;\beta) } < \infty \}, \qquad \beta \in \bbR,
		\end{equation}
		with norm defined by 
		\begin{equation*}
			\vnorm{u}_{ V^{1,p}[\delta](\Omega;\beta) }^p := \Vnorm{u}_{L^p(\Omega;\beta+p)}^p + \Vnorm{\grad u}_{L^p(\Omega;\beta)}^p.
		\end{equation*}
		This space will help us further understand $W^{1,p}(\Omega;\beta)$, and also give a clear framework for the variational problems.
		
		\begin{theorem}[\cite{Rakosnik1989embeddings}, Theorem 1]\label{thm:Density:V}
			For any $\beta \in \bbR$, $C^\infty_c(\overline{\Omega} \setminus \Gamma)$ is dense in $V^{1,p}(\Omega;\beta)$.
		\end{theorem}
  
		\begin{theorem}[Embedding]\label{thm:Hardy}
			Let $d - p < \beta$. Then there exists a constant $C = C(d,p,\beta,\Omega)$ such that
			\begin{equation}\label{eq:Hardy:1Pt}
				\int_{\Omega} \frac{|u(\bx)|^p }{ \gamma(\bx)^{\beta+p} } \, \rmd \bx \leq C \int_{\Omega} \frac{|\grad u(\bx)|^p }{ \gamma(\bx)^{\beta} } \, \rmd \bx, \qquad \forall u \in V^{1,p}(\Omega;\beta).
			\end{equation}
		\end{theorem}
		
		\begin{proof}
			By a localization argument, we can again assume that $\Gamma = \{\bx_0\}$ and by \Cref{thm:Density:V} we can assume that $u \in C^{\infty}_c(\overline{\Omega} \setminus \Gamma)$.
			If $\bx_0 \in \p \Omega^*$, then the result is proved in \cite[Theorem 8.15]{Kufner1980Weighted}. So we assume that $\bx \in \Omega^*$.
			Thus, it suffices to show that
			\begin{equation}\label{eq:Hardy:Localized}
				\int_{B(\bx_0,R)} \frac{|u(\bx)|^p}{|\bx-\bx_0|^{\beta+p}} \, \rmd \bx \leq C \int_{B(\bx_0,R)} \frac{|\grad u(\bx)|^p}{|\bx-\bx_0|^\beta} \, \rmd \bx, \qquad \forall u \in C^\infty_c(B(\bx_0,R) \setminus \{\bx_0\}).
			\end{equation}
			This inequality is invariant under translations and dilations, so we need only show \eqref{eq:Hardy:Localized} for the case $\bx_0 = 0$ and $R = 1$.
			
			Since $\beta > d - p$, we can apply the one-dimensional Hardy inequality \eqref{eq:Hardy1D:0} to the function $v(r) = u(r\bsomega)$ for $\bsomega \in \bbS^{d-1}$, which satisfies $v \in C^1_c((0,\infty)) \cap C^0([0,\infty))$ and  $v(0) = 0$:
			\begin{equation*}
				\int_0^\infty \frac{|v(r)|^p}{ r^{\beta-d+1+p} } \, \rmd r \leq C(\beta,d,p) \int_0^\infty \frac{|v'(r)|^p}{ r^{\beta-d+1} } \, \rmd r.
			\end{equation*}
			Integrating this inequality over $\bsomega \in \bbS^{d-1}$, we obtain the desired inequality \eqref{eq:Hardy:Localized} after reverting from polar coordinates (also using that $|\p_r u| \leq |\grad_{\bx} u|$).
		\end{proof}
		
        Now, we define a homogeneous space using the closure with respect to the $W^{1,p}(\Omega;\beta)$-norm:
		\begin{equation*}
			W^{1,p}_{0,\Gamma}(\Omega;\beta) :=\overline{C^{\infty}_c(\overline{\Omega} \setminus \Gamma)}^{\vnorm{\cdot}_{W^{1,p}(\Omega;\beta)}}, \text{ for } \beta \in \bbR.
		\end{equation*}
        The following theorem relates the space $V^{1,p}(\Omega;\beta)$ to the spaces $W^{1,p}(\Omega;\beta)$ and $W^{1,p}_{0,\Gamma}(\Omega;\beta)$:

		\begin{theorem}\label{thm:V=W}
			For $\beta < d - p$, $W^{1,p}(\Omega;\beta) = V^{1,p}(\Omega;\beta)$. For $d- p < \beta$, $W^{1,p}_{0,\Gamma}(\Omega;\beta) = V^{1,p}(\Omega;\beta)$.
		\end{theorem}
		
		\begin{proof}
			Since $\vnorm{u}_{L^p(\Omega;\beta)} \leq C(p,\Omega) \vnorm{u}_{L^p(\Omega;\beta+p)}$, it suffices to show that
                \begin{align}
                \label{eq:V=W:Pf1}
				\vnorm{u}_{L^p(\Omega;\beta+p)} &\leq C \vnorm{u}_{W^{1,p}(\Omega;\beta)}, \qquad \forall u \in W^{1,p}(\Omega;\beta) \text{ when } \beta < d - p, \text{ and } \\
                \label{eq:V=W:Pf2}
				\vnorm{u}_{L^p(\Omega;\beta+p)} &\leq C \vnorm{u}_{W^{1,p}(\Omega;\beta)}, \qquad \forall u \in W^{1,p}_{0,\Gamma}(\Omega;\beta) \text{ when } \beta > d - p.
                \end{align}
			As before, we may assume that $\Gamma= \{\bx_0\}$. The result is stated and proved in \cite[Theorem 2.3]{Edmunds1985Embeddings} when $\bx_0 \in \p \Omega^*$, so we assume that $\bx_0 \in \Omega^*$. Again using a similar localization argument, and noting that \eqref{eq:V=W:Pf1}-\eqref{eq:V=W:Pf2} are invariant under translation and dilation, we may assume that $\bx_0 = 0$ and that $u \in C^\infty_c(B(0,1) \setminus \{0\})$.
			
			First, assume that $\beta < d - p$. 			
			Then apply the Hardy inequality \eqref{eq:Hardy1D:inf} to the function $v(r) = u(r\bsomega)$ for $\bsomega \in \bbS^{d-1}$:
			\begin{equation*}
				\int_0^\infty \frac{|v(r)|^p}{ r^{\beta-d+1+p} } \, \rmd r \leq C(\beta,d,p) \int_0^\infty \frac{|v'(r)|^p}{ r^{\beta-d+1} } \, \rmd r, \qquad \text{ for } -\infty < \beta - d + 1 < 1-p.
			\end{equation*}
			Integrating over $\bsomega \in \bbS^{d-1}$, we obtain the desired inequality \eqref{eq:V=W:Pf1} after reverting from polar coordinates (also using that $|\p_r u| \leq |\grad_{\bx} u|$).
			
			Now assume that $\beta > d - p$. Then the Hardy inequality \eqref{eq:Hardy:1Pt} 
			actually holds for all $u \in C^\infty(\overline{\Omega}\setminus \Gamma)$. Then \eqref{eq:V=W:Pf2} follows by density of $C^\infty(\overline{\Omega}\setminus \Gamma)$ in $W^{1,p}_{0,\Gamma}(\Omega;\beta)$ as per the definition.
			
		\end{proof}
    
		Although we did not define the space for $\beta \geq d$, this result, in essence, says that functions with finite $W^{1,p}(\Omega;\beta)$ norm must have trace $0$ on $\Gamma$ when $\beta \geq d$. On the other hand, when $\beta < d-p$, this result says that functions in $W^{1,p}(\Omega;\beta)$ do not have traces on $\Gamma$, since as a consequence both $C^\infty(\overline{\Omega})$ and $C^\infty(\overline{\Omega} \setminus \Gamma)$ are dense in $W^{1,p}(\Omega;\beta)$.
        In the third regime, $d-p < \beta < d$, we can characterize the homogeneous space $W^{1,p}_{0,\Gamma}(\Omega;\beta)$ using the trace as follows:

        \begin{theorem}\label{thm:TraceChar}
			For $d - p < \beta < d$, $u \in W^{1,p}(\Omega;\beta)$ belongs to $W^{1,p}_{0,\Gamma}(\Omega;\beta)$ if and only if and $T_{\Gamma} u = 0$.
		\end{theorem}

        \begin{proof}
            One implication is clear thanks to continuity of the trace; we will provide a proof for the reverse implication. Assume that $T_{\Gamma} u = 0$.
            We assume that $\Gamma = \{\bx_0\}$, that $\gamma(\bx) = |\bx-\bx_0|$, that $R = 1$, and that $\bx_0 = 0$; the general case will follow by a using a localization/partition of unity argument, as well as noting the density condition is invariant under translations and dilations. We may also assume -- flattening $\p \Omega^*$ if necessary -- that either $\Omega = B(0,1)$ or $\Omega = B(0,1) \cap \{ x_d > 0 \}$. If $\Omega^* = B(0,1)$ we can assume that $\supp u \Subset B(0,1)$.

            \underline{Step 1:} By \Cref{thm:Density} there exists a sequence $\{u_n\} \subset C^\infty(\overline{\Omega})$ such that $\lim\limits_{n \to \infty} \vnorm{u_n - u}_{W^{1,p}(\Omega;\beta)} = 0$ and by continuity of the trace we have $\lim\limits_{n \to \infty} |T_\Gamma u_n| = 0$. For $0 < r < 1$ and $\bsomega \in \cC := \{ \frac{\bx}{|\bx|} \, :  \bx \in \Omega \} \subset \bbS^{d-1}$, define $v_n(r,\bsomega) = u_n(r\bsomega)$ and $v(r,\bsomega) = u(r\bsomega)$. Then
            \begin{equation*}
                v_n(r,\bsomega) - u_n(0) = \int_0^r \p_r v_n(\varrho,\bsomega) \, \rmd \varrho = \int_0^r ( \p_r v_n(\varrho,\bsomega) \varrho^{\frac{d-1-\beta}{p}} ) \varrho^{-\frac{d-1-\beta}{p}} \, \rmd \varrho,
            \end{equation*}
            so that by H\"older's inequality
            \begin{equation*}
                \begin{split}
                    \int_{\cC} |v_n(r,\bsomega)|^p \, \rmd \sigma(\bsomega )
                    &\leq C |u_n(0)|^p + C \int_{\cC} \left( \int_0^r \varrho^{ - \frac{d-1-\beta}{p-1} } \, \rmd \varrho \right)^{p-1} \int_0^r \frac{|\p_r v_n(\varrho,\bsomega)|^p}{ \varrho^{\beta-d+1} } \, \rmd \varrho \, \rmd \sigma(\bsomega) \\
                    &\leq C |u_n(0)|^p + C r^{\beta-d+p} \int_0^r \int_{\cC} \frac{|\p_r v_n(\varrho,\bsomega)|^p}{ \varrho^{\beta-d+1} } \, \rmd \sigma(\bsomega) \, \rmd \varrho;
                \end{split}
            \end{equation*}
            the right-hand side is finite since $\beta > d - p$. Converting from polar coordinates, using that $|\p_r v| \leq |\grad_\bx u|$, and letting $n \to \infty$, we get that
            \begin{equation}\label{eq:TraceChar:Pf1}
                \int_{\cC} |v(r,\bsomega)|^p \, \rmd \sigma(\bsomega ) \leq C r^{\beta-d+p} \int_{B(0,r) \cap \Omega } \frac{|\grad u(\bx)|^p}{ |\bx|^\beta } \, \rmd \bx, \qquad \text{ for a.e. } r \in (0,1). 
            \end{equation}

            \underline{Step 2:} The rest of the proof is similar to that of \Cref{thm:Density}. Let $\zeta \in C^\infty(\bbR^d)$ be a function satisfying $0 \leq \zeta \leq 1$, $\zeta \equiv 1$ on $B(0,1)$ and $\zeta \equiv 0$ on $\bbR^d \setminus B(0,2)$. Then define $u_k(\bx) := u(\bx)(1-\zeta(k\bx))$. Then
			\begin{equation*}
				\begin{split}
					\int_{B(0,1) \cap \Omega} \frac{|\grad u_k - \grad u|^p}{|\bx|^\beta} \, \rmd \bx &\leq C \int_{B(0,2/k) \cap \Omega} \frac{|\grad u|^p}{|\bx|^\beta} \, \rmd \bx \\
                    &\quad + C k^p \int_{\Omega \cap (B(0,2/k) \setminus B(0,1/k))} \frac{|u(\bx)|^p}{|\bx|^\beta} \, \rmd \bx.
				\end{split}
			\end{equation*}
                We estimate the second integral by converting to polar coordinates and using \eqref{eq:TraceChar:Pf1}:
                \begin{equation*}
                    \begin{split}
                        &k^p \int_{\Omega \cap (B(0,2/k) \setminus B(0,1/k))} \frac{|u(\bx)|^p}{|\bx|^\beta} \, \rmd \bx \\
                        =& k^p \int_{1/k}^{2/k} \int_{\cC} \frac{|v(r,\bsomega)|^p}{r^{\beta-d+1}} \, \rmd \sigma(\bsomega) \, \rmd r \\
                        \leq& C k^p \int_{1/k}^{2/k} r^{p-1} \, \rmd r \int_{B(0,2/k) \cap \Omega } \frac{|\grad u(\bx)|^p}{ |\bx|^\beta } \, \rmd \bx 
                        \leq C \int_{B(0,2/k) \cap \Omega } \frac{|\grad u(\bx)|^p}{ |\bx|^\beta } \, \rmd \bx.
                    \end{split}
                \end{equation*}
			Therefore by continuity of the integral we get that
            \begin{equation}\label{eq:TraceChar:Pf2}
                \vnorm{u_k - u}_{W^{1,p}(\Omega;\beta)} \leq C \vnorm{u}_{W^{1,p}(\Omega \cap B(0,2/k);\beta)} \to 0 \text{ as } k \to \infty.
            \end{equation}
            We note that $u_k = 0$ on the set $\{ \bx \in \Omega \, : \, \gamma(\bx) < 1/k \}$. Therefore, $\beta + p > d$, but the weight $\gamma^{-\beta-p}$ is nonsingular on $\supp u_k$. Thus $\vnorm{u_k}_{V^{1,p}(\Omega;\beta)} < \infty$, and so by \Cref{thm:Density:V} we can for each $k \in \bbN$ choose a function $w_k \in C^{\infty}_c(\overline{\Omega} \setminus \Gamma)$ such that
            \begin{equation*}
                \vnorm{w_k - u_k}_{V^{1,p}(\Omega;\beta)} < \frac{1}{C' 2^k},
            \end{equation*}
            where $C' = C'(d,\beta,p,\Omega)$ is the constant for which $\vnorm{f}_{W^{1,p}(\Omega;\beta)} \leq C' \vnorm{f}_{V^{1,p}(\Omega;\beta)}$ for all $f \in V^{1,p}(\Omega;\beta)$. Therefore combining this with \eqref{eq:TraceChar:Pf2} gives
            \begin{equation*}
                \vnorm{w_k - u}_{W^{1,p}(\Omega;\beta)} \leq \vnorm{u_k-u}_{W^{1,p}(\Omega;\beta)} + \vnorm{w_k-u_k}_{W^{1,p}(\Omega;\beta)} \to 0 \text{ as } k \to \infty.
            \end{equation*}
        \end{proof}

{
\begin{remark}
    These results are consistent when $\beta=0$, which corresponds to the unweighted classical Sobolev space $W^{1,p}(\Omega)$. When $p < d$, we cannot make sense of Sobolev functions defined at a single point, consistent with the nonexistence of $T_\Gamma$ for $d-p < \beta$. On the other hand, when $p > d$, Sobolev functions are actually H\"older continuous by Morrey's inequality, i.e. $T_\Gamma$ exists for $d-p < \beta$.
\end{remark}
}
		
  The following theorem is a direct consequence by applying the earlier results in the following order: \Cref{thm:Trace}, \Cref{thm:Extension}, \Cref{thm:TraceChar}, \Cref{thm:V=W}, and \Cref{thm:Hardy}:
		
		\begin{theorem}
			For $d-p < \beta < d$ and for any $u \in W^{1,p}(\Omega;\beta)$, the function $\bar{u} := u - E_\Gamma \circ T_\Gamma u$ satisfies $\bar{u} \in V^{1,p}(\Omega;\beta)$. Moreover, for any $\bx_0 \in \Gamma$ the following Hardy inequality holds:
			\begin{equation*}
				\int_{\Omega \cap B(\bx_0,R)} \frac{|u(\bx)-u(\bx_0)|^p}{|\bx-\bx_0|^{\beta+p}} \, \rmd \bx \leq C(d,p,\beta) \int_{\Omega \cap B(\bx_0,R)} \frac{|\grad u(\bx)|^p}{|\bx-\bx_0|^{\beta}} \, \rmd \bx.
			\end{equation*}
		\end{theorem}
		
		Finally, we conclude the section with an embedding theorem for weighted Sobolev spaces:
		\begin{theorem}\label{thm:CompactEmbed:ZeroDimCase}
			Let $q \in [1,\frac{dp}{d-p}]$, and let $\alpha \in \bbR$. The space $V^{1,p}(\Omega;\beta)$ is continuously embedded in $L^q(\Omega;\alpha)$ if and only if $\alpha$ and $\beta$ satisfy $d \left( \frac{1}{q} - \frac{1}{p} \right) - \frac{\alpha}{q} + \frac{\beta}{p} + 1 \geq 0$.
			
			If $q \in [1,\frac{dp}{d-p})$ and if $d \left( \frac{1}{q} - \frac{1}{p} \right) - \frac{\alpha}{q} + \frac{\beta}{p} + 1 > 0$ then the embedding is compact.
		\end{theorem}
		
		The proof of \Cref{thm:CompactEmbed:ZeroDimCase} is postponed to \Cref{sec:GeneralCodim}, in which more general sets $\Gamma$ are simultaneously treated.

\section{Nonlocal weighted function spaces}\label{sec:WeightedNonlocalSpaces}

We introduce the nonlocal Banach space with
        a finer scale of weights
		\begin{equation}
			\mathfrak{V}^{p}[\delta](\Omega;\beta) := \{ u \in L^p(\Omega;\beta+p) \, :\, \vnorm{u}_{ \mathfrak{V}^{p}[\delta](\Omega;\beta) } < \infty \},  \qquad \beta \in \bbR,
		\end{equation}
		equipped with a norm defined by 
		\begin{equation*}
			\vnorm{u}_{ \mathfrak{V}^{p}[\delta](\Omega;\beta) }^p := \Vnorm{u}_{L^p(\Omega;\beta+p)}^p + [u]_{\mathfrak{W}^{p}[\delta](\Omega;\beta)}^p.
		\end{equation*}
        we also define the space that is homogeneous with respect to $\Gamma$
		\begin{equation*}
			\mathfrak{W}^{p}_{0,\Gamma}[\delta](\Omega;\beta) :=\overline{C^{\infty}_c(\overline{\Omega} \setminus \Gamma)}^{\vnorm{\cdot}_{\mathfrak{W}^{p}[\delta](\Omega;\beta)}}, \qquad  \beta \in \bbR.
		\end{equation*}

        Throughout this section, we take the assumptions $p \in (1,\infty)$, 
 \eqref{assump:weight}, \eqref{Assump:Kernel}, \eqref{assump:Horizon}, and $\beta \in \bbR$ unless noted otherwise.
				
		\subsection{Boundary-localized convolutions}\label{sec:LocalizedConvolution}

        In order to analyze the nonlocal variational problems, we seek an understanding of the nonlocal function spaces, analogous to the study of the local weighted Sobolev spaces in \Cref{sec:WeightedLocalSpaces}. To do this, we use a strategy similar to that of \cite{Scott2023Nonlocal,Scott2023Nonlocala}, and establish the desired properties (density of smooth functions, traces, embeddings, etc.) of the nonlocal function spaces by leveraging estimates for 
        the following convolution-type operator
\begin{equation}\label{eq:ConvolutionOperator}
			K_{\delta}u (\bx) := \int_{\Omega^0} \frac{1}{(\lambda_\delta(\bx))^d} \psi \left( \frac{|\by-\bx|}{ \lambda_\delta(\bx) } \right) u(\by) \,\rmd \by , \;\; \bx \in \Omega.
		\end{equation}
		Here, $\psi:\bbR \to [0,\infty)$ is a standard mollifier satisfying
		\begin{equation}\label{Assump:Kernel}
			\begin{gathered}
				\psi \in C^{k}(\bbR) \text{ for some } k \in \bbN_0 \cup \{\infty\},
				\; \psi(x) \geq 0\, \text{ and } \, \psi(-x) = \psi(x),\;\forall x\in\bbR,\\
				[-c_\psi,c_{\psi}] \subset \supp \psi \Subset (-1,1) \text{ for fixed } c_{\psi} > 0, \; \text{ and }
				\int_{\bbR^d} \psi(|\bx|) \, \rmd \bx = 1.
			\end{gathered}
			\tag{\ensuremath{\rmA_{\psi}}}
		\end{equation}
		
        The function governing the dilation in the mollifier is defined using the notation of \eqref{eq:localizationfunction}, that is, $\lambda_\delta(\bx) = \delta \lambda(\bx)$ where $\lambda :\overline{\Omega} \to [0,\infty)$ is a given function.
        This generalized distance -- or generalized heterogeneous localization -- function $\lambda$  satisfies the following:
		\begin{equation}\label{assump:Localization}
			\begin{aligned}
				i) & \, \text{there exists a constant } \kappa_0 \geq 1 \text{ such that } \\
				&\quad \frac{1}{\kappa_0} \dist(\bx,\p \Omega) \leq \lambda(\bx) \leq \kappa_0 \dist(\bx,\p \Omega),\; \forall  \bx \in \overline{\Omega}\,;\\
				ii) & \, \text{there exists a constant } \kappa_1 > 0 \text{ such that } \\
				&\quad |\lambda(\bx) - \lambda(\by)| \leq \kappa_1 |\bx-\by|, \; \forall \bx,\by \in \overline{\Omega}; \\
				iii) & \, \text{there exists } k_\lambda \in \bbN_0 \cup \{\infty\} \text{ such that the following holds: } \\
                & \quad \lambda \in C^0(\overline{\Omega}) \cap C^{k_\lambda}(\Omega) 
                \text{ and } 
                \text{for each multi-index } \alpha \in \bbN^d_0
                \text{ with } |\alpha| \leq k_\lambda, 
                \\
				&\quad \exists \kappa_{\alpha} > 0 \text{ such that } |D^\alpha \lambda(\bx)| \leq \kappa_{\alpha} |\dist(\bx,\p \Omega)|^{1-|\alpha|}, \;  \forall \bx \in \Omega.
			\end{aligned} \tag{\ensuremath{\rmA_{\lambda}}}
		\end{equation}
		For any domain $\Omega$, a generalized distance function $\lambda$ with 
        $k_\lambda = \infty$ and 
        with all $\kappa_\alpha$ depending only on $d$ is guaranteed to exist \cite[Chapter VI, Theorem 2]{Stein1970Singular}.
        We also note that \eqref{assump:weight} and \eqref{assump:Localization} have the quantities $\kappa_\alpha$ in common.
		
		As in \cite{Scott2023Nonlocal, Scott2023Nonlocala}, we refer to $K_{\delta}$ as a \textit{boundary-localized convolution} operator. This operator has all of the smoothing properties of classical convolution operators, and additionally recovers the boundary values of a function. To be precise, for all functions $u \in C^0(\overline{\Omega})$, $T K_\delta u = T u$, where $T u = u|_{\p \Omega}$ denotes the trace operator. This property of the boundary-localized convolution is preserved when the operator $T$ is extended to more general Sobolev and nonlocal function spaces.

		We now further specify notation. For any function $\psi: [0,\infty) \to \bbR$, we
		define
		\begin{equation}\label{eq:OperatorKernelDef}
			\begin{split}
				\psi_{\delta}(\bx,\by) &:= \frac{1}{\lambda_\delta(\bx)^{d}} {\psi} \left( \frac{|\by-\bx|}{\lambda_\delta(\bx)} \right) .
			\end{split}
		\end{equation}
		In particular, $\psi_{\delta}$
		defines a \textit{boundary-localizing} mollifier
		corresponding to a standard mollifier $\psi$ described in 
		\eqref{Assump:Kernel}. 
		Note that $\int_{\Omega} \psi_\delta(\bx,\by) \, \rmd \by = 1$ for all $\bx \in \Omega$ and for all $\delta < \underline{\delta}_0$. This is not the case when the arguments are reversed, and so we define the function
		\begin{equation}\label{eq:KernelIntegral}
			\Psi_{\delta}(\bx):= \int_{\Omega} \psi_{\delta}(\by,\bx) \, \rmd \by.
		\end{equation}

		In the event that $\Gamma = \emptyset$, such boundary-localized convolutions have been analyzed previously in \cite{Scott2023Nonlocal,Scott2023Nonlocala}. The next theorems, which hold for general $\Gamma$, are established by following similar lines of reasoning to those in the proofs of its analogue in \cite{Scott2023Nonlocal}. The adaptations to the present case $\Omega = \Omega^* \setminus \Gamma$, where $\Omega^*$ is a Lipschitz domain, are straightforward, and we provide them for completeness.
		
		\begin{theorem}\label{thm:KnownConvResults}
		  The following hold:
			\begin{enumerate}[1)]
				\item If $u \in L^1_{loc}(\Omega)$, then $K_{\delta} u \in C^{\infty}(\Omega)$.
				
				\item If $u \in C^0(\overline{\Omega})$,  then $K_{\delta} u \in  C^0(\overline{\Omega})$. Moreover, $K_{\delta} u(\bx) = u (\bx)$ for all $\bx \in \p \Omega$, and $K_{\delta} u \to u$ uniformly on $\overline{\Omega}$ as $\delta \to 0$.
				
				\item There exists a constant $C_0 = C_0(d,p,\beta,\psi,\kappa_1) > 0$ such that
				\begin{equation}\label{eq:ConvEst:Lp}
					\Vnorm{K_{\delta} u}_{L^p(\Omega;\beta)} \leq C_0 \Vnorm{u}_{L^p(\Omega;\beta)}, \quad \forall u \in L^p(\Omega;\beta),
				\end{equation}
				and in fact
				\begin{equation}\label{eq:ConvergenceOfConv}
					\lim\limits_{\delta \to 0} \Vnorm{K_{\delta} u - u}_{L^p(\Omega;\beta)} = 0, \quad \forall u \in L^p(\Omega;\beta).
				\end{equation}
				
				\item Let $\wt{K}_\delta$ be the operator defined for any $\bv : \Omega \to \bbR^d$ and $\bx \in \Omega$ as
				\begin{equation}\label{eq:AuxOperatorDefn}
					\wt{K}_{\delta} \bv(\bx) := \int_{\Omega}   \psi_{\delta}(\bx,\by) \left[ \bI - \frac{(\bx-\by) \otimes \grad \eta_\delta(\bx)}{\lambda_\delta(\bx)} \right] \bv(\by) \, \rmd \by.
				\end{equation}
				Then there exists a constant $C_1= C_1(d,p,\beta,\psi,\kappa_1) > 0$ such that
				\begin{align}
					\label{eq:ConvEst1:W1p:Pf1}
					&\grad K_\delta u(\bx) = \wt{K}_\delta[\grad u](\bx), \quad  \forall \, u \in V^{1,p}(\Omega;\beta), \, \bx \in \Omega, \\
					\label{eq:ConvEst:W1p}
					&\Vnorm{\grad K_{\delta} u}_{L^p(\Omega;\beta)} \leq C_1 \Vnorm{\grad u}_{L^{p}(\Omega;\beta)}, \;\; \forall u \in V^{1,p}(\Omega;\beta), 
					\\
					\label{eq:ConvergenceOfConv:W1p}	  &\lim\limits_{\delta \to 0} \Vnorm{K_{\delta} u - u}_{V^{1,p}(\Omega)} = 0, \quad \forall u \in V^{1,p}(\Omega).
				\end{align}
				Additionally, \eqref{eq:ConvEst1:W1p:Pf1}-\eqref{eq:ConvEst:W1p}-\eqref{eq:ConvergenceOfConv:W1p} also hold with $V^{1,p}(\Omega;\beta)$ replaced with $W^{1,p}(\Omega;\beta)$ and $\beta < d$.

				\item For $d-p < \beta < d$, the trace operator $T_\Gamma : W^{1,p}(\Omega;\beta) \to \bbR^{|\Gamma|}$ satisfies $T_\Gamma K_{\delta} u = T_\Gamma u$ for all $u \in W^{1,p}(\Omega;\beta)$.

			\end{enumerate}
		\end{theorem}
		
		\begin{proof}
            Item 1) and Item 2) are straightforward to verify in a direct way.

            Now we show \eqref{eq:ConvEst:Lp}. 
            By H\"older's inequality and Tonelli's theorem
			\begin{equation*}
				\vnorm{K_\delta u}_{L^p(\Omega;\beta)}^p \leq  \int_\Omega \int_{\Omega} \psi_\delta(\bx,\by) \gamma(\bx)^{-\beta} \, \rmd \bx \; |u(\by)|^p \, \rmd \by.
			\end{equation*}
                Now, we have that $\lambda(\bx) \leq \kappa_0 d_{\p \Omega}(\bx) \leq  \kappa_0 d_{\Gamma}(\bx) \leq \kappa_0^2 \gamma(\bx)$, and so
			\begin{equation*}
				|\gamma(\by) - \gamma(\bx)| \leq \kappa_1 |\bx-\by| \leq \kappa_1 \lambda_\delta(\bx) \leq \delta \kappa_0^2 \kappa_1 \gamma(\bx), \quad \text{ for } |\bx-\by| \leq \lambda_\delta(\bx).
			\end{equation*}
			Thus, since $\kappa = \kappa_0^2 \kappa_1$ we have
			\begin{equation}\label{eq:Comp1}
				(1-\kappa \delta)\gamma(\bx) 
				\leq \gamma(\by)
				\leq (1+\kappa\delta) \gamma(\bx), \text{ for } |\bx-\by| \leq \lambda_\delta(\bx).
			\end{equation}
			Therefore thanks to the support of $\psi_\delta$ we have for any $\beta \in \bbR$
			\begin{equation*}
				\int_{\Omega} \psi_\delta(\bx,\by) \gamma(\bx)^{-\beta} \, \rmd \bx \leq C \Psi_\delta(\by) \gamma(\by)^{-\beta} \leq C \gamma(\by)^{-\beta},
			\end{equation*}
			which is \eqref{eq:ConvEst:Lp}.

            Thanks to the $L^p$-continuity of the operator $K_\delta$ established in \eqref{eq:ConvEst:Lp}, it suffices to show \eqref{eq:ConvergenceOfConv} for $u \in C^1(\overline{\Omega})$. But this in turn follows from the uniform convergence in item 2).

            The identity \eqref{eq:ConvEst1:W1p:Pf1} can be established by first differentiating the convolution with the variables changed, $K_\delta u(\bx) = \int_{B(0,1)} \psi(|\bz|) u(\bx + \lambda_\delta(\bx) \bz) \, \rmd \bz$, and then reversing the change of variables. Then \eqref{eq:ConvEst:W1p} follows from estimating $|\wt{K}_\delta [\grad u](\bx)| \leq (1+\delta \kappa_1) K_\delta [|\grad u|] (\bx)$ and then applying \eqref{eq:ConvEst:Lp}.
            The convergence \eqref{eq:ConvergenceOfConv:W1p} is proved analogously to \eqref{eq:ConvergenceOfConv}.

            To see item 5): by item 2), $K_\delta v = v$ on $\Gamma$ for all $v \in C^\infty(\overline{\Omega})$. Let $\veps > 0$; choosing $v \in C^{\infty}(\overline{\Omega})$ with $\vnorm{u-v}_{W^{1,p}(\Omega;\beta)} < \veps$ (which is possible by \Cref{thm:Density}) and then applying \Cref{thm:Trace}, \eqref{eq:ConvEst:Lp}, and \eqref{eq:ConvEst:W1p},
            \begin{equation*}
                \begin{split}
                    \vnorm{T_\Gamma K_\delta u - T_\Gamma u}_{L^\infty(\Gamma)} 
                    &\leq \vnorm{T_\Gamma K_\delta u - T_\Gamma K_\delta v}_{L^\infty(\Gamma)} 
                    + \vnorm{T_\Gamma K_\delta v - T_\Gamma v}_{L^\infty(\Gamma)} \\
                    &\qquad + \vnorm{T_\Gamma v - T_\Gamma u}_{L^\infty(\Gamma)} \\
                    &\leq C \vnorm{u-v}_{W^{1,p}(\Omega;\beta)} < C \veps.
                \end{split}
            \end{equation*}
		\end{proof}

With these results in hand, we turn to the relationship of the boundary-localized convolution with the nonlocal function spaces $\mathfrak{W}^p[\delta](\Omega;\beta)$ and $\mathfrak{V}^p[\delta](\Omega;\beta)$. The following results, proved in \Cref{sec:Apdx:Conv}, are valuable tools that will be used to establish properties of the nonlocal spaces essential for our analysis.

\begin{theorem}\label{thm:KnownConvResults:Nonlocal}
    The following hold:
    \begin{enumerate}[1)]
        \item There exists
		a constant $C_2
		= C_2(d,p,\beta,\psi,\kappa_0,\kappa_1)$
		such that 
		\begin{align}
		      \Vnorm{ K_{\delta} u - u }_{L^{p}(\Omega;\beta+p)}
				\leq C_2 \Vnorm{ (K_{\delta} u - u) \eta^{-1} }_{L^{p}(\Omega;\beta)}
				&\leq C_2 \delta [u]_{\mathfrak{W}^{p}[\delta](\Omega;\beta)},
                \label{eq:Intro:ConvEst} \\
                \text{ and } \Vnorm{ \grad K_{\delta} u }_{L^{p}(\Omega;\beta)} 
		      &\leq C_2 [u]_{\mathfrak{V}^{p}[\delta](\Omega;\beta)},
            \label{eq:Intro:ConvEst:Deriv}
		\end{align}
        for all $u \in \mathfrak{W}^{p}[\delta](\Omega;\beta)$.
		The same result holds for $\mathfrak{V}^p[\delta](\Omega;\beta)$ replaced with $\mathfrak{W}^p[\delta](\Omega;\beta)$ when $\beta < d$.

        \item There exist continuous functions $\theta : [0,\underline{\delta}_0) \to (0,1]$ and $\phi : [0,\underline{\delta}_0) \to [1,\infty)$ with $\theta(0) = \phi(0) = 1$ determined only from $d$, $p$, $\beta$, $\kappa_0$ and $\kappa_1$ such that the following holds: 
                For all $\veps \in (0,\underline{\delta}_0)$, and all $r \in (0,\diam(\Omega))$
                \begin{equation}\label{eq:Comp:ConvolutionEstimate}
                    \begin{split}
                    &\int_{\Omega 
                    \cap \{ \eta(\bx) < \theta(\veps) r  \} 
                    } \int_{\Omega} \frac{1}{\gamma(\bx)^\beta}  \rho \left( \frac{|\bx-\by|}{ \eta_{\theta(\veps)\delta}(\bx)} \right) \frac{ |K_\veps u(\bx) - K_\veps u(\by)|^p }{ \eta_{\theta(\veps)\delta}(\bx)^{d+p} } \, \rmd \by \, \rmd \bx \\
                    \leq& \phi(\veps) \int_{\Omega 
                    \cap \{ \eta(\bx) <  r \} 
                    } \int_{\Omega} \frac{1}{\gamma(\bx)^\beta} \rho \left( \frac{|\bx-\by|}{ \eta_{\delta}(\bx)} \right) \frac{ |u(\bx) - u(\by)|^p }{ \eta_{\delta}(\bx)^{d+p} } \, \rmd \by \, \rmd \bx,
                    \end{split}
                \end{equation}
                for all $u \in \mathfrak{V}^p[\delta](\Omega;\beta)$ (and for all $u \in \mathfrak{W}^p[\delta](\Omega;\beta)$ for $\beta < d$).

        \item $\lim\limits_{\veps \to 0} \vnorm{K_\veps u - u}_{\mathfrak{W}^p[\delta](\Omega;\beta)} = 0$, for all $u \in \mathfrak{V}^p[\delta](\Omega;\beta)$ and for all $u \in \mathfrak{W}^p[\delta](\Omega;\beta)$ with $\beta < d$.
        
    \end{enumerate}
\end{theorem}

  \subsection{Properties of weighted nonlocal function spaces}\label{subsec:NonlocalSpaces}

        \begin{theorem}\label{thm:InvariantHorizon}

        For constants $0 < \delta_1 \leq \delta_2 < \underline{\delta}_0$,
				\begin{equation*}
					\begin{aligned}
						\left( \frac{1-\delta_2}{2(1+\delta_2)} \right)^{\frac{d+p}{p}} \left( \frac{1-\kappa_0 \kappa_1 \delta_2}{1+\kappa_0 \kappa_1 \delta_2} \right)^{ \frac{|\beta|}{p}} [u]_{\mathfrak{W}^{p}[\delta_2](\Omega;\beta)} 
						&\leq [u]_{\mathfrak{W}^{p}[\delta_1](\Omega;\beta)} \\
						&\leq \left( \frac{\delta_2}{\delta_1} \right)^{\frac{d+p}{p}} [u]_{\mathfrak{W}^{p}[\delta_2](\Omega;\beta)},
					\end{aligned}
				\end{equation*}
        $\forall u \in \mathfrak{V}^{p}[\delta_2](\Omega;\beta)$ and $\forall u \in \mathfrak{W}^{p}[\delta_2](\Omega;\beta)$ with $\beta < d$.
        \end{theorem}

        \begin{theorem}\label{thm:EnergySpaceIndepOfKernel}
            For $\rho$ satisfying \eqref{assump:VarProb:Kernel} and $\lambda$ satisfying \eqref{assump:Localization}, define the seminorm
				\begin{equation*}
					[u]_{\wt{\mathfrak{W}}^{p}[\delta](\Omega;\beta)}^p := \int_{\Omega} \int_{\Omega} \frac{1}{\gamma(\bx)^\beta} \rho \left( \frac{  |\bx-\by| }{ \lambda_\delta(\bx)  } \right) \frac{ |u(\by)-u(\bx)|^p }{ \lambda_\delta(\bx)^{d+p} }\, \rmd \by \, \rmd \bx.
				\end{equation*}
				Then there exist positive constants $c$ and $C$ depending only on $d$, $p$, $\beta$, $\rho$, and $\kappa_0$ such that for any $u \in \mathfrak{W}^{p}[\delta](\Omega;\beta)$,
    \begin{equation}\label{eq:EnergySpaceIndepOfKernel:General}
					c [u]_{\mathfrak{W}^{p}[\delta](\Omega;\beta)} \leq
					[u]_{\wt{\mathfrak{W}}^{p}[\delta](\Omega;\beta)} \leq C  [u]_{\mathfrak{W}^{p}[\delta](\Omega;\beta)}.
				\end{equation}
                In particular,
                \begin{equation}\label{eq:EnergySpaceIndepOfKernel}
					c [u]_{\mathfrak{W}^{p}[\delta](\Omega;\beta)} \leq
					\cE_\delta(u) \leq C  [u]_{\mathfrak{W}^{p}[\delta](\Omega;\beta)}.
				\end{equation}
        \end{theorem}

        The proofs of \Cref{thm:InvariantHorizon} and \Cref{thm:EnergySpaceIndepOfKernel} are in the appendix.

        \begin{theorem}[An embedding result]\label{thm:Embedding:Nonlocal}
            Let $u \in V^{1,p}(\Omega;\beta)$, or let $u \in W^{1,p}(\Omega;\beta)$ when $\beta < d$. Then 
            \begin{equation}\label{eq:Embedding}
		      [u]_{\mathfrak{W}^{p}[\delta](\Omega;\beta)} \leq \frac{1}{(1-\delta)^{1/p}} \left(\frac{1+\kappa \delta}{1-\kappa \delta} \right)^{|\beta|/p} \vnorm{\grad u}_{L^p(\Omega;\beta)}.
		  \end{equation}
        \end{theorem}

        \begin{proof}
            First assume that $u \in C^{\infty}_c(\overline{\Omega} \setminus \Gamma)$ (respectively, $u \in C^{\infty}(\overline{\Omega})$). Then by changing coordinates $\bz = \frac{\by-\bx}{\eta_\delta(\bx)}$ and applying the fundamental theorem of calculus to the difference quotient of $u$
            \begin{equation*}
                \begin{split}
                    [u]_{\mathfrak{W}^{p}[\delta](\Omega;\beta)}^p 
                    &= \frac{(d+p)\overline{C}_{d,p}}{ \sigma(\bbS^{d-1}) } \int_{\Omega} \int_{B(0,1)} \frac{ |u(\bx + \eta_\delta(\bx) \bz) - u(\bx)|^p }{ \gamma(\bx)^\beta \eta_\delta(\bx)^p } \, \rmd \bz \, \rmd \bx \\
                    &\leq \frac{(d+p)\overline{C}_{d,p}}{ \sigma(\bbS^{d-1}) } \int_{\Omega} \int_{B(0,1)} \frac{ 1 }{ \gamma(\bx)^\beta } \int_0^1 |\grad u(\bx + t \eta_\delta(\bx) \bz) \cdot \bz|^p \, \rmd t \, \rmd \bz \, \rmd \bx.
                \end{split}
            \end{equation*}
            Next, define $\bszeta_{t\bz}^\delta(\bx) = \bx + t \eta_\delta(\bx)\bz$, c.f.\ \Cref{lma:CoordChange2}. Then use \eqref{eq:CoordChange2:Conseq} to estimate $\gamma(\bx)^{-\beta}$:
            \begin{equation*}
                [u]_{\mathfrak{W}^{p}[\delta](\Omega;\beta)}^p 
                \leq \left(\frac{1+\kappa \delta}{1-\kappa \delta} \right)^{|\beta|} 
                \int_{B(0,1)} \frac{(d+p)\overline{C}_{d,p}}{ \sigma(\bbS^{d-1}) } \int_0^1 \int_{\Omega}  \frac{|\grad u(\bszeta_{t\bz}^\delta(\bx)) \cdot \bz|^p}{ \gamma(\bszeta_{t\bz}^\delta(\bx))^\beta } \, \rmd \bx \, \rmd t \, \rmd \bz.
            \end{equation*}
            
            Now apply the change of coordinates $\bw = \bszeta_{t\bz}^\delta(\bx)$ in the $\bx$-integral. Then $\bw \in \Omega$ whenever $\bx \in \Omega$, and \eqref{eq:Det:DistanceFxn} holds. Therefore, 
            \begin{equation*}
                \begin{split}
                    [u]_{\mathfrak{W}^{p}[\delta](\Omega;\beta)}^p 
                    &\leq \left(\frac{1+\kappa \delta}{1-\kappa \delta} \right)^{|\beta|} \frac{1}{1-\delta}
                \int_{B(0,1)} \frac{(d+p)\overline{C}_{d,p}}{ \sigma(\bbS^{d-1}) } \int_0^1 \int_{\Omega}  \frac{|\grad u(\bw) \cdot \bz|^p}{ \gamma(\bw)^\beta } \, \rmd \bw \, \rmd t \, \rmd \bz \\
                    &= \left(\frac{1+\kappa \delta}{1-\kappa \delta} \right)^{|\beta|} \frac{1}{1-\delta} \vnorm{\grad u}_{L^p(\Omega;\beta)}^p.
                \end{split}
            \end{equation*}

            Now, for general $u \in V^{1,p}(\Omega;\beta)$ (resp. $u \in W^{1,p}(\Omega;\beta)$) use the density result \Cref{thm:Density:V} (resp. \Cref{thm:Density}) to obtain a sequence $\{u_n\}$ of smooth functions converging to $u$ both in norm and almost everywhere. Each $u_n$ satisfies \eqref{eq:Embedding}, hence
            \begin{equation*}
                \liminf_{n \to \infty} [u_n]_{\mathfrak{W}^p[\delta](\Omega;\beta)}^p \leq \left(\frac{1+\kappa \delta}{1-\kappa \delta} \right)^{|\beta|} \frac{1}{1-\delta} \vnorm{\grad u}_{L^p(\Omega;\beta)}^p.
            \end{equation*}
            The estimate \eqref{eq:Embedding} then follows for general $u$ by applying Fatou's lemma to the integrand in the double integral defining $[u_n]_{\mathfrak{W}^p[\delta](\Omega;\beta)}^p$.
        \end{proof}

        \begin{theorem}\label{thm:Density:Nonlocal}
        $C^\infty(\overline{\Omega})$ is dense in $\mathfrak{W}^p[\delta](\Omega;\beta)$ for all $\beta \in (-\infty,d)$. $C^\infty_c(\overline{\Omega} \setminus \Gamma)$ is dense in $\mathfrak{V}^p[\delta](\Omega;\beta)$ for all $\beta \in \bbR$.
        \end{theorem}

        \begin{proof}
            Let $u \in \mathfrak{W}^p[\delta](\Omega;\beta)$. By \Cref{thm:KnownConvResults:Nonlocal} item 3), $\vnorm{K_\veps u - u}_{\mathfrak{W}^p[\delta](\Omega;\beta)} \to 0$ as $\veps \to 0$, where $K_\veps u$ is the boundary-localized convolution defined in \eqref{eq:ConvolutionOperator}.
            Then, the estimate \eqref{eq:Intro:ConvEst:Deriv} shows that $K_\veps u \in W^{1,p}(\Omega;\beta)$ for any $\veps > 0$, since any value of the bulk horizon results in an equivalent norm according to \Cref{thm:InvariantHorizon}. Therefore by \Cref{thm:Density} we can find for each $\veps > 0$ a function $v_\veps \in C^\infty(\overline{\Omega})$ such that $\vnorm{v_\veps - K_\veps u}_{W^{1,p}(\Omega;\beta)} < \frac{\veps}{2}$. Hence by \Cref{thm:Embedding:Nonlocal}
        \begin{equation*}
            \begin{split}
                \vnorm{v_{\veps} - u}_{\mathfrak{W}^p[\delta](\Omega;\beta)} &\leq \vnorm{K_\veps u - u}_{\mathfrak{W}^p[\delta](\Omega;\beta)} + C(\underline{\delta}_0) \vnorm{v_\veps - K_\veps u}_{W^{1,p}(\Omega;\beta)} \\
                &\leq \vnorm{K_\veps u - u}_{\mathfrak{W}^p[\delta](\Omega;\beta)} + \frac{\veps}{2}. 
            \end{split}
        \end{equation*}
        Let $\veps = \veps_j$ be a sequence converging to $0$ as $j \to \infty$; it follows from the first part of the proof that $\{v_\veps\}_\veps$ is the desired sequence.

        The density result for $u \in \mathfrak{V}^p[\delta](\Omega;\beta)$ is similar, with \Cref{thm:Density:V} used in place of \Cref{thm:Density}.
        \end{proof}

\begin{theorem}\label{thm:Trace:Nonlocal}
Let $d-p < \beta < d$. The operator $T_\Gamma : \mathfrak{W}^{p}[\delta](\Omega;\beta) \to \bbR^{|\Gamma|}$ is a bounded linear operator that satisfies $T_\Gamma u = u|_\Gamma$ for all $u$ in $C^\infty(\overline{\Omega})$. 
Moreover, the limit
	\begin{equation*}
		  \begin{gathered}
		  u(\bx_0) := 
		  \begin{cases}
		  \lim\limits_{\veps \to 0} (u)_{B(\bx_0,\veps) \cap \Omega}, &\quad \beta \geq 0, \\
		  \lim\limits_{\veps \to 0} (u)_{B(\bx_0,\veps) \cap \Omega,\beta}, &\quad \beta < 0,
						\end{cases}
					\end{gathered}
				\end{equation*}
				is well-defined for any $\bx_0 \in \Gamma$ and $u \in \mathfrak{W}^{p}[\delta](\Omega;\beta)$, with
				\begin{align}
				\label{eq:LebPtProp:Nonlocal:pos}
                |u(\bx_0) - (u)_{B(\bx_0,\veps) \cap \Omega }|^p &\leq C \veps^{\beta-(d-p)} [u]_{\mathfrak{W}^{p}[\delta](\Omega;\beta)}^p, \qquad \beta \geq 0, 0 < \veps < R,\\
                \label{eq:LebPtProp:Nonlocal:neg}
				|u(\bx_0) - (u)_{B(\bx_0,\veps) \cap \Omega,\beta}|^p &\leq C \veps^{\beta-(d-p)} [u]_{\mathfrak{W}^{p}[\delta](\Omega;\beta)}^p, \qquad \beta < 0, 0 < \veps < R.
				\end{align}
\end{theorem}

\begin{proof}
Since the boundary-localized convolution $K_\delta u $ belongs to $W^{1,p}(\Omega;\beta)$, we have by \Cref{thm:Trace} and \eqref{eq:Intro:ConvEst:Deriv}
	\begin{equation*}
		\Vnorm{T_\Gamma K_\delta u}_{L^\infty(\Gamma)} \leq C \Vnorm{K_\delta u}_{W^{1,p}(\Omega;\beta)} \leq C \Vnorm{u}_{\mathfrak{W}^{p}[\delta](\Omega;\beta)}.
	\end{equation*}
	We now use \Cref{thm:Density:Nonlocal}. Let $\{ u_n \} \subset C^{\infty}(\overline{\Omega})$ be a sequence converging to $u$ in $\mathfrak{W}^{p}[\delta](\Omega;\beta)$. Then since $T_\Gamma K_\delta u_n = T_\Gamma u_n$ for all $n$ by \Cref{thm:KnownConvResults} item 2),
	\begin{equation*}
		\begin{split}
			\Vnorm{ T u_n - T u_m }_{L^{\infty}(\Gamma)} 
			&= \Vnorm{ T_\Gamma K_\delta u_n - T_\Gamma K_\delta u_m }_{L^{\infty}(\Gamma)} \\
			&\leq C \Vnorm{ u_n - u_m }_{\mathfrak{W}^{p}[\delta](\Omega;\beta)}.
		\end{split}
	\end{equation*}
	Therefore the bounded linear operator
 $T_\Gamma : \mathfrak{W}^{p}[\delta](\Omega;\beta) \to L^\infty(\Gamma)$ is well-defined. 

Now we prove the Lebesgue point property. Fix $\bx_0 \in \Gamma$, and denote $B_r = B(\bx_0,r)$ for $r < 0$.
    First we assume that $\beta \geq 0$. By Jensen's inequality
    \begin{equation*}
        \begin{split}
            |(u)_{B_\veps} - (K_\delta u)_{B_\veps}|^p \leq C \veps^{-d} \int_{B_\veps \cap \Omega} \int_{\Omega} \psi_\delta(\bx,\by) |u(\bx)-u(\by)|^p \, \rmd \by \, \rmd \bx.
        \end{split}
    \end{equation*}
    Since $\gamma(\bx)^\beta \leq \veps^\beta$ and $\eta(\bx) \leq \veps$ on $B_\veps$, we get
    \begin{equation*}
        |(u)_{B_\veps} - (K_\delta u)_{B_\veps}|^p \leq C \delta^p \veps^{-d+\beta+p} \int_{B_\veps \cap \Omega} \int_{\Omega} \psi_\delta(\bx,\by) \frac{|u(\bx)-u(\by)|^p}{ \gamma(\bx)^\beta \eta_\delta(\bx)^p  } \, \rmd \by \, \rmd \bx.
    \end{equation*}
    We apply the equivalence of kernel functions in the nonlocal seminorm in \Cref{thm:EnergySpaceIndepOfKernel}, and obtain
    \begin{equation}\label{eq:LebPtProp:Nonlocal:Pf1}
        |(u)_{B_\veps} - (K_\delta u)_{B_\veps}|^p \leq C \delta^p \veps^{\beta-d+p} [u]_{\mathfrak{W}^p[\delta](\Omega;\beta)}^p.
    \end{equation}
    Now, since $K_\delta u$ satisfies \eqref{eq:Intro:ConvEst} and \eqref{eq:Intro:ConvEst:Deriv}, we can apply \Cref{thm:Trace} to get that $K_\delta u(\bx_0) = \lim_{\veps \to 0} (K_\delta u)_{B_\veps}$ exists. This combined with \eqref{eq:LebPtProp:Nonlocal:Pf1} shows that $\lim\limits_{\veps \to 0} (u)_{B_\veps} := u(\bx_0)$ exists, and is equal to $K_\delta u(\bx_0)$. With \Cref{thm:Trace}, \eqref{eq:LebPtProp:Nonlocal:Pf1} and \eqref{eq:Intro:ConvEst:Deriv}, we arrive at \eqref{eq:LebPtProp:Nonlocal:pos}:
    \begin{equation*}
        \begin{split}
            |u(\bx_0)-(u)_{B_\veps}| = |K_\delta u(\bx_0) - (u)_{B_\veps}| &\leq |K_\delta u(\bx_0) - (K_\delta u)_{B_\veps}| + |(K_\delta u)_{B_\veps} - (u)_{B_\veps}| \\
            &\leq C \veps^{\frac{\beta-d+p}{p}} \vnorm{\grad K_\delta u}_{L^p(\Omega;\beta)} + \delta \veps^{\frac{\beta-d+p}{p}} [u]_{\mathfrak{W}^p[\delta](\Omega;\beta)} \\
            &\leq C \veps^{\frac{\beta-d+p}{p}}[u]_{\mathfrak{W}^p[\delta](\Omega;\beta)}.
        \end{split}
    \end{equation*}

    Now assume that $\beta < 0$; the proof is similar.
    By Jensen's inequality
    \begin{equation*}
        \begin{split}
            |(u)_{B_\veps,\beta} - (K_\delta u)_{B_\veps,\beta}|^p \leq C \frac{1}{\int_{B_\veps \cap \Omega} \gamma(\bx)^{-\beta} \, \rmd \bx } \int_{B_\veps \cap \Omega} \int_{\Omega} \psi_\delta(\bx,\by) \frac{|u(\bx)-u(\by)|^p}{\gamma(\bx)^\beta} \, \rmd \by \, \rmd \bx.
        \end{split}
    \end{equation*}
    Since $\eta(\bx) \leq \veps$ on $B_\veps$ and since \eqref{eq:IntConeCond:Conseq} holds, we get
    \begin{equation*}
        |(u)_{B_\veps,\beta} - (K_\delta u)_{B_\veps,\beta}|^p \leq C \delta^p \veps^{\beta-d+p} \int_{B_\veps \cap \Omega} \int_{\Omega} \psi_\delta(\bx,\by) \frac{|u(\bx)-u(\by)|^p}{ \gamma(\bx)^\beta \eta_\delta(\bx)^p  } \, \rmd \by \, \rmd \bx.
    \end{equation*}
    We apply the equivalence of kernel functions in the nonlocal seminorm in \Cref{thm:EnergySpaceIndepOfKernel}, and obtain
    \begin{equation*}
        |(u)_{B_\veps,\beta} - (K_\delta u)_{B_\veps,\beta}|^p \leq C \delta^p \veps^{\beta-d+p} [u]_{\mathfrak{W}^p[\delta](\Omega;\beta)}^p.
    \end{equation*}
    The rest of the proof is similar to the $\beta \geq 0$ case.
\end{proof}

The following corollary is a consequence of the previous proof; compare to \Cref{thm:KnownConvResults} item 5).

\begin{corollary}\label{cor:ConvTraceEqual:Nonlocal}
    For $d - p < \beta < d$, the trace operator $T_\Gamma : \mathfrak{W}^{p}[\delta](\Omega;\beta) \to \bbR^{|\Gamma|}$ satisfies $T_\Gamma K_\delta u = T_\Gamma u$ for all $u \in \mathfrak{W}^{p}[\delta](\Omega;\beta)$.
\end{corollary}

With these facts for traces established, we can now characterize the homogeneous nonlocal space in terms of the trace:

\begin{theorem}\label{thm:TraceChar:Nonlocal}
    For $d-p<\beta<d$, $u \in \mathfrak{W}^{p}[\delta](\Omega;\beta)$ belongs to $\mathfrak{W}^{p}_{0,\Gamma}[\delta](\Omega;\beta)$ if and only if  and $T_\Gamma u = 0$.
\end{theorem}

\begin{proof}
The forward implication is clear from the continuity of the trace, so we need to prove the reverse implication.
Let $u \in \mathfrak{W}^p[\delta](\Omega;\beta)$ with $T_\Gamma u = 0$. 
Then, for each $n \in \bbN$ there exists $\veps_n > 0$ such that the boundary localized convolution $K_{\veps_n} u$ satisfies $\vnorm{K_{\veps_n} u - u}_{\mathfrak{W}^p[\delta](\Omega;\beta)} < \frac{1}{n}$.
Moreover, for each $n$ $K_{\veps_n} u$ satisfies $T_\Gamma K_{\veps_n} u = 0$ by \Cref{cor:ConvTraceEqual:Nonlocal}. Hence $K_\veps u \in W^{1,p}_{0,\Gamma}(\Omega)$ by \eqref{eq:Intro:ConvEst:Deriv}, \Cref{thm:InvariantHorizon}, and \Cref{thm:TraceChar}. Additionally by \Cref{thm:TraceChar}, for each $n \in \bbN$ there exists $v_n \in C^\infty_c(\overline{\Omega} \setminus \Gamma)$ such that $\vnorm{ v_n - K_{\veps_n} u }_{W^{1,p}(\Omega;\beta)} < \frac{1}{n}$. The sequence $\{v_n\}_n$ is the desired sequence; indeed, by \Cref{thm:Embedding:Nonlocal}
\begin{equation*}
    \vnorm{v_n-u}_{\mathfrak{W}^p[\delta](\Omega;\beta)} \leq \vnorm{K_{\veps_n} u - u}_{\mathfrak{W}^p[\delta](\Omega;\beta)} + C \vnorm{v_n-K_{\veps_n} u}_{W^{1,p}(\Omega;\beta)} \leq \frac{1+C}{n}.
\end{equation*}
\end{proof}

The following Hardy-type inequality is essential to our analysis of the variational problem.

\begin{theorem}\label{thm:Hardy:Nonlocal}
    For $d-p < \beta$, there exists a constant $C = C(d,p,\beta,\Omega) > 0$ such that
				\begin{equation*}
					\int_{\Omega} \frac{|u(\bx)|^p}{ \gamma(\bx)^{\beta+p} } \, \rmd \bx \leq C [u]_{ \mathfrak{W}^p[\delta](\Omega;\beta) }^p, \qquad \forall u \in \mathfrak{V}^p[\delta](\Omega;\beta).
				\end{equation*}
\end{theorem}

\begin{proof}
    By \eqref{eq:Intro:ConvEst}, \Cref{thm:Hardy}, and \eqref{eq:Intro:ConvEst:Deriv} all applied to $K_\delta u$,
    \begin{equation*}
        \begin{split}
            \vnorm{u}_{L^p(\Omega;\beta+p)} 
            &\leq \vnorm{u - K_\delta u}_{L^p(\Omega;\beta+p)} + \vnorm{K_\delta u}_{L^p(\Omega;\beta+p)} \\
            &\leq C \delta [u]_{\mathfrak{W}^p[\delta](\Omega;\beta)} + C \vnorm{\grad K_\delta u }_{L^p(\Omega;\beta)} \leq C [u]_{\mathfrak{W}^p[\delta](\Omega;\beta)}. 
        \end{split}
    \end{equation*}
\end{proof}

\begin{theorem}\label{thm:V=W:Nonlocal}
    For $\beta < d - p$, $\mathfrak{W}^p[\delta](\Omega;\beta) = \mathfrak{V}^p[\delta](\Omega;\beta)$. For $d-p < \beta$, $\mathfrak{W}^p_{0,\Gamma}[\delta](\Omega;\beta) = \mathfrak{V}^p[\delta](\Omega;\beta)$.
\end{theorem}

\begin{proof}
    Since $\vnorm{u}_{L^p(\Omega;\beta)} \leq C(p,\Omega) \vnorm{u}_{L^p(\Omega;\beta+p)}$, it suffices to show that
    \begin{align}
        \label{eq:V=W:Nonlocal:Pf1}
				\vnorm{u}_{L^p(\Omega;\beta+p)} &\leq C \vnorm{u}_{\mathfrak{W}^p[\delta](\Omega;\beta)}, \qquad \forall u \in \mathfrak{W}^p[\delta](\Omega;\beta) \text{ when } \beta < d - p, \text{ and } \\
        \label{eq:V=W:Nonlocal:Pf2}
				\vnorm{u}_{L^p(\Omega;\beta+p)} &\leq C \vnorm{u}_{\mathfrak{W}^p[\delta](\Omega;\beta)}, \qquad \forall u \in \mathfrak{W}^p_{0,\Gamma}[\delta](\Omega;\beta) \text{ when } \beta > d - p.
    \end{align}

   First assume that $\beta < d - p$. Then by \eqref{eq:Intro:ConvEst} 
    \begin{equation}\label{eq:V=W:Nonlocal:Pf3}
        \vnorm{u}_{L^p(\Omega;\beta+p)} \leq \vnorm{K_\delta u}_{L^p(\Omega;\beta+p)} + C \delta [u]_{ \mathfrak{W}^p[\delta](\Omega;\beta)}.
    \end{equation}
    In the case $\beta < d-p$, \eqref{eq:ConvEst:Lp} and \eqref{eq:Intro:ConvEst:Deriv} imply that $K_\delta u \in W^{1,p}(\Omega;\beta)$ if $u \in \mathfrak{W}^p[\delta](\Omega;\beta)$. Therefore, we can apply \eqref{eq:V=W:Pf1} to $K_\delta u$, and obtain 
    \begin{equation}\label{eq:V=W:Nonlocal:Pf4}
        \vnorm{u}_{L^p(\Omega;\beta+p)} \leq \vnorm{K_\delta u}_{W^{1,p}(\Omega;\beta)} + C \delta [u]_{ \mathfrak{W}^p[\delta](\Omega;\beta)}.
    \end{equation}
    Then \eqref{eq:V=W:Nonlocal:Pf1} follows from applying \eqref{eq:ConvEst:Lp} and \eqref{eq:Intro:ConvEst:Deriv} in this estimate.
    
    In the case $\beta > d-p$, the same estimate \eqref{eq:V=W:Nonlocal:Pf3} holds for all $u \in \mathfrak{W}^p_{0,\Gamma}[\delta](\Omega;\beta)$. Then \eqref{eq:ConvEst:Lp}, \eqref{eq:Intro:ConvEst:Deriv} and \Cref{cor:ConvTraceEqual:Nonlocal} imply that $K_\delta u \in W^{1,p}_{0,\Gamma}(\Omega;\beta)$. Therefore we can apply \eqref{eq:V=W:Pf2} to $K_\delta u$ to obtain \eqref{eq:V=W:Nonlocal:Pf4}, and \eqref{eq:V=W:Nonlocal:Pf2} follows in the same way.
\end{proof}

Finally, we note the following as a corollary of the results in this section, applied in the following order: \Cref{thm:Trace:Nonlocal}, \Cref{thm:Extension}, \Cref{thm:TraceChar:Nonlocal}, \Cref{thm:V=W:Nonlocal}, and \Cref{thm:Hardy:Nonlocal}:

\begin{theorem}\label{thm:Combined:Nonlocal}
    For $d-p < \beta < d$ and for any $u \in \mathfrak{W}^{p}[\delta](\Omega;\beta)$, the function $\bar{u} := u - E_\Gamma \circ T_\Gamma u$ satisfies $\bar{u} \in \mathfrak{V}^{p}[\delta](\Omega;\beta)$. Moreover, 
    the following Hardy-type inequality holds:
				\begin{equation*}
					\int_{\Omega} \frac{|\bar{u}(\bx)|^p}{\gamma(\bx)^{\beta+p}} \, \rmd \bx \leq C [u]_{ \mathfrak{W}^p[\delta](\Omega;\beta) }^p.
				\end{equation*}
\end{theorem}

\section{The nonlocal variational problem}\label{sec:VarProb}
For the functional
$\cE_\delta(u)$ given in \eqref{eq:Intro:MinProb}, and 
a given general $g : \Gamma \to \bbR$,
        the variational problem
        \begin{equation*}
            \text{Minimize } \cE_\delta(u) \text{ subject to } u = g \text{ on } \Gamma
        \end{equation*}
         is a well-posed problem in the natural energy space $\mathfrak{W}^p[\delta](\Omega;\beta)$ only for certain ranges of $\beta$, due to the function space properties explored in the previous section.
        If $\beta < d - p$, then the function values on $\Gamma$ are not well-defined, and if $\beta \geq d$, then functions necessarily have trace zero on $\Gamma$; see \Cref{thm:V=W:Nonlocal}.

		Assume that $d- p < \beta < d$.
		Given a function $\wt{g} : \Gamma \to \bbR$, define its extension $g := E_\Gamma \wt{g}$ to $\Omega$ as in \Cref{thm:Extension}.
		We seek a function $u \in \mathfrak{W}^{p}[\delta](\Omega;\beta)$ such that
		\begin{equation}\label{eq:MinProb:Case1}
			u = \text{argmin} \left\{ \cE_\delta(v) \, : \, v - g \in \mathfrak{V}^{p}[\delta](\Omega;\beta) \right\}.
		\end{equation}
		
        The following coercivity estimate -- which follows from the Hardy-type inequality in \Cref{thm:Hardy:Nonlocal}, \Cref{thm:EnergySpaceIndepOfKernel}, \Cref{thm:Embedding:Nonlocal}, and \Cref{thm:Extension} -- will be used repeatedly in the next theorems:
        \begin{equation}\label{eq:CoercivityEstimate1}
            \begin{split}
                \vnorm{u - g}_{\mathfrak{V}^{p}[\delta](\Omega;\beta)}
                &\leq C [u-g]_{\mathfrak{W}^{p}[\delta](\Omega;\beta)}  \\
                &\leq C \cE_\delta(u)^{1/p} + C [g]_{ W^{1,p}(\Omega;\beta) } \\
                &\leq C \big( \cE_\delta(u)^{1/p} + \vnorm{\wt{g}}_{L^\infty(\Gamma)} \big),
            \end{split}
             \qquad 
             \begin{gathered}
                 \forall u \in \mathfrak{W}^{p}[\delta](\Omega;\beta) \text{ with } \\
                 u - g \in \mathfrak{V}^{p}[\delta](\Omega;\beta),
             \end{gathered}
        \end{equation}
        where $C$ is independent of $\wt{g}$ and $\delta$. Consequently (again applying \Cref{thm:Embedding:Nonlocal} and \Cref{thm:Extension}),
        \begin{equation}\label{eq:CoercivityEstimate2}
            \begin{split}
                \vnorm{u}_{\mathfrak{W}^p[\delta](\Omega;\beta)}
                &\leq C \vnorm{u - g}_{\mathfrak{V}^p[\delta](\Omega;\beta)} + \vnorm{g}_{\mathfrak{W}^p[\delta](\Omega;\beta)} \\
                &\leq C \big( \cE_\delta(u)^{1/p} + \vnorm{\wt{g}}_{L^\infty(\Gamma)} \big),
            \end{split}
            \qquad
            \begin{gathered}
                 \forall u \in \mathfrak{W}^{p}[\delta](\Omega;\beta) \text{ with } \\
                 u - g \in \mathfrak{V}^{p}[\delta](\Omega;\beta).
             \end{gathered}
        \end{equation}

		\begin{theorem}\label{thm:equi:wellposed}
			There exists a unique solution $u \in \mathfrak{W}^{p}[\delta](\Omega;\beta)$ of \eqref{eq:MinProb:Case1}. Moreover, $T_\Gamma u = \wt{g}$. Equivalently, \Cref{thm:Intro:WellPosed} holds.
		\end{theorem}
		
		\begin{proof}
			We use the direct method of the calculus of variations. Let $\{u_k\}_k \subset \mathfrak{W}^{p}[\delta](\Omega;\beta)$ be a minimizing sequence.
			Applying \eqref{eq:CoercivityEstimate1} to this sequence,
            we see that the sequence $\{ u_k - g \}_k \subset \mathfrak{V}^{p}[\delta](\Omega;\beta)$ is uniformly bounded, and so converges weakly in $\mathfrak{V}^{p}[\delta](\Omega;\beta)$ to a function $v$. 
			Moreover,
			\eqref{eq:CoercivityEstimate2} implies that $\{u_k\}_k$ is uniformly bounded as a sequence in $\mathfrak{W}^{p}[\delta](\Omega;\beta)$, and so converges weakly in $\mathfrak{W}^{p}[\delta](\Omega;\beta)$ to a function $u$.
                Since $\mathfrak{V}^p[\delta](\Omega;\beta) \subset \mathfrak{W}^p[\delta](\Omega;\beta)$, and since weak limits are unique, 
			it follows 
                that $u - v = g$, i.e. $u - g \in \mathfrak{V}^{p}[\delta](\Omega;\beta)$.
			Since $\cE_\delta$ defines a seminorm on $\mathfrak{W}^{p}[\delta](\Omega;\beta)$, it is weakly lower semicontinuous, so $\cE_\delta(u)$ gives a minimum value in the desired function space. The uniqueness follows from the strict convexity of $\cE_\delta$. Moreover, \Cref{thm:V=W:Nonlocal} implies that $u - g \in \mathfrak{V}^{p}[\delta](\Omega;\beta) = \mathfrak{W}^{p}_{0,\Gamma}[\delta](\Omega;\beta)$, and by \Cref{thm:TraceChar:Nonlocal} we have $T_\Gamma(u-g) = 0$, i.e. $T_\Gamma u = \wt{g}$.
		\end{proof}
		
        Finally, we consider the case $\beta \geq d$.
        In this case the only possible function trace on $\Gamma$ is $\wt{g} = 0$, i.e. the only well-posed problem in this setting is the following:
		\begin{equation}\label{eq:MinProb:Case2}
			\text{Find } u \in \mathfrak{W}^{p}_{0,\Gamma}[\delta](\Omega;\beta) \text{ such that } u = \text{argmin} \{ \cE_\delta(v) \}.
		\end{equation}
        The unique minimizer is $u \equiv 0$; clearly this is a solution, and the uniqueness follows from the strict convexity of $\cE_\delta$.

\section{Variational convergence in the localization limit}\label{sec:Conv}

In this section, we show the following:

		\begin{theorem}\label{thm:LocLimit:Dirichlet}
		  Let $d-p < \beta < d$. Let $\wt{g} : \Gamma \to \bbR$, and let $g = E_\Gamma \wt{g}$. Let $\{ u_\delta \in \mathfrak{W}^p[\delta](\Omega;\beta) \}$ be the unique minimizers of \eqref{eq:MinProb:Case2}. Then $\{u_\delta\}_\delta$ converges strongly in $L^p(\Omega;\beta)$ to the function $u \in W^{1,p}(\Omega;\beta)$ that is the unique function satisfying
			\begin{equation}\label{eq:MinProb:Limit}
				\cE_0(u) = \mathrm{argmin} \{ \cE_0(v) \, : \, v - g \in V^{1,p}(\Omega;\beta) \}, \quad \text{ where } \cE_0(v) = \frac{1}{p} \int_{\Omega} \frac{|\grad v(\bx)|^p}{\gamma(\bx)^\beta} \, \rmd \bx. 
			\end{equation}
        That is, \Cref{thm:Intro:LocLimit:Dirichlet} holds.
		\end{theorem}
  
        The following theorem is central in calculating the local limit as the bulk horizon parameter $\delta$ approaches $0$.

\begin{theorem}\label{thm:LocalizationOfEnergy}
    $\lim\limits_{\delta \to 0} \cE_\delta(u) = \cE_0(u)$ for all $u \in W^{1,p}(\Omega;\beta)$.
    Moreover, if a sequence $\{u_\delta\}_\delta$ converges to $u$ in $C^2(\overline{V})$ for any $V \Subset \Omega$ as $\delta \to 0$, then
    \begin{equation*}
    \begin{split}
        \lim\limits_{\delta \to 0} &\int_{V} \int_{V} \rho \left( \frac{|\bx-\by|}{\eta_\delta(\bx)} \right) \frac{ |u_\delta(\bx)-u_\delta(\by)|^p }{ \gamma(\bx)^\beta \eta_\delta(\bx)^{d+p} } 
        \, \rmd \by \, \rmd \bx
        = \int_{V} \frac{ |\grad u(\bx)|^p }{ \gamma(\bx)^\beta } \, \rmd \bx.
    \end{split}
    \end{equation*}
    If $u \in L^p(\Omega;\beta) \setminus W^{1,p}(\Omega;\beta)$, then $\lim\limits_{\delta \to 0}  \cE_\delta(u) = + \infty$.
\end{theorem}

\begin{proof} 
    The proof of the first two statements follow exactly the same steps as \cite[Proposition 4.1, Remarks 4.1 and 4.2]{ponce2004new}.
    To prove the last statement, we show that for any $u \in L^p(\Omega;\beta)$
    \begin{equation}
        \liminf_{\delta \to 0} \cE_\delta(u) < \infty \qquad \Rightarrow \qquad u \in W^{1,p}(\Omega;\beta).
    \end{equation}
    To this end, set $M := \liminf_{\delta \to 0} \cE_\delta(u)$; we may assume by taking a subsequence that $\lim\limits_{\delta \to 0} \cE_{\delta}(u) = M$. 
    By \eqref{eq:ConvEst:Lp}, \eqref{eq:Intro:ConvEst} and \eqref{eq:Intro:ConvEst:Deriv}, we get that the sequence $\{K_\delta u\}_\delta$ is uniformly bounded in $W^{1,p}(\Omega;\beta)$ and that $K_\delta u \to u$ strongly in $L^p(\Omega;\beta+p)$ as $\delta \to 0$. Therefore $K_\delta u$ converges weakly in $W^{1,p}(\Omega;\beta)$ to a function $v$ as $\delta \to 0$. Since $K_\delta u \to u$ strongly in $L^p_{loc}(\Omega)$, it follows from the definition of weak derivative that $\grad v$ is the gradient vector consisting of the weak derivatives of $u$. Thus $u \in W^{1,p}(\Omega;\beta)$.
\end{proof}

\subsection{Compactness}

Next we state and prove a compactness result, which will be instrumental in showing the convergence of minimizers in \Cref{thm:LocLimit:Dirichlet}:

\begin{theorem}\label{thm:Compactness}
    Let $\delta = \{\delta_n\}$ be a sequence converging to $0$. Suppose that $\{u_\delta \}_\delta \subset \mathfrak{V}^{p}[\delta](\Omega;\beta)$ is a sequence such that $\sup_{\delta > 0} \vnorm{u_\delta}_{L^p(\Omega;\beta+p)} = B_1 < \infty$ and $\sup_{\delta > 0} [u_\delta]_{\mathfrak{W}^{p}[\delta](\Omega;\beta)} := B_2 < \infty$. Then $\{ u_\delta \}$ is precompact in $L^p(\Omega;\alpha)$ for any $\alpha < \beta + p$. Moreover, any limit point $u$ satisfies $u \in V^{1,p}(\Omega;\beta)$, with $\vnorm{u}_{L^p(\Omega;\beta+p)} \leq B_1$ and $\vnorm{\grad u}_{L^p(\Omega;\beta)} \leq B_2$.
                
\end{theorem}

\begin{proof}
    Let $\alpha < \beta + p$; It suffices to show that a subsequence of $\{u_\delta\}$ is Cauchy in $L^p(\Omega;\alpha)$. First we use \eqref{eq:Intro:ConvEst} to get
    \begin{equation*}
        \Vnorm{ u_\delta -  K_{\delta} u_\delta }_{L^{p}(\Omega;\alpha)} \leq C(p,\alpha,\beta,\Omega) \vnorm{u_\delta - K_\delta u_\delta}_{L^p(\Omega;\beta+p)} \leq C\delta [u_{\delta}]_{\mathfrak{W}^{p}[\delta](\Omega;\beta)} \leq C B_2 \delta.
    \end{equation*}
    Next, by \eqref{eq:ConvEst:Lp} and \eqref{eq:Intro:ConvEst:Deriv}
    \begin{equation*}
        \Vnorm{ K_{\delta} u_\delta }_{L^{p}(\Omega;\beta+p)} \leq C \Vnorm{ u_\delta }_{L^{p}(\Omega;\beta+p)} \leq C B_1 \text{ and } \vnorm{\grad  K_{\delta} u_\delta }_{L^{p}(\Omega;\beta)} \leq C [ u_\delta ]_{\mathfrak{W}^{p}[\delta](\Omega;\beta)}.
    \end{equation*}
    Therefore the sequence $\{ K_{\delta_n} u_{\delta_n} \}_{n \in \bbN}$ is bounded in $V^{1,p}(\Omega;\beta)$, hence by \Cref{thm:CompactEmbed:ZeroDimCase} is precompact in the strong topology of $L^p(\Omega;\alpha)$ since $\alpha < \beta+p$.
    So for a convergent subsequence $\{ K_{\delta_n} u_{\delta_n} \}_n$ (not relabeled), we have for $n$, $m \in \bbN$
    \begin{equation*}
    \begin{split}
        \Vnorm{ u_{\delta_n} - u_{\delta_m} }_{L^p(\Omega;\alpha)} 
        &\leq \Vnorm{ K_{\delta_n} u_{\delta_n} - u_{\delta_n} }_{L^p(\Omega;\alpha)} + \Vnorm{ K_{\delta_m} u_{\delta_m} - u_{\delta_m} }_{L^p(\Omega;\alpha)} \\
            &\qquad + \Vnorm{ K_{\delta_n} u_{\delta_n} - K_{\delta_m} u_{\delta_m} }_{L^p(\Omega;\alpha)} \\
        &\leq CB_2( \delta_m + \delta_n) + \Vnorm{ K_{\delta_n} u_{\delta_n} - K_{\delta_m} u_{\delta_m} }_{L^p(\Omega;\alpha)} 
    \end{split}
    \end{equation*}
which approaches $0$ as $\min \{m,n\} \to \infty$. Thus $ \{u_{\delta_n} \}_n$ is also convergent.

Now we show that any limit point $u$ belongs to $V^{1,p}(\Omega;\beta)$. First, since a sequence $u_{\delta_n}$ converges to $u$ in $L^p(\Omega;\alpha)$, we use the almost-everywhere convergence of a subsequence (not relabeled) $\{u_{\delta_n}\}_n$ and Fatou's lemma to obtain the $L^p(\Omega;\beta+p)$-finiteness of $u$:
\begin{equation*}
    \vnorm{u}_{L^p(\Omega;\beta+p)} \leq \liminf_{n \to \infty} \vnorm{u_\delta}_{L^p(\Omega;\beta+p)} \leq B_1.
\end{equation*}
Second, we use \eqref{eq:Comp:ConvolutionEstimate}; specifically, we can estimate for any $\veps \in (0,\underline{\delta}_0)$
\begin{equation*}
    \frac{ \overline{C}_{d,p} (d+p) }{ \sigma(\bbS^{d-1}) } \int_{\Omega^\veps} \int_{\Omega^\veps \cap 
    B(\bx,\eta_{\theta(\veps)\delta}(\bx))
    } \frac{ |K_\veps u_{\delta_n}(\bx) - K_\veps u_{\delta_n}(\by)|^p }{ \gamma(\bx)^\beta \eta_{\theta(\veps)\delta}(\bx)^{d+p} } \, \rmd \by \, \rmd \bx \leq \phi(\veps) [u_{\delta_n}]_{ \mathfrak{W}^p[\delta](\Omega;\beta) }^p,
\end{equation*}
where $\Omega^\veps := \{ \bx \in \Omega \, : \, \eta(\bx) > \veps \}$.
    Now for any fixed $\veps > 0$, the sequence $\{ K_\veps u_{\delta_n} \}_n$ converges to $K_{\veps} u$ in $C^2(\overline{\Omega^{ \veps}})$ as $\delta_n \to 0$, since $\Omega^{\veps} \Subset \Omega$.
    Therefore we can use \Cref{thm:LocalizationOfEnergy} (with $\rho$ specified as the normalized characteristic function) when taking $\delta_n \to 0$ in the previous inequality to get
    \begin{equation*}
        \int_{\Omega^{\veps}} \frac{|\grad K_\veps u(\bx)|^p}{\gamma(\bx)^\beta} \, \rmd \bx \leq \phi(\veps) B_2^p.
    \end{equation*}
    This inequality holds uniformly in $\veps$, so the inequality $\vnorm{\grad u}_{L^p(\Omega;\beta)} \leq B_2$ follows by taking $\veps \to 0$.
\end{proof}

\subsection{Gamma-limit}

\begin{proposition}\label{prop:GammaLimit:Dirichlet}
    With all the above assumptions,
    we extend the functional 
$\cE_\delta$ to all of $L^p(\Omega;\beta)$ by setting
\begin{equation}\label{eq:Fxnal:Ext:Dirichlet}
    \overline{\cE}_\delta(u) :=
    \begin{cases}
        \cE_\delta(u), & \text{ for } u \in \mathfrak{W}^{p}[\delta](\Omega;\beta) \text{ with } u - g \in \mathfrak{V}^p[\delta](\Omega;\beta), \\
        + \infty, & \text{ for } u \in L^p(\Omega;\beta) \setminus (\mathfrak{W}^{p}[\delta](\Omega;\beta) \text{ with } u - g \in \mathfrak{V}^p[\delta](\Omega;\beta) ).
    \end{cases}
\end{equation}
    Moreover, we
    define
    \begin{equation}
        \overline{\cE}_0(u) :=
        \begin{cases}
        \cE_0(u), & \text{ for } u \in W^{1,p}(\Omega;\beta) \text{ with } u - g \in V^{1,p}(\Omega;\beta), \\
        + \infty, & \text{ for } u \in L^p(\Omega) \setminus ( W^{1,p}(\Omega;\beta) \text{ with } u - g \in V^{1,p}(\Omega;\beta) ).
    \end{cases}
    \end{equation}
    Then we have
    \begin{equation}
        \overline{\cE}_0(u) = \Gammalim_{\delta \to 0} \overline{\cE}_\delta(u),
    \end{equation}
    where the $\Gamma$-limit is
    with respect to the topology of strong convergence on $L^p(\Omega;\beta)$.
\end{proposition}

\begin{proof}
    We proceed in two steps. First, we prove that
    \begin{equation}\label{eq:GammaLiminf:Dirichlet}
        \overline{\cE}_0(u) \leq \liminf_{\delta \to 0} \overline{\cE}_\delta(u_\delta), 
    \end{equation}
    for any sequence $\{ u_\delta \}_\delta \subset L^p(\Omega;\beta)$ that converges strongly in $L^p(\Omega;\beta)$ to $u$.
    If the right-hand side is $\infty$ then there is nothing to show, so assume that $\liminf_{\delta \to 0} \overline{\cE}_\delta(u_\delta) < \infty$.
    If this is the case, then applying \eqref{eq:CoercivityEstimate1} to the sequence $\{u_\delta\}_\delta$ shows that the sequence $\{ u_\delta - g\}_\delta$ is uniformly bounded in $\mathfrak{V}^p[\delta](\Omega;\beta)$. 
    Thus, $\{u_\delta -g \}_\delta$ is bounded as a sequence in $\mathfrak{V}^p[\delta](\Omega;\beta)$, and so by \Cref{thm:Compactness} there exists $v \in V^{1,p}(\Omega;\beta)$ such that a subsequence $\{u_\delta' - g \}_{\delta'}$ converges to $v$ in $L^p(\Omega;\beta)$. However, $u_\delta \to u$ in $L^p(\Omega;\beta)$, so therefore $v = u - g$. Thus $u - g \in V^{1,p}(\Omega;\beta)$, and \Cref{thm:Extension} implies
    \begin{equation*}
        [u]_{W^{1,p}(\Omega;\beta)} \leq [u-g]_{W^{1,p}(\Omega;\beta)} + C \vnorm{\wt{g}}_{L^\infty(\Gamma)} < \infty.
    \end{equation*}
    Therefore 
    $\overline{\cE}_0(u) < \infty$, and we just need to show that
    \begin{equation}\label{eq:GammaLiminf:Finite:Dirichlet}
        \cE_0(u) \leq \liminf_{\delta \to 0} \cE_\delta(u_\delta).
    \end{equation}
    
    To this end, \Cref{thm:LocalizationOfEnergy} and \eqref{eq:Comp:ConvolutionEstimate}  imply that
    \begin{equation*}
        \begin{split}
            \frac{1}{p} \int_{\Omega^\veps} \frac{|\grad K_\veps u|^p}{\gamma(\bx)^\beta} \, \rmd \bx
            &= 
            \lim\limits_{\delta \to 0} 
            \frac{1}{p} \int_{\Omega^\veps} \int_{\Omega^\veps}  \rho \left( \frac{|\bx-\by|}{ \eta_{\theta(\veps)\delta}(\bx) } \right) \frac{ |K_\veps u_\delta(\bx) - K_\veps u_\delta(\by)|^p }{ \gamma(\bx)^\beta \eta_{\theta(\veps)\delta}(\bx)^{d+p} } \, \rmd \by \, \rmd \bx  \\
            &\leq \phi(\veps) \liminf_{\delta \to 0} \cE_\delta(u_\delta),
        \end{split}
    \end{equation*}
    and so \eqref{eq:GammaLiminf:Finite:Dirichlet} follows by taking $\veps \to 0$.

    Second, we note that the constant sequence $\{u_\delta\}_\delta = u \in L^p(\Omega;\beta)$ serves as a recovery sequence:
    \begin{equation}\label{eq:GammaLimsup:Dirichlet}
        \overline{\cE}_0(u) = \lim\limits_{\delta \to 0} \overline{\cE}_\delta(u).
    \end{equation}
    This follows from 
    \Cref{thm:LocalizationOfEnergy}.
    
    Together \eqref{eq:GammaLiminf:Dirichlet} and \eqref{eq:GammaLimsup:Dirichlet} conclude the proof.    
\end{proof}

\begin{proof}[proof of \Cref{thm:LocLimit:Dirichlet}]
    We first establish the well-posedness of \eqref{eq:MinProb:Limit}. The existence and uniqueness of this problem posed on a slightly different weighted Sobolev space was established for $p = 2$ in \cite{Calder2020Properly}, but we write a complete proof here for general $p$ and for the specific spaces we are using. 
    We proceed in a manner similar to \Cref{thm:equi:wellposed}.
    We first note two coercivity estimates analogous to \eqref{eq:CoercivityEstimate1}-\eqref{eq:CoercivityEstimate2}. 
    By the Hardy inequality in \Cref{thm:Hardy} and \Cref{thm:Extension}
    \begin{equation}\label{eq:CoercivityEstimate1:Local}
        \vnorm{u-g}_{V^{1,p}(\Omega;\beta)} \leq C (\cE_0(u)^{1/p} + \vnorm{\wt{g}}_{L^\infty(\Gamma)}), \quad \forall u \in W^{1,p}(\Omega;\beta) \text{ with } u - g \in V^{1,p}(\Omega;\beta).
    \end{equation}
    Consequently (again applying \Cref{thm:Extension}),
        \begin{equation}\label{eq:CoercivityEstimate2:Local}
            \begin{split}
                \vnorm{u}_{W^{1,p}(\Omega;\beta)}
                &\leq C \vnorm{u - g}_{V^{1,p}(\Omega;\beta)} + \vnorm{g}_{W^{1,p}(\Omega;\beta)} \\
                &\leq C \big( \cE_0(u)^{1/p} + \vnorm{\wt{g}}_{L^\infty(\Gamma)} \big),
            \end{split}
            \qquad
            \begin{gathered}
                 \forall u \in W^{1,p}(\Omega;\beta) \text{ with } \\
                 u - g \in V^{1,p}(\Omega;\beta).
             \end{gathered}
        \end{equation}
    The well-posedness of \eqref{eq:MinProb:Limit} can then be proved exactly the same way as \Cref{thm:equi:wellposed}, using \eqref{eq:CoercivityEstimate1:Local}-\eqref{eq:CoercivityEstimate2:Local} in place of \eqref{eq:CoercivityEstimate1}-\eqref{eq:CoercivityEstimate2}.

    The convergence result follows from the framework described in \cite[Theorem 1.21]{braides2002gamma}. By the $\Gamma$-limit computation in \Cref{prop:GammaLimit:Dirichlet}, it suffices to show that $\{\cE_\delta(u_\delta)\}_\delta$ is equi-coercive in the strong $L^p(\Omega;\beta)$ topology, i.e. that $\{u_\delta\}_\delta$ or equivalently $\{u_\delta - g\}_\delta$ is precompact in the strong $L^p(\Omega;\beta)$ topology. But, this follows by noting that the constant $C$ appearing in \eqref{eq:CoercivityEstimate1} 
    is independent of $\delta$, 
    permitting us to apply the compactness result \Cref{thm:Compactness} to the sequence $\{u_\delta - g\}_\delta$.
        
\end{proof}

        \begin{remark}
            In the case that $\beta \geq d$, the proof of \Cref{thm:LocLimit:Dirichlet} is trivial, since in this case the unique minimizer $u_\delta$ is zero for all $\delta$.
        \end{remark}

		\subsection{The Euler-Lagrange equation}
		\label{sec:EL}

        Here, we provide a formal description of the variational limit result in terms of Euler-Lagrange equations.
        A minimizer of \eqref{eq:MinProb:Limit} is a weak solution of the boundary-value problem
		\begin{equation}\label{eq:BVP:Limit}
			\begin{split}
				-\div(\gamma(\bx)^{-\beta} |\grad u|^{p-2} \grad u) &= 0 \text{ on } \Omega, \\
				u &= g \text{ on } \Gamma, \\
				|\grad u|^{p-2} \frac{\p u}{\p \bsnu} &= 0, \text{ on } \p \Omega \setminus \Gamma.
			\end{split}
		\end{equation}
        To identify a boundary-value problem associated to \eqref{eq:MinProb:Case1}, one can apply a nonlocal Green's identity (see \cite{Scott2023Nonlocala} for the case $\beta = 0$) to a minimizer $u$ of \eqref{eq:MinProb:Case1}; for test functions $v \in C^\infty_c(\overline{\Omega} \setminus \Gamma)$, we get that the first variation of $\cE_\delta$ satisfies
		\begin{equation*}
			\begin{split}
				d \cE_\delta(u)[v] = \int_{\Omega} \cL_{p,\delta} u v \, \rmd \bx + \int_{\p \Omega \setminus \Gamma} \gamma^{-\beta} BF_{p,\delta}(\grad u, \bsnu) v \, \rmd \sigma,
			\end{split}
		\end{equation*}
        where the operator $\cL_\delta$ is defined as
        \begin{equation*}
            \cL_{p,\delta}u(\bx) := \int_{\Omega} \left( \frac{\rho(\frac{|\bx-\by|}{\eta_\delta(\bx)}) }{\eta(\bx)^{d+p} \gamma(\bx)^\beta} + \frac{\rho(\frac{|\bx-\by|}{\eta_\delta(\by)}) }{\eta(\by)^{d+p} \gamma(\by)^\beta} \right) |u(\bx) - u(\by)|^{p-2} (u(\bx) - u(\by)) \, \rmd \by,
        \end{equation*}
        and where the function $BF_{p,\delta}(\grad u, \bsnu)$ is defined as
        \begin{equation*}
            BF_{p,\delta}(\grad u,\bsnu) := \int_{B(0,1)} \ln \left( \frac{1+\delta (\bsnu \cdot \bz)}{1-\delta (\bsnu \cdot \bz)} \right)
            \frac{\rho(|\bz|)}{2 \delta} |\grad u(\bx) \cdot \bz |^{p-2}  (\grad u(\bx) \cdot \bz) \, \rmd \bz,
        \end{equation*}
        which is consistent with the local weighted $p$-Laplacian flux as $\delta \to 0$.
        
        Note that the punctures are excluded in $\p \Omega \setminus \Gamma$, so the normal derivative is well-defined in the boundary integral.
		
		We can then re-state the results of the previous two sections referencing the variational forms.
		
		\begin{proposition}
            Let $d-p < \beta < d$. Then, for a given $g : \Gamma \to \bbR$, there exists a unique weak solution $u_\delta \in \mathfrak{W}^p[\delta](\Omega;\beta)$ to the boundary-value problem
			\begin{equation*}
				\begin{split}
					\cL_{p,\delta} u_\delta &= 0, \quad \text{ on } \Omega, \\
					u_\delta &= g, \quad \text{ on } \Gamma, \\
					BF_{p,\delta}(\grad u,\bsnu) &= 0, \quad \text{ on } \p \Omega \setminus \Gamma.
				\end{split}
			\end{equation*}
			Moreover, the solution sequence $\{u_\delta\}_\delta$ converges strongly in $L^p(\Omega;\beta)$ as $\delta \to 0$ to the unique weak solution $u \in W^{1,p}(\Omega;\beta)$ of \eqref{eq:BVP:Limit}. 
		\end{proposition}

		\section{The case of general dimension}\label{sec:GeneralCodim}
        
		Let $\Omega$ be a bounded Lipschitz domain, and let $\Gamma \subset \p \Omega$ be a closed set. We now consider the case of $\Gamma$ possessing general 
        dimension, denoted by $\ell$. Denote the open d-dimensional unit cube by $Q^d := (0,1)^d$.

        \begin{definition}
            We say that $(\Omega,\Gamma) \in \cA(d,\ell)$ for a fixed $\ell \in \{1,2,\ldots,d-1\}$ if the following holds: There exists an open cover of $\overline{\Omega}$ such that, whenever an element $U$ of the open cover satisfies $U \cap \Gamma \neq \emptyset$, there exists a bi-Lipschitz function $\varphi : U \cap \Omega \to Q$ such that $\varphi(\Gamma) = \overline{Q}^\ell \times \{ 0 \}^{d-\ell} \subset \p Q$.
        \end{definition}

        It follows that if $(\Omega,\Gamma) \in \cA(d,\ell)$ then $\scH^{\ell}(\Gamma) \in (0,\infty)$. Clearly $(\Omega,\p \Omega) \in \cA(d,d-1)$.
        In this section, define $\gamma(\bx)$ to be a generalized distance function for the set $\Gamma$, i.e.
        \begin{equation}\label{assump:HigherDimWeight}
			\begin{aligned}
				i) & \, \text{there exists a constant } \kappa_0 \geq 1 \text{ such that } \\
				&\quad \frac{1}{\kappa_0} \dist(\bx,\Gamma) \leq \gamma(\bx) \leq \kappa_0 \dist(\bx,\Gamma),\; \forall  \bx \in \bbR^d\,;\\
				ii) & \, \gamma \in C^0(\bbR^d) \cap C^{\infty}(\bbR^d \setminus \Gamma) 
                \text{ and for each multi-index } \alpha \in \bbN^d_0\\
				&\quad \exists \kappa_{\alpha} > 0 \text{ such that } |D^\alpha \gamma(\bx)| \leq \kappa_{\alpha} |\dist(\bx,\Gamma)|^{1-|\alpha|}, \;  \forall \bx \in \bbR^d.
			\end{aligned} \tag{\ensuremath{\rmA_{\gamma,\ell}}}
		\end{equation}
In the case $\ell = 0$, the distance to the finite set of points $\Gamma$ is described locally by $|\bx-\bx_0|$. No such straightforward characterization is available for higher-dimensional sets, hence we adopt different, though analogous, assumptions for general $\ell$.

        \subsection{Local weighted Sobolev spaces}\label{subsec:localspaces:HigherCodim}
		
		Define the Sobolev space $W^{1,p}(\Omega;\beta)$ for $\beta < d- \ell$, and the homogeneous spaces $V^{1,p}(\Omega;\beta)$ and $W^{1,p}_{0,\Gamma}(\Omega;\beta)$ for $\beta \in \bbR$, in a way similarly to those in the case $\ell = 0$.
        Then, just as in that setting, we have the following set of results:

        \begin{theorem}\label{thm:ClassicalFxnSpace:GeneralCodim}
            Suppose that $\beta \in \bbR$. Let $\ell \in \{1,\ldots,d-1\}$, and let $(\Omega,\Gamma) \in \cA(d,\ell)$. The following hold:
            \begin{enumerate}[1)]
                \item For all $\beta \in (-\infty,d-\ell)$, the class $C^\infty(\overline{\Omega})$ is dense in $W^{1,p}(\Omega;\beta)$.

                \item Let $d - \ell - p < \beta < d - \ell$. The operator $T_\Gamma : W^{1,p}(\Omega;\beta) \to  W^{ 1 - \frac{d-\ell-\beta}{p}, p }(\Gamma)$ is a bounded linear operator that satisfies $T_\Gamma u = u \big|_\Gamma$ for all $u$ in $C^\infty(\overline{\Omega})$.
                Moreover, define $(d-p-\beta)_+ = \max \{ d-p-\beta,0\}$, and for $\alpha \geq 0$ denote $\alpha$-dimensional Hausdorff measure by $\scH^{\alpha}$.
				Then for $\scH^{(d-p-\beta)_+}$-almost every $\bx \in \Gamma$ 
				\begin{equation*}
					\lim\limits_{\veps \to 0} \fint_{B(\bx,\veps) \cap \Omega} u(\by) \, \rmd \by \text{ is well-defined, } \qquad \text{ for all } 
					u \in W^{1,p}(\Omega;\beta),
				\end{equation*}
				and for $\scH^{\ell}$-almost every $\bx \in \Gamma$
				\begin{equation*}
					\lim\limits_{\veps \to 0} \fint_{B(\bx,\veps) \cap \Omega} |u(\by) - Tu(\bx)|^p \, \rmd \by = 0.
				\end{equation*}
				(Note that $\ell - p < d - p -\beta < \ell$.)

                \item Assume that $d-\ell-p < \beta < d - \ell$. Given a function $g \in W^{ 1 - \frac{d-\ell-\beta}{p}, p }(\Gamma)$, there exists a linear extension operator $E_\Gamma : W^{ 1 - \frac{d-\ell-\beta}{p}, p }(\Gamma) \to W^{1,p}(\Omega;\beta)$ with $E_\Gamma g = g$ on $\Gamma$.

                \item $C^\infty_c(\overline{\Omega} \setminus \Gamma)$ is dense in $V^{1,p}(\Omega;\beta)$.

                \item Hardy's inequality: for $d- \ell - p < \beta$, there exists a constant $C = C(d,p,\beta,\ell,\Omega)$ such that
				\begin{equation*}
					\int_{\Omega} \frac{|u(\bx)|^p }{ \gamma(\bx)^{\beta+p} } \, \rmd \bx \leq C \int_{\Omega} \frac{|\grad u(\bx)|^p }{ \gamma(\bx)^{\beta} } \, \rmd \bx, \qquad \forall u \in V^{1,p}(\Omega;\beta).
				\end{equation*}

                \item For $\beta < d - \ell - p$, $W^{1,p}(\Omega;\beta) = V^{1,p}(\Omega;\beta)$. For $d-\ell-p < \beta$, $W^{1,p}_{0,\Gamma}(\Omega;\beta) = V^{1,p}(\Omega;\beta)$.

                \item For $d - \ell - p < \beta < d-\ell$, $u \in W^{1,p}(\Omega;\beta)$ belongs to $W^{1,p}_{0,\Gamma}(\Omega;\beta)$ if and only if $T_{\Gamma} u = 0$.

                \item For $d-\ell-p < \beta < d-\ell$ and for any $u \in W^{1,p}(\Omega;\beta)$, the function $\bar{u} := u - E_\Gamma \circ T_\Gamma u$ satisfies $\bar{u} \in V^{1,p}(\Omega;\beta)$, and the following Hardy inequality holds:
				\begin{equation*}
					\int_{\Omega} \frac{|u(\bx)-E_\Gamma \circ T_\Gamma u(\bx)|^p}{\gamma(\bx)^{\beta+p}} \, \rmd \bx \leq C(d,p,\beta) \int_{\Omega} \frac{|\grad u(\bx)|^p}{\gamma(\bx)^{\beta}} \, \rmd \bx.
				\end{equation*}

            \end{enumerate}
        \end{theorem}

        \begin{proof}
            For item 1), in the case $(d-\ell)(1-p) < \beta < d - \ell$, the function $\gamma(\bx)$ is a Muckenhoupt $A_p$-weight, and so the result follows from a special case of \cite[Theorem 6.1]{Chua1992Extension}. In the case $\beta < (d-\ell)(1-p)$, 
            the result follows from \cite[Remark 7.5 and Proposition 7.6]{Kufner1980Weighted}, which proves a density result for all $\beta \leq 0$ and for Lipschitz domains satisfying a uniform exterior cone condition on all of $\Gamma$.

           As for Item 2), it is proved for smooth functions in \cite{Nekvinda1993Characterization}, and the result for general Sobolev functions follows from item 1). 
            Item 3) is also established in \cite{Nekvinda1993Characterization}. Meanwhile,
            item 4) is a special case of \cite[Theorem 1]{Rakosnik1989embeddings}.

            For item 5), by using localization arguments and transforming to the flat case, as well as the conclusion of Item 4), it suffices to show that
            \begin{equation*}
                \int_{Q^\ell} \int_{Q^{d-\ell}} \frac{|u(\bx',\bx'')|^p}{ |\bx''|^{\beta+p} } \, \rmd \bx'' \, \rmd \bx' \leq C \int_{Q^\ell} \int_{Q^{d-\ell}} \frac{|\grad u(\bx',\bx'')|^p}{ |\bx''|^{\beta} } \, \rmd \bx'' \, \rmd \bx',
            \end{equation*}
            for all $u \in C^1_c( [0,1]^{\ell} \times (0,1]^{d-\ell})$.
            Converting to spherical coordinates $\bx'' = r \bsomega''$ and writing $v(\bx',r,\bsomega'') = u(\bx',r\bsomega'')$, we apply the one-dimensional Hardy inequality \eqref{eq:Hardy1D:0} to the function $r \mapsto v(\bx',r,\bsomega'')$; in the context of \eqref{eq:Hardy1D:0} we use $d-\ell$ in place of $d$. The result then follows from a procedure similar to the proof of \Cref{thm:Hardy}.
            
         To complete the proof, we note that item 6) is a special case of \cite[Corollary 3.2]{Edmunds1985Embeddings};
item 7) can be proved similarly to \Cref{thm:TraceChar}; and
item 8) follows from the previously-proved items. 
        \end{proof}

		The next two theorems establish compactness for a general $\Gamma$; the statements are independent of the dimension $\ell$ of the set $\Gamma$. The results are inspired by \cite{Gurka1988Continuous}, but the exact results we need are not stated there, so we provide the proof.
		
		\begin{theorem}
			Let $q \in [1,\frac{dp}{d-p}]$, and let $\alpha \in \bbR$. Let $\Omega^* \subset \bbR^d$ be a bounded Lipschitz domain. Let the two sets $\Omega \subset \Omega^*$ and $\Gamma \subset \overline{\Omega}^*$, and the function $\gamma$, be defined as in one of the following two scenarios:
			\begin{enumerate}
				\item[1)] $\Gamma \subset \overline{\Omega}^*$ is a finite set of points, with $\Omega = \Omega^* \setminus \Gamma$, and $\gamma$ satisfies \eqref{assump:weight}.
				\item[2)] $\Omega = \Omega^*$ and $\Gamma \subset \p \Omega$ is a closed set, and $(\Omega,\Gamma) \in \cA(d,\ell)$ for some $\ell \in \{1,2,\ldots,d-1\}$, and $\gamma$ satisfies \eqref{assump:HigherDimWeight}.
			\end{enumerate}
			Then the space $V^{1,p}(\Omega;\beta)$ is continuously embedded in $L^q(\Omega;\alpha)$ if $\alpha$ and $\beta$ satisfy $d \left( \frac{1}{q} - \frac{1}{p} \right) - \frac{\alpha}{q} + \frac{\beta}{p} + 1 \geq 0$.
		\end{theorem}
		
		\begin{proof}
			Set $\tau := d \left( \frac{1}{q} - \frac{1}{p} \right) - \frac{\alpha}{q} + \frac{\beta}{p} + 1$.
			For $n \in \bbN$, define 
			\begin{equation*}
	   			\Omega_n := \{ \bx \in \Omega \, : \, \gamma(\bx) < 1/n \}, \qquad \Omega^n := \{ \bx \in \Omega \, : \, \gamma(\bx) > 1/n \}\,;
			\end{equation*}
			note that in both scenarios presented in the theorem,
            the set $\Omega^n$ satisfies the uniform interior-exterior cone condition, and thus is still a Lipschitz domain, for any $n \geq n_1$, where $n_1 \in \bbN$ is sufficiently small and depending only on $\Omega$. For each such $n$ we can write
			\begin{equation*}
				\vnorm{u}_{L^q(\Omega;\alpha)}^q = \vnorm{u}_{L^q(\Omega^n;\alpha)}^q + \vnorm{u}_{L^q(\Omega_n;\alpha)}^q.
			\end{equation*}
			First, by the classical Sobolev embedding theorem, for each $n$ there exists a constant $C_n$ depending only on $d$, $p$, $q$, $\alpha$, $\beta$, $\Omega$ and $n$ such that
			\begin{equation}\label{eq:Compactness:Pf0}
				\vnorm{u}_{L^q(\Omega^n;\alpha)} \leq C_n \vnorm{u}_{V^{1,p}(\Omega^n;\beta)},
			\end{equation}
		  since the weights are bounded on $\Omega^n$. 
		  Next, fix $n_0 \in \bbN$ with $n_0 > \max \{ n_1, 1/\text{diam}(\Omega) \}$; 
            we will show that there exists a constant $C > 0$ such that
            \begin{equation}\label{eq:Compactness:Pf1}
                \vnorm{u}_{L^q(\Omega_n;\alpha)} \leq \frac{C}{n^{\tau}} \vnorm{u}_{V^{1,p}(\Omega;\beta)} \qquad \forall n \geq n_0 \text{ and } \forall u \in V^{1,p}(\Omega;\beta).
            \end{equation}

            To this end, the Besikovitch covering lemma implies that there exists a countable collection $\{\bx_k\}_{k=1}^\infty \subset \Omega_{n_0}$ such that the Euclidean balls $ B_k := B(\bx_k,\frac{1}{30}\gamma(\bx_k))$ satisfy $\Omega_{n_0} \subset \sup_{k=1}^\infty B_k$ and $\sum_{k=1}^\infty \mathds{1}_{B_k}(\bx) \leq N$ for all $\bx \in \Omega_{n_0}$, where $N$ is a number depending only on $d$. Then for each $k$, $\gamma(\bx) \approx \gamma(\bx_k)$ for all $\bx \in B_k$ with constants of comparison independent of $k$, and along with the Sobolev embedding theorem on $B_k$ we have
			\begin{equation*}
				\begin{split}
					\vnorm{u}_{L^q(B_k;\alpha)} &\leq \gamma(\bx_k)^{-\alpha/q} \vnorm{u}_{L^q(B_k)} \\
					&\leq C \gamma(\bx_k)^{-\alpha/q+d/q-d/p+1} \left( \frac{1}{\gamma(\bx_k)^p} \int_{B_k} |u(\bx)|^p \, \rmd \bx + \int_{B_k} |\grad u(\bx)|^p \, \rmd \bx  \right)^{1/p} \\
					&\leq C \gamma(\bx_k)^{-\alpha/q+d/q-d/p+1+\beta/p} \left( \int_{B_k} \frac{ |u(\bx)|^p }{\gamma(\bx)^{\beta+p} } \, \rmd \bx + \int_{B_k} \frac{ |\grad u(\bx)|^p }{ \gamma(\bx)^\beta }\, \rmd \bx  \right)^{1/p},
				\end{split}
			\end{equation*}
			where $C > 0$ depends only on $d$, $p$, $q$, $\alpha$ and $\beta$. Defining $I_n := \{k \in \bbN \, : \, \Omega_n \cap B_k \neq \emptyset \}$, we therefore have
			\begin{equation*}
				\begin{split}
					\vnorm{u}_{L^q(\Omega_n;\alpha)} &\leq \sum_{k \in I_n} \vnorm{u}_{L^q(B_k;\alpha)}
					\leq \sum_{k \in I_n} C \frac{1}{n^\tau} \vnorm{u}_{V^{1,p}(B_k;\beta)} 
					\leq C N^{1/p} \frac{1}{n^\tau} \vnorm{u}_{V^{1,p}(\Omega;\beta)},
				\end{split}
			\end{equation*}
			and \eqref{eq:Compactness:Pf1}
            is proved. Together \eqref{eq:Compactness:Pf0}-\eqref{eq:Compactness:Pf1} establish the result, since $\tau \geq 0$.
		\end{proof}
		
		\begin{theorem}
			In the setting of the previous theorem, the space $V^{1,p}(\Omega;\beta)$ is compactly embedded in $L^q(\Omega;\alpha)$ if $\alpha$ and $\beta$ satisfy $d \left( \frac{1}{q} - \frac{1}{p} \right) - \frac{\alpha}{q} + \frac{\beta}{p} + 1 > 0$.
		\end{theorem}
		
		\begin{proof}
		  We adopt the notation of the previous proof; we note that $\tau > 0$ by assumption.
            Let $\veps > 0$. Then by \eqref{eq:Compactness:Pf1}, there exists $n_2 \in \bbN$ such that
			\begin{equation}\label{eq:Compactness:Pf2}
				\vnorm{v}_{L^q(\Omega_n;\alpha)} \leq \veps \vnorm{v}_{V^{1,p}(\Omega;\beta)} \qquad \forall v \in V^{1,p}(\Omega;\beta), \: \forall n \geq n_2.
			\end{equation}
			Let $\{ u_j \}_{j=1}^\infty \subset V^{1,p}(\Omega;\beta)$ be a bounded sequence, with $M := \sup_{j \in \bbN } \vnorm{u_j}_{V^{1,p}(\Omega;\beta)}$. Fix $n \geq n_2$; the weights in the norm are nonsingular on $\Omega^{n}$, and since $q < \frac{dp}{d-p}$ the classical Sobolev embedding theorem implies the existence of a subsequence (not relabeled) of $\{u_j\}$ convergent in $L^q(\Omega^n;\alpha)$.
			Therefore,
			\begin{equation*}
                \begin{split}
				\vnorm{ u_j - u_k }_{L^q(\Omega;\alpha)} 
                &\leq \vnorm{ u_j - u_k }_{L^q(\Omega^n;\alpha)} + \vnorm{ u_j }_{L^q(\Omega_n;\alpha)} + \vnorm{ u_k }_{L^q(\Omega_n;\alpha)} \\
                &\leq \vnorm{ u_j - u_k }_{L^q(\Omega^n;\alpha)} + 2 M \veps,
                \end{split}
			\end{equation*}
			hence 
            \begin{equation*}
                \limsup\limits_{ \min{j,k} \to \infty } \vnorm{ u_j - u_k }_{L^q(\Omega;\alpha)} \leq 2M \veps.
            \end{equation*}
            Since $\veps > 0$ is arbitrary, the subsequence converges in $L^q(\Omega;\alpha)$.
		\end{proof}

        \begin{remark}
			The assumptions of the previous two theorems are actually necessary conditions for continuous (compact) embeddings. To see this, one can essentially run this argument in reverse to show that the embedding actually implies inequality \eqref{eq:Compactness:Pf1} (inequality \eqref{eq:Compactness:Pf2}), which in turn implies the conditions on $\alpha$, $\beta$, $p$ and $q$.
		\end{remark}
		
		\subsection{Nonlocal weighted function spaces}

        \begin{theorem}\label{thm:KnownConvResults:HighereCodim}
            The results of \Cref{thm:KnownConvResults} and \Cref{thm:KnownConvResults:Nonlocal} remain true under the assumptions $(\Omega;\Gamma) \in \cA(d,\ell)$, \eqref{assump:HigherDimWeight}, with the spaces $L^p(\Omega;\beta)$, $W^{1,p}(\Omega;\beta)$, $V^{1,p}(\Omega;\beta)$, $\mathfrak{W}^p[\delta](\Omega;\beta)$, and $\mathfrak{V}^p[\delta](\Omega;\beta)$ defined accordingly.
        \end{theorem}

        The proof follows exactly the same lines of reasoning as those of \Cref{thm:KnownConvResults} and \Cref{thm:KnownConvResults:Nonlocal}; the derivative estimates of \eqref{assump:HigherDimWeight} are used analogously to those of \eqref{assump:weight}. 
        For instance, note that \eqref{assump:Localization} holds for the Lipschitz domain $\Omega$.

        The following properties of the weighted nonlocal function spaces can also be established. The main tools used in the proofs -- omitted due to their similarity to the proofs in \Cref{subsec:NonlocalSpaces} -- 
          are the results for the boundary-localized convolutions described in \Cref{thm:KnownConvResults:HighereCodim}, and the results in \Cref{subsec:localspaces:HigherCodim} for the local spaces:

    \begin{theorem}[Properties of the nonlocal function spaces]\label{thm:FxnSpaceProp}
			Under the assumptions $(\Omega;\Gamma) \in \cA(d,\ell)$, $p \in (1,\infty)$, $\beta \in \bbR$, \eqref{assump:Horizon}, and \eqref{assump:HigherDimWeight}, the following hold:
			\begin{enumerate}[1)]

                \item \Cref{thm:InvariantHorizon}, \Cref{thm:EnergySpaceIndepOfKernel}, and \Cref{thm:Embedding:Nonlocal} remain true, with the same inequalities.

				\item $C^\infty(\overline{\Omega})$ is dense in $\mathfrak{W}^p[\delta](\Omega;\beta)$ for all $\beta \in (-\infty,d-\ell)$. $C^\infty_c(\overline{\Omega} \setminus \Gamma)$ is dense in $\mathfrak{V}^p[\delta](\Omega;\beta)$ for all $\beta \in \bbR$.

				\item Let $d-\ell -p < \beta < d-\ell$. The operator $T_\Gamma : \mathfrak{W}^{p}[\delta](\Omega;\beta) \to W^{1-\frac{d-\ell-\beta}{p},p}(\Gamma)$ is a bounded linear operator that satisfies $T_\Gamma u = u|_\Gamma$ for all $u$ in $C^\infty(\overline{\Omega})$.

				\item For $d-\ell-p<\beta<d-\ell$, $u \in \mathfrak{W}^{p}[\delta](\Omega;\beta)$ belongs to $\mathfrak{W}^{p}_{0,\Gamma}[\delta](\Omega;\beta)$ if and only if  and $T_\Gamma u = 0$.

				\item For $d-\ell-p < \beta$, there exists a constant $C = C(d,p,\beta,\Omega) > 0$ such that
				\begin{equation*}
					\int_{\Omega} \frac{|u(\bx)|^p}{ \gamma(\bx)^{\beta+p} } \, \rmd \bx \leq C [u]_{ \mathfrak{W}^p[\delta](\Omega;\beta) }^p, \qquad \forall u \in \mathfrak{V}^p[\delta](\Omega;\beta).
				\end{equation*}
				
				\item For $\beta < d -\ell-p$, $\mathfrak{W}^p[\delta](\Omega;\beta) = \mathfrak{V}^p[\delta](\Omega;\beta)$. For $d-\ell-p < \beta$, $\mathfrak{W}^p_{0,\Gamma}[\delta](\Omega;\beta) = \mathfrak{V}^p[\delta](\Omega;\beta)$.

				\item For $d-\ell-p < \beta < d-\ell$ and for any $u \in \mathfrak{W}^{p}[\delta](\Omega;\beta)$, the function $\bar{u} := u - E_\Gamma \circ T_\Gamma u$ satisfies $\bar{u} \in \mathfrak{V}^{p}[\delta](\Omega;\beta)$. Moreover, the following Hardy inequality holds:
				\begin{equation*}
					\int_{\Omega} \frac{|\bar{u}(\bx)|^p}{\gamma(\bx)^{\beta+p}} \, \rmd \bx \leq C [u]_{ \mathfrak{W}^p[\delta](\Omega;\beta) }^p.
				\end{equation*}

				\item Let $\delta = \{\delta_n\}$ be a sequence converging to $0$. Suppose that $\{u_\delta \}_\delta \subset \mathfrak{V}^{p}[\delta](\Omega;\beta)$ is a sequence such that $\sup_{\delta > 0} \vnorm{u_\delta}_{L^p(\Omega;\beta+p)} = B_1 < \infty$
				and $\sup_{\delta > 0} [u_\delta]_{\mathfrak{W}^{p}[\delta](\Omega;\beta)} := B_2 < \infty$. Then $\{ u_\delta \}$ is precompact in $L^p(\Omega;\alpha)$ for any $\alpha < \beta + p$. Moreover, any limit point $u$ satisfies $u \in V^{1,p}(\Omega;\beta)$, with $\vnorm{u}_{L^p(\Omega;\beta+p)} \leq B_1$ and $[u]_{W^{1,p}(\Omega;\beta)} \leq B_2$.

			\end{enumerate}
		\end{theorem}

        \subsection{The variational problem}
        Although it is not immediately obvious, one can repeat the program of the previous sections for the case $(\Omega;\Gamma) \in \cA(d,\ell)$, with the function spaces like $\mathfrak{W}^p[\delta](\Omega;\beta)$, the energy $\cE_\delta(u)$, etc. all defined in a similar way; the proofs are not significantly different.
        That is, the results for the boundary-localized convolutions still hold, which allows all the analogous theorems for the nonlocal function spaces to be obtained. Then one can establish the well-posedness of the mixed boundary-value problems, as well as their variational convergence to the local problem. To illustrate the main difference, which is the range of $\beta$, we state the following theorem:

        \begin{theorem}\label{thm:LocLimit:Dirichlet:GeneralCodim}
        Under assumptions \eqref{assump:HigherDimWeight}, 
        \eqref{assump:VarProb:Kernel}, and \eqref{assump:Horizon},
		let $d-\ell - p < \beta < d-\ell$. Let $\wt{g} \in W^{1-\frac{d-\ell-\beta}{p},p}(\Gamma)$, and let $g = E_\Gamma \wt{g}$ be its $W^{1,p}(\Omega;\beta)$-extension given as in \Cref{thm:ClassicalFxnSpace:GeneralCodim} item 3). Then the problem  \eqref{eq:MinProb:Case1} is well-posed; let $\{ u_\delta \in \mathfrak{W}^p[\delta](\Omega;\beta) \}$ be the unique minimizers corresponding to each $\delta$. Moreover, $\{u_\delta\}_\delta$ converges strongly in $L^p(\Omega;\beta)$ to the function $u \in W^{1,p}(\Omega;\beta)$ that is the unique function that satisfies \eqref{eq:MinProb:Limit}.
		\end{theorem}

\section{Discrete-to-continuum convergence of the graph energy}\label{sec:Graph}

In this section we prove \Cref{thm:GraphToNonlocalConv}.

Recalling the definition of transportation map in the introduction, we note that the assumption on $c_n$ in the \Cref{thm:GraphToNonlocalConv}
is feasible, as demonstrated by the asymptotics of $\vnorm{T_n-Id}_{L^\infty(\Omega)}$ presented in the following theorem. 

\begin{theorem}[\cite{Trillos2015rate}]\label{thm:TransportMapConvergence}
With the above setup, let $\mu_n$ be the empirical measure on $\cX_n$. Then there exists a constant $C > 0$ such that almost surely there exists a sequence of transportation maps $T_n : \Omega^* \to \Omega^*$ from $\mu$ to the empirical measure $\mu_n$ on $\cX_n$ such that
    \begin{equation*}
        \frac{\ell_n}{C} \leq 
        \esssup_{\bx \in \Omega}  |T_n(\bx) - \bx| \leq C \ell_n
    \end{equation*}
for all $n \in \bbN$, where
\begin{equation*}
    \ell_n = \frac{\ln(n)^{3/4}}{n^{1/2}} \text{ when } d = 2, \qquad \text{ while } \qquad 
    \ell_n = \left( \frac{\ln(n)}{n} \right)^{1/d} \text{ when } d \geq 3.
\end{equation*}
\end{theorem}

\begin{remark}
    Although in the definitions of 
    $\eta_\delta^\tau$ and $\gamma^\tau$ we chose to truncate the weight $\gamma$ and localization function $\eta$ by redefining the bandwidth at a critical value $\tau$, other policies to avoid degeneracy can be used. For instance, \Cref{thm:GraphToNonlocalConv} and all the results in this section remain true if
    \begin{equation*}
        \wt{\eta}_\delta^\tau(\bx) = \delta ( \eta(\bx) + \tau), \qquad
    \wt{\gamma}^\tau(\bx) = \gamma(\bx) +  \tau,
    \end{equation*}
    are used in place of $\eta_\delta^\tau$ and $\gamma^\tau$.
\end{remark}

\subsection{Nonlocal and nonlocal truncated energies}

In order to compare the discrete and nonlocal energies, we introduce an intermediate, nonlocal truncated energy
\begin{equation*}
    \cE_{\delta,\tau}(u) := \int_{\Omega} \int_{\Omega} \rho \left( \frac{|\bx-\by|}{ \eta_\delta^\tau(\bx) } \right) \frac{|u(\bx)-u(\by)|^p}{ (\gamma^\tau(\bx))^{\max\{\beta,0\}} (\gamma(\bx))^{\min\{\beta,0\}}  (\eta_\delta^\tau(\bx))^{d+p} } \, \rmd \by \, \rmd \bx.
\end{equation*}
This energy is defined for general $\beta \in \bbR$, but we will only consider the case $\beta \geq 0$.

\begin{theorem}\label{thm:NonlocalTruncToNonlocal}
    Let $d-p < \beta < d$ with $\beta \geq 0$.
    Let $\delta \in (0,\underline{\delta}_0)$, and let $\{ \delta_n \}_n \subset (0,\underline{\delta}_0)$ be a sequence converging to $\delta$. Let $\{ \tau_n \} \subset (0,\diam(\Omega)/4)$ be a sequence converging to $0$. Then there exists $n_0 \in \bbN$ depending on $\delta$ and $R$ such that for all $n \geq n_0$
    \begin{equation}\label{eq:TruncNonlocal:Finite}
        \cE_{\delta_n,\tau_n}(u) \leq \frac {C(d,p,\beta,\rho,\kappa,\Omega)}{ \delta_n^p }
        \vnorm{u}_{\mathfrak{W}^p[\delta](\Omega;\beta)}^p,
        \qquad \forall u \in \mathfrak{W}^p[\delta](\Omega;\beta),
    \end{equation}
    and moreover
    \begin{equation}\label{eq:TruncNonlocal:Limit}
        \lim\limits_{n \to \infty} \cE_{\delta_n,\tau_n}(u) = \cE_\delta(u), \qquad \forall u \in \mathfrak{W}^p[\delta](\Omega;\beta).
    \end{equation}
\end{theorem}

We will prove this result at the end of the section after establishing some preliminary lemmas. In the first lemma, we estimate a ``remainder'' term using a local weighted Sobolev norm of a function extended to all of $\bbR^d$.
This in turn requires some new notation. We first note that because the weight $\gamma$ was constructed via cutoff and mollification, we can regard it as a function belonging to $C^\infty_c(\bbR^d \setminus \Gamma) \cap C^0_c(\bbR^d)$. Next we define the Banach space
\begin{equation*}
    \begin{split}
        &W^{1,p}(\bbR^d,\gamma) \\
        :=& \left\{ u : \bbR^d \to \bbR \, : \, \vnorm{u}_{W^{1,p}(\bbR^d,\gamma)}^p := \int_{\bbR^d} \frac{|u(\bx)|^p}{ \gamma(\bx)^\beta } \, \rmd \bx + \int_{\bbR^d} \frac{|\grad u(\bx)|^p}{ \gamma(\bx)^\beta } \, \rmd \bx < \infty \right\},
    \end{split}
\end{equation*}
which serves as an extension space for functions in $W^{1,p}(\Omega;\beta)$.
We restate an existing extension result for weighted Sobolev spaces, adapted for our purposes:

\begin{theorem}[\cite{Chua1992Extension}, Theorem 1.1]\label{thm:ApExtension}
    Let $d(1-p) < \beta < d$. Then there exists a bounded linear extension operator $\bar{E} : W^{1,p}(\Omega;\beta) \to W^{1,p}(\bbR^d,\gamma)$, i.e. $\bar{E}v |_{\Omega} = v$ and $\vnorm{\bar{E}v}_{W^{1,p}(\bbR^d,\gamma)} \leq C(d,p,\beta,\Omega) \vnorm{v}_{W^{1,p}(\Omega;\beta)}$
    for all $v \in W^{1,p}(\Omega;\beta)$.
\end{theorem}

The assumption $d(1-p)<\beta<d$ ensures that $\gamma$ is a Muckenhoupt $A_p$-weight; see \Cref{lma:ApWeight}. We are ready to state the lemma in full.

\begin{lemma}\label{lma:NonlocalTrunc:Remainder1}
    Let $d(1-p) < \beta < d$, and let $v \in W^{1,p}(\Omega;\beta)$.
    Then
    \begin{equation*}
        \begin{split}
            &\int_{\Omega} \int_{\Omega} \mathds{1}_{ \{ \eta(\bx) < \tau \} } \mathds{1}_{ \{ \gamma(\bx) \geq \tau \} } \rho \left( \frac{|\bx-\by|}{ \delta \tau } \right) \frac{|v(\bx)-v(\by)|^p}{ (\gamma(\bx))^\beta (\delta \tau)^{d+p} } \, \rmd \by \, \rmd \bx \\
            \leq & C(\beta,\kappa) \int_{ \{ \bx \, : \, \eta(\bx) < (1+\kappa_1 \delta) \tau \} \cap \{ \bx \, : \, \gamma(\bx) > (1-\kappa_1 \delta) \tau \} }  \frac{|\grad \bar{E}v(\bx)|^p}{\gamma(\bx)^\beta} \, \rmd \bx,
        \end{split}
    \end{equation*}
    for any $\delta \in (0,\underline{\delta}_0)$ and any $\tau \in (0,\diam(\Omega)/4)$.
\end{lemma}

\begin{proof}
The argument is similar to that of \Cref{thm:Embedding:Nonlocal}. First, we prove that 
    \begin{equation}\label{eq:NonlocalTrunc:Remainder1:Pf1}
        \begin{split}
            &\int_{\bbR^d} \int_{\bbR^d} \mathds{1}_{ \{ \eta(\bx) < \tau \} } \mathds{1}_{ \{ \gamma(\bx) \geq \tau \} } \rho \left( \frac{|\bx-\by|}{ \delta \tau } \right) \frac{|u(\bx)-u(\by)|^p}{ (\gamma(\bx))^\beta (\delta \tau)^{d+p} } \, \rmd \by \, \rmd \bx \\
            \leq & C(\beta,\kappa) \int_{ \{ \bx \, : \, \eta(\bx) < (1+\kappa_1 \delta) \tau \} \cap \{ \bx \, : \, \gamma(\bx) > (1-\kappa_1 \delta) \tau \} }  \frac{|\grad u(\bx)|^p}{\gamma(\bx)^\beta} \, \rmd \bx,
        \end{split}
    \end{equation}
for any $u \in C^\infty_c(\bbR^d)$. By a change of variables, and applying the fundamental theorem of calculus, 
    \begin{equation*}
        \begin{split}
            &\int_{\bbR^d} \int_{\bbR^d} \mathds{1}_{ \{ \eta(\bx) < \tau \} } \mathds{1}_{ \{ \gamma(\bx) \geq \tau \} } \rho \left( \frac{|\bx-\by|}{ \delta \tau } \right) \frac{|u(\bx)-u(\by)|^p}{ (\gamma(\bx))^\beta (\delta \tau)^{d+p} } \, \rmd \by \, \rmd \bx \\
            \leq& \int_{\bbR^d} \int_{B(0,1)} \mathds{1}_{ \{ \eta(\bx) < \tau \} } \mathds{1}_{ \{ \gamma(\bx) \geq \tau \} } \rho \left( |\bz| \right) \int_0^1 \frac{|\grad u(\bx+t\delta \tau \bz) \cdot \bz|^p}{ (\gamma(\bx))^\beta } \, \rmd t \, \rmd \bz \, \rmd \bx.
        \end{split}
    \end{equation*}
Next, since in the integrand $\gamma(\bx) \geq \tau$, it follows that $|\gamma(\bx+t\tau \delta \bz) -  \gamma(\bx)| \leq \kappa_1 \delta \tau \leq \kappa_1 \delta \gamma(\bx)$. We also have $\eta(\bx+t\tau\delta \bz) \leq \eta(\bx) + \kappa_1 \delta \tau \leq (1+\kappa_1 \delta) \tau$. Therefore along with \eqref{assump:VarProb:Kernel}
    \begin{equation*}
        \begin{split}
            &\int_{\bbR^d} \int_{\bbR^d} \mathds{1}_{ \{ \eta(\bx) < \tau \} } \mathds{1}_{ \{ \gamma(\bx) \geq \tau \} } \rho \left( \frac{|\bx-\by|}{ \delta \tau } \right) \frac{|u(\bx)-u(\by)|^p}{ (\gamma(\bx))^\beta (\delta \tau)^{d+p} } \, \rmd \by \, \rmd \bx \\
            \leq& \left( \frac{1+\kappa_1 \delta}{1-\kappa_1 \delta} \right)^{|\beta|} \int_{B(0,1)} \int_0^1  \rho \left( |\bz| \right) \int_{\bbR^d} \mathds{1}_{ \{ \eta(\bx+t\delta \tau \bz) < (1+\kappa_1 \delta) \tau \} }  \mathds{1}_{ \{ \gamma(\bx+t\tau\delta\bz) \geq (1-\kappa_1 \delta) \tau \} }  \\
                &\qquad \qquad \frac{|\grad u(\bx+t\delta \tau \bz) \cdot \bz|^p}{ (\gamma(\bx+t\delta \tau \bz))^\beta } \, \rmd \bx \, \rmd t \, \rmd \bz \\
            =&\left( \frac{1+\kappa_1 \delta}{1-\kappa_1 \delta} \right)^{|\beta|} \int_{ \{ \bx \, : \, \eta(\bx) < (1+\kappa_1 \delta) \tau \} \cap \{ \bx \, : \, \gamma(\bx) \geq (1-\kappa_1 \delta) \tau \} } \frac{ |\grad u(\bx)|^p }{ \gamma(\bx)^\beta } \, \rmd \bx.
        \end{split}
    \end{equation*}

Now we complete the proof for a general $v \in W^{1,p}(\Omega;\beta)$. Recall that $\gamma$ is a Muckenhoupt $A_p$-weight by assumption and by \Cref{lma:ApWeight}.
Thus, as a special case of a density result for $A_p$-weighted Sobolev spaces of functions defined on Jones domains (see \cite[Theorem 6.1 and Remark 6.4]{Chua1992Extension}), there exists a sequence $\{ v_m \} \subset C^{\infty}_c(\bbR^d)$ such that $\vnorm{ v_m - \bar{E}v }_{W^{1,p}(\bbR^d,\gamma)} \to 0$ as $m \to \infty$. Each $v_m$ satisfies \eqref{eq:NonlocalTrunc:Remainder1:Pf1}, hence by Fatou's lemma
    \begin{equation*}
        \begin{split}
            &\int_{\Omega} \int_{\Omega} \mathds{1}_{ \{ \eta(\bx) < \tau \} } \mathds{1}_{ \{ \gamma(\bx) \geq \tau \} } \rho \left( \frac{|\bx-\by|}{ \delta \tau } \right) \frac{|v(\bx)-v(\by)|^p}{ (\gamma(\bx))^\beta (\delta \tau)^{d+p} } \, \rmd \by \, \rmd \bx \\
            \leq & \int_{\bbR^d} \int_{\bbR^d} \mathds{1}_{ \{ \eta(\bx) < \tau \} } \mathds{1}_{ \{ \gamma(\bx) \geq \tau \} } \rho \left( \frac{|\bx-\by|}{ \delta \tau } \right) \frac{|\bar{E} v(\bx)-\bar{E} v(\by)|^p}{ (\gamma(\bx))^\beta (\delta \tau)^{d+p} } \, \rmd \by \, \rmd \bx \\
            \leq & \liminf_{m \to \infty} \int_{\bbR^d} \int_{\bbR^d} \mathds{1}_{ \{ \eta(\bx) < \tau \} } \mathds{1}_{ \{ \gamma(\bx) \geq \tau \} } \rho \left( \frac{|\bx-\by|}{ \delta \tau } \right) \frac{|v_m(\bx)-v_m(\by)|^p}{ (\gamma(\bx))^\beta (\delta \tau)^{d+p} } \, \rmd \by \, \rmd \bx \\
            \leq & C(\beta,\kappa) \int_{ \{ \bx \, : \, \eta(\bx) < (1+\kappa_1 \delta) \tau \} \cap \{ \bx \, : \, \gamma(\bx) > (1-\kappa_1 \delta) \tau \} }  \frac{|\grad \bar{E} v(\bx)|^p}{\gamma(\bx)^\beta} \, \rmd \bx.
        \end{split}
    \end{equation*}

\end{proof}

\begin{lemma}\label{lma:NonlocalTrunc:Remainder2}
    For $\beta \in \bbR$, let $v \in \mathfrak{V}^{p}[\delta](\Omega;\beta)$.
    Then the function $\bar{v} := \frac{v}{\gamma^{\beta/p}}$ belongs to $\mathfrak{W}^p[\delta](\Omega;0)$, with 
    \begin{equation*}
        \begin{split}
            \vnorm{\bar{v}}_{ \mathfrak{W}^p[\delta](\Omega;0) } \leq C(d,p,\beta,\kappa) \vnorm{ v }_{ \mathfrak{V}^p[\delta](\Omega;\beta) }.
        \end{split}
    \end{equation*}
\end{lemma}

\begin{proof}
    Since $|\bar{v}(\bx) - \bar{v}(\by)| \leq \gamma(\bx)^{-\beta/p} |v(\bx)-v(\by)| + |\gamma(\bx)^{-\beta/p}-\gamma(\by)^{-\beta/p}| |v(\by)|$, we have
    \begin{equation*}
        \begin{split}
            [\bar{v}]_{\mathfrak{W}^p[\delta](\Omega;0)}^p &\leq 2^{p-1} [v]_{\mathfrak{W}^p[\delta](\Omega;\beta)}^p \\
            &\quad + 2^{p-1} C(d,p) \int_\Omega \int_{B(\bx,\eta_\delta(\bx))} \frac{|\gamma(\bx)^{-\beta/p}-\gamma(\by)^{-\beta/p}|^p }{ (\eta_{\delta}(\bx))^{d+p} } |v(\by)|^p \, \rmd \by \, \rmd \bx.
        \end{split}
    \end{equation*}
    By \eqref{eq:Comp1}
    \begin{equation}\label{eq:NonlocalTrunc:GammaEst1}
        \begin{split}
            |\gamma(\bx)^{-\beta/p}-\gamma(\by)^{-\beta/p}| 
            &\leq \frac{|\beta|}{p} \kappa_1 \int_0^1 |\gamma(\bx) + t(\gamma(\by)-\gamma(\bx))|^{-\beta/p-1} |\gamma(\bx)-\gamma(\by)| \, \rmd t \\
            &\leq \frac{|\beta|}{p} \kappa_1^2 (1+C(\kappa)\delta) \gamma(\by)^{-\beta/p-1} |\bx-\by|.
        \end{split}
    \end{equation}
    Using this estimate in the second integral and then integrating in $\bx$ gives the desired result:
    \begin{equation}
        [\bar{v}]_{\mathfrak{W}^p[\delta](\Omega;0)}^p \leq C [v]_{\mathfrak{W}^p[\delta](\Omega;\beta)}^p + C \vnorm{v}_{L^p(\Omega;\beta+p)}^p.
    \end{equation}
\end{proof}

\begin{proof}[Proof of \Cref{thm:NonlocalTruncToNonlocal}]
    To begin, we use that $\eta(\bx) \leq \gamma(\bx)$ for all $\bx \in \Omega$, and split the integral defining $\cE_{\delta_n,\tau_n}(u)$ as follows:
    \begin{equation*}
        \begin{split}
            \cE_{\delta_n,\tau_n}(u) 
            &= \int_{\Omega \cap \{ \eta(\bx) \geq \tau_n \} \cap \{ \gamma(\bx) \geq \tau_n \} } \int_{\Omega} \rho \left( \frac{|\bx-\by|}{ \eta_{\delta_n}(\bx) } \right) \frac{|u(\bx)-u(\by)|^p}{ (\gamma(\bx))^\beta (\eta_{\delta_n}(\bx))^{d+p} } \, \rmd \by \, \rmd \bx \\
                &\qquad + \int_{\Omega \cap \{ \eta(\bx) < \tau_n \} \cap \{ \gamma(\bx) \geq \tau_n \} } \int_{\Omega} \rho \left( \frac{|\bx-\by|}{ \delta_n \tau_n } \right) \frac{|u(\bx)-u(\by)|^p}{ (\gamma(\bx))^\beta (\delta_n \tau_n)^{d+p} } \, \rmd \by \, \rmd \bx \\
                &\qquad + \int_{\Omega  \cap \{ \gamma(\bx) < \tau_n \}  } \int_{\Omega} \rho \left( \frac{|\bx-\by|}{ \delta_n \tau_n } \right) \frac{|u(\bx)-u(\by)|^p}{ \tau_n^{\beta} (\delta_n \tau_n)^{d+p} } \, \rmd \by \, \rmd \bx \\
            &:= I + II + III.
        \end{split}
    \end{equation*}
    Clearly $I \leq \cE_{\delta_n}(u) \leq C \cE_\delta(u)$ by \Cref{thm:InvariantHorizon} and \Cref{thm:EnergySpaceIndepOfKernel}, and by the dominated convergence theorem $I \to \cE_\delta(u)$ as $n \to \infty$. Therefore we just need to show that
    \begin{equation}\label{eq:NonlocalTrunc:MainPart}
        II + III \leq C \delta_n^{-p} \vnorm{u}_{ \mathfrak{W}^p[\delta](\Omega;\beta) } \qquad \text{ and } \qquad \lim\limits_{n \to \infty} II + III = 0.
    \end{equation}
    
    To this end, we estimate $II$.
    First, define $u_1 := u - u_\Gamma$, where $u_\Gamma:= E_\Gamma \circ T_\Gamma u$, where $E_\Gamma$ is defined in \Cref{thm:Extension} and where $T_\Gamma$ is defined as in \Cref{thm:Trace:Nonlocal}. 
    Next, define $u_2 := \frac{u_1}{ \gamma^{\beta/p} }$, which belongs to $\mathfrak{W}^p[\delta](\Omega;0)$ by \Cref{lma:NonlocalTrunc:Remainder2} and by \Cref{thm:Combined:Nonlocal}. 
    Third, denote the trace operator $T_{\p \Omega^*} : \mathfrak{W}^p[\delta](\Omega^*;0) \to W^{1-1/p,p}(\p \Omega^*)$ on the Lipschitz boundary of the domain $\Omega^*$ (see \Cref{thm:FxnSpaceProp} and also \cite{Scott2023Nonlocal}), and define a bounded linear extension operator $E_{\p \Omega^*} : W^{1-1/p,p}(\p \Omega^*) \to W^{1,p}(\Omega^*)$. Finally, define $u_{\p \Omega^*} := E_{\p \Omega^*} \circ T_{\p \Omega^*} u_2$ and define $u_3 := u_2 - u_{\p \Omega^*}$. Then we use the triangle inequality to get
    \begin{equation*}
        \begin{split}
            \frac{|u(\bx)-u(\by)|}{\gamma(\bx)^{\beta/p}}
            &\leq |u_3(\bx) - u_3(\by)| + |u_{\p \Omega^*}(\bx) - u_{\p \Omega^*}(\by)| \\
            &\quad + \frac{ |u_\Gamma(\bx) - u_\Gamma (\by)| }{ \gamma(\bx)^{\beta/p} } + \big| \gamma(\by)^{-\beta/p} - \gamma(\bx)^{-\beta/p} \big| |u_1(\by)|,
        \end{split}
    \end{equation*}
    for $\bx \in \Omega$ satisfying $\eta(\bx) < \tau_n$ and $\gamma(\bx) \geq \tau_n$ and $\by \in \Omega$,
    and we split the integral $II$ accordingly:
    \begin{equation*}
        II \leq C(p) (II_A + II_B + II_C + II_D).
    \end{equation*}
    We show that each piece has the appropriate upper bound, and we show that each piece also converges to $0$ as $n \to \infty$. 
    
    First, we note the following inequalities, valid in the integral $II$: 
    \begin{equation}\label{eq:NonlocalTrunc:BasicEstII}
        \begin{gathered}
        \eta(\by) \leq \eta(\bx) + \kappa_1 \delta_n \tau_n \leq (1+\kappa_1 \delta_n)\eta(\bx), \\
        |\gamma(\bx)-\gamma(\by)| \leq \kappa_1 \delta_n \tau_n \leq \kappa_1 \delta_n \gamma(\bx).
        \end{gathered}
    \end{equation}
    Then we estimate
    \begin{equation*}
        \begin{split}
            &II_A \\
            \leq& 2^{p-1} \left( \int_{\Omega \cap \{ \eta(\bx) < \tau_n \} \cap \{ \gamma(\bx) \geq \tau_n \} } \int_{B(\bx,\delta_n \tau_n)} \rho \left( \frac{|\bx-\by|}{ \delta_n \tau_n } \right) \frac{|u_3(\bx)|^p}{ (\delta_n \tau_n)^{d+p} } \, \rmd \by \, \rmd \bx \right. \\
            \left.+\right.&\left.\int_{\Omega \cap \{ \eta(\by) < (1+\kappa_1 \delta_n) \tau_n \} \cap \{ \gamma(\by) \geq (1-\kappa_1 \delta_n) \tau_n \} } \int_{B(\by,\delta_n \tau_n)} \rho \left( \frac{|\bx-\by|}{ \delta_n \tau_n } \right) \frac{|u_3(\by)|^p}{ (\delta_n \tau_n)^{d+p}  } \, \rmd \bx \, \rmd \by \right) \\
            \leq& C(p,\rho) \int_{\Omega \cap \{ \eta(\bx) < C(\kappa) \tau_n \} \cap \{ \gamma(\bx) \geq \frac{\tau_n}{C(\kappa)} \} } \frac{|u_3(\bx)|^p}{ (\delta_n \tau_n)^{p} } \, \rmd \bx \\
            \leq& \frac{ C(p,\rho,\kappa) }{ \delta_n^p } \int_{\Omega \cap \{ \eta(\bx) < C(\kappa) \tau_n \} } \frac{|u_3(\bx)|^p}{ \eta(\bx)^{p} } \, \rmd \bx.
        \end{split}
    \end{equation*}
    By the Hardy inequality \Cref{thm:FxnSpaceProp} item 7), 
    \Cref{lma:NonlocalTrunc:Remainder2}, and \Cref{thm:Combined:Nonlocal},
    \begin{equation}\label{eq:NonlocalTrunc:BdII1}
        \begin{split}
            \int_{\Omega} \frac{|u_3(\bx)|^p}{ \eta(\bx)^{p} } \, \rmd \bx 
            \leq C [u_2]_{\mathfrak{W}^p[\delta](\Omega;0)}^p 
            &\leq C\vnorm{u_1}_{ \mathfrak{V}^p[\delta](\Omega;\beta) }^p 
        \leq C \vnorm{u}_{\mathfrak{W}^p[\delta](\Omega;\beta)}^p.
        \end{split}
    \end{equation}
    Therefore, $\frac{|u_3(\bx)|^p}{ \eta(\bx)^{p} }  \in L^1(\Omega)$ and so by the dominated convergence theorem
    \begin{equation}\label{eq:NonlocalTrunc:ConvII1}
        \lim\limits_{n \to \infty} II_A \leq \lim\limits_{n \to \infty} \frac{C}{\delta_n^p} \int_{ \{ \bx : \, \eta(\bx) < C(\kappa) \tau_n \} } \frac{|u_3(\bx)|^p}{\eta(\bx)^p} \, \rmd \bx = 0.
    \end{equation}
    
    Second, thanks to \Cref{lma:NonlocalTrunc:Remainder1}, we have
    \begin{equation*}
        II_B \leq C \int_{ \{ \bx : \, \eta(\bx) < (1+\kappa_1 \delta_n) \tau_n \} } |\grad \bar{E} u_{\p \Omega^*}(\bx)|^p \, \rmd \bx.
    \end{equation*}
    By the classical Sobolev extension theorem, the boundedness of $E_{\p \Omega^*}$ and $T_{\p \Omega^*}$, \Cref{lma:NonlocalTrunc:Remainder2}, and \Cref{thm:Combined:Nonlocal}, 
    \begin{equation}\label{eq:NonlocalTrunc:BdII2}
        \begin{split}
        \vnorm{ \bar{E} u_{\p \Omega^*} }_{W^{1,p}(\bbR^d)} 
        \leq C(d,p,\Omega) \vnorm{ u_{\p \Omega^*} }_{W^{1,p}(\Omega)}
        &\leq C \vnorm{u_2}_{ \mathfrak{W}^p[\delta](\Omega;0) } \\
        &\leq C \vnorm{u_1}_{ \mathfrak{V}^p[\delta](\Omega;\beta) }
        \leq C \vnorm{u}_{\mathfrak{W}^p[\delta](\Omega;\beta)}.
        \end{split}
    \end{equation}
    Therefore, $|\grad \bar{E} u_{\p \Omega^*}|^p \in L^1(\bbR^d)$ and by the dominated convergence theorem
    \begin{equation}\label{eq:NonlocalTrunc:ConvII2}
        \lim\limits_{n \to \infty} II_B \leq C \lim\limits_{n \to \infty} \int_{ \{ \bx : \, \eta(\bx) < (1+\kappa_1 \delta_n) \tau_n \} } |\grad \bar{E} u_{\p \Omega^*}(\bx)|^p \, \rmd \bx = 0.
    \end{equation}
    
    Third, a similar argument can be used to estimate $II_C$; using \Cref{lma:NonlocalTrunc:Remainder1}
    \begin{equation*}
        II_C \leq C \int_{ \{ \bx : \, \eta(\bx) < (1+\kappa_1 \delta_n) \tau_n \} } \frac{|\grad \bar{E} u_{\Gamma}(\bx)|^p}{\gamma(\bx)^\beta} \, \rmd \bx.
    \end{equation*}
    By \Cref{thm:ApExtension}, \Cref{thm:Extension}, and \Cref{thm:Trace:Nonlocal}, 
    \begin{equation}\label{eq:NonlocalTrunc:BdII3}
        \begin{split}
        \int_{\bbR^d} \frac{|\grad \bar{E} u_{\Gamma} (\bx)|^p}{ \gamma(\bx)^\beta } \, \rmd \bx
        \leq C \vnorm{ u_{\Gamma} }_{W^{1,p}(\Omega;\beta)}
        &\leq C \vnorm{u}_{ \mathfrak{W}^p[\delta](\Omega;\beta) }.
        \end{split}
    \end{equation}
    Therefore, $\frac{|\grad \bar{E} u_{\Gamma}|^p}{\gamma(\bx)^\beta} \in L^1(\bbR^d)$ and by the dominated convergence theorem
    \begin{equation}\label{eq:NonlocalTrunc:ConvII3}
        \lim\limits_{n \to \infty} II_C \leq C \lim\limits_{n \to \infty} \int_{ \{ \bx : \, \eta(\bx) < (1+\kappa_1 \delta_n) \tau_n \} } \frac{|\grad \bar{E} u_{\Gamma}(\bx)|^p}{\gamma(\bx)^\beta} \, \rmd \bx = 0.
    \end{equation}

    Fourth, we use 
    \eqref{eq:NonlocalTrunc:BasicEstII} to get that
    \begin{equation*}
        |\gamma(\bx)^{-\beta/p}-\gamma(\by)^{-\beta/p}| \leq C(p,\beta,\kappa) \gamma(\by)^{-\beta/p-1} |\bx-\by|,
    \end{equation*}
    similarly to the estimate \eqref{eq:NonlocalTrunc:GammaEst1}. Then again using \eqref{eq:NonlocalTrunc:BasicEstII} and the change of variables $\bz = \frac{\bx-\by}{\delta_n \tau_n}$
    \begin{equation*}
        \begin{split}
        II_D 
        &\leq \frac{C}{|\Omega|^2} \int_{\Omega \cap \{ \eta(\by) < C(\kappa) \tau_n \} \cap \{ \gamma(\by) \geq \frac{\tau_n}{C(\kappa)} \} } \int_{B(0,1)} \rho \left( |\bz| \right) |\bz|^p \, \rmd \bz \frac{|u_1(\by)|^p}{ (\gamma(\by))^{\beta+p} }  \, \rmd \by \\
        &\leq C(d,p,\beta,\kappa,\Omega) \int_{\Omega \cap \{ \eta(\by) < C(\kappa) \tau_n \} \cap \{ \gamma(\by) \geq \frac{\tau_n}{C(\kappa)} \} } \frac{|u_1(\by)|^p}{ (\gamma(\by))^{\beta+p} } \, \rmd \by.
        \end{split}
    \end{equation*}
    By \Cref{thm:Combined:Nonlocal}, 
    \begin{equation}\label{eq:NonlocalTrunc:BdII4}
        \begin{split}
        \int_{\Omega} \frac{|u_1(\bx)|^p}{ \gamma(\bx)^{\beta+p} } \, \rmd \bx
        &\leq C(d,p,\beta,\Omega,\kappa) \vnorm{u}_{ \mathfrak{W}^p[\delta](\Omega;\beta) }.
        \end{split}
    \end{equation}
    Therefore, $\frac{|u_1(\bx)|^p}{\gamma(\bx)^\beta} \in L^1(\Omega)$ and by the dominated convergence theorem
    \begin{equation}\label{eq:NonlocalTrunc:ConvII4}
        \lim\limits_{n \to \infty} II_D \leq C \lim\limits_{n \to \infty} \int_{ \{ \bx : \, \eta(\bx) < C(\kappa) \tau_n \} } \frac{|u_1(\bx)|^p}{\gamma(\bx)^{\beta+p}} \, \rmd \bx = 0.
    \end{equation}
    Combining \eqref{eq:NonlocalTrunc:BdII1}-\eqref{eq:NonlocalTrunc:BdII2}-\eqref{eq:NonlocalTrunc:BdII3}-\eqref{eq:NonlocalTrunc:BdII4} gives
    \begin{equation}\label{eq:NonlocalTrunc:BdII}
        II \leq \frac{C}{\delta_n^p} \vnorm{u}_{ \mathfrak{W}^p[\delta](\Omega;\beta) },
    \end{equation}
    and combining \eqref{eq:NonlocalTrunc:ConvII1}-\eqref{eq:NonlocalTrunc:ConvII2}-\eqref{eq:NonlocalTrunc:ConvII3}-\eqref{eq:NonlocalTrunc:ConvII4} gives
    \begin{equation}\label{eq:NonlocalTrunc:ConvII}
        \limsup_{n \to \infty} II = 0.
    \end{equation}

    Last, we estimate $III$. Choose $n_0$ such that $(1+\underline{\delta}_0)\tau_n < R$, and write
    \begin{equation*}
        III = \sum_{\bx_0 \in \Gamma} \int_{\Omega \cap B(\bx_0,\tau_n)} \int_\Omega \rho \left(  \frac{|\bx-\by|}{\delta_n \tau_n} \right) \frac{|u(\bx)-u(\by)|^p}{ \tau_n^\beta (\delta_n \tau_n)^{d+p} } \, \rmd \by \, \rmd \bx. 
    \end{equation*}
    For each $\bx_0 \in \Gamma$, we have $|\by-\bx_0| \leq |\bx-\bx_0| + |\by-\bx| \leq (1+\delta_n) \tau_n$ in the integrand, so 
    \begin{equation*}
        \begin{split}
            &\int_{\Omega \cap B(\bx_0,\tau_n)} \int_\Omega \rho \left(  \frac{|\bx-\by|}{\delta_n \tau_n} \right) \frac{|u(\bx)-u(\by)|^p}{ \tau_n^\beta (\delta_n \tau_n)^{d+p} } \, \rmd \by \, \rmd \bx \\
            \leq& 2^{p-1} \int_{\Omega \cap B(\bx_0,\tau_n)} \int_\Omega \rho \left(  \frac{|\bx-\by|}{\delta_n \tau_n} \right) \frac{|u(\bx)-u(\bx_0)|^p}{ \tau_n^\beta (\delta_n \tau_n)^{d+p} } \, \rmd \by \, \rmd \bx \\
            &\quad + 2^{p-1} \int_{\Omega \cap B(\bx_0,(1+\delta_n)\tau_n)} \int_\Omega \rho \left(  \frac{|\bx-\by|}{\delta_n \tau_n} \right) \frac{|u(\by)-u(\bx_0)|^p}{ \tau_n^\beta (\delta_n \tau_n)^{d+p} } \, \rmd \bx \, \rmd \by \\
            \leq& \frac{C(p,\rho)}{\delta_n^p} \int_{\Omega \cap B(\bx_0,(1+\delta_n)\tau_n)} \frac{|u(\bx)-u(\bx_0)|^p}{ \tau_n^{\beta+p} } \, \rmd \bx.
        \end{split}
    \end{equation*}
    Since $\beta \geq 0$, we have 
    \begin{equation*}
        III \leq \frac{C(p,\rho)}{\delta_n^p} \sum_{\bx_0 \in \Gamma} \int_{\Omega \cap B(\bx_0,(1+\delta_n)\tau_n)} \frac{|u(\bx)-u(\bx_0)|^p}{ |\bx-\bx_0|^{\beta+p} } \, \rmd \bx.
    \end{equation*}
    Now, by \Cref{thm:Combined:Nonlocal} 
    \begin{equation}\label{eq:NonlocalTrunc:BdIII}
        III \leq \frac{C(p,\rho)}{\delta_n^p} \sum_{\bx_0 \in \Gamma} \int_{\Omega \cap B(\bx_0,(1+\delta_n)\tau_n)} \frac{|u(\bx)-u(\bx_0)|^p}{ |\bx-\bx_0|^{\beta+p} } \, \rmd \bx \leq \frac{C(p,\rho)}{\delta_n^p} [u]_{ \mathfrak{W}^p[\delta](\Omega;\beta) }^p.
    \end{equation}
    Therefore $\frac{|u(\bx)-u(\bx_0)|^p}{|\bx-\bx_0|^{\beta+p}} \in L^1(\Omega)$ and by the dominated convergence theorem
    \begin{equation}\label{eq:NonlocalTrunc:ConvIII}
        \lim\limits_{n \to \infty} III \leq C \lim\limits_{n \to \infty} \sum_{\bx_0 \in \Gamma} \int_{ \{ \bx : \, \gamma(\bx) < (1+\delta_0) \tau_n \} } \frac{|u(\bx)-u(\bx_0)|^p}{|\bx-\bx_0|^{\beta+p}} \, \rmd \bx = 0.
    \end{equation}
    We finally conclude \eqref{eq:NonlocalTrunc:MainPart}, and thus the proof, by using \eqref{eq:NonlocalTrunc:BdII}, \eqref{eq:NonlocalTrunc:BdIII}, \eqref{eq:NonlocalTrunc:ConvII}, and \eqref{eq:NonlocalTrunc:ConvIII}.
        
\end{proof}

\subsection{Elementary estimates}

\begin{lemma}\label{lma:NonlocalTrunc:}
    For $T_n$, $\zeta_n$, $\tau_n$ and $c_n$ as defined in \Cref{thm:GraphToNonlocalConv}, we have 
    \begin{align}
        \label{eq:CutoffHorizon:Comp1}
        \eta^{\tau_n}(T_n(\bx)) &\leq (1+\kappa_1 c_n/2) \eta^{\tau_n}(\bx), \quad \forall \bx \in \Omega, \\
        \label{eq:CutoffHorizon:Comp2}
        \eta^{\tau_n}(\bx) &\leq (1+\kappa_1 c_n/2) \eta^{\tau_n}(T_n(\bx)), \quad \forall \bx \in \Omega.
    \end{align}            
\end{lemma}
\begin{proof}
    First, suppose that $\eta(\bx) < \tau_n$. Then
    \begin{equation*}
        \begin{split}
            \eta(T_n(\bx)) &\leq \eta(\bx) + \kappa_1 |T_n(\bx) - \bx| \\
            &\leq \eta(\bx) + \kappa_1 \vnorm{T_n - Id}_{L^\infty(\Omega)} \\
            &\leq \eta(\bx) + \frac{\kappa_1 c_n}{2} \tau_n < (1+\kappa_1  c_n/2) \tau_n.
        \end{split}
    \end{equation*}
    So we have that $\eta^{\tau_n}(T_n(\bx)) \leq (1+\kappa_1  c_n/2) \tau_n \leq (1+\kappa_1  c_n/2) \max \{ \eta(\bx), \tau_n \}$.
    Now suppose that $\eta(\bx) \geq \tau_n$. Then
    \begin{equation*}
        \begin{split}
            \eta(T_n(\bx)) &\leq \eta(\bx) + \kappa_1 |T_n(\bx) - \bx| \\
            &\leq \eta(\bx) + \kappa_1 \vnorm{T_n - Id}_{L^\infty(\Omega)} \\
            &\leq \eta(\bx) + \frac{\kappa_1 c_n}{2} \tau_n < (1+\kappa_1 c_n/2) \eta(\bx).
        \end{split}
    \end{equation*}
    So we have that $\eta^{\tau_n}(T_n(\bx)) \leq (1+\kappa_1 c_n/2) \eta(\bx) \leq (1+\kappa_1 c_n/2) \max \{\eta(\bx), \tau_n\}$. Thus \eqref{eq:CutoffHorizon:Comp1} is proved.

    The estimate \eqref{eq:CutoffHorizon:Comp2} is proved similarly, exchanging the roles of $\bx$ and $T_n(\bx)$ in the above estimates. 
\end{proof}

The function $\gamma$ has the same Lipschitz constant as $\eta$, so we can similarly establish the following theorem:

\begin{lemma}
    With $T_n$, $\zeta_n$, $\tau_n$, and $c_n$ defined as in \Cref{thm:GraphToNonlocalConv}, we have
    \begin{align}
        \label{eq:CutoffWeight:Comp1}
        \gamma^{\tau_n}(T_n(\bx)) &\leq (1+\kappa_1 c_n /2) \gamma^{\tau_n}(\bx), \\
        \label{eq:CutoffWeight:Comp2}
        \gamma^{\tau_n}(\bx) &\leq  (1+\kappa_1 c_n /2) \gamma^{\tau_n}(T_n(\bx)).
    \end{align}
\end{lemma}

\begin{lemma}\label{lma:DiscToCont:LowerBound}
    With the assumptions of \Cref{thm:GraphToNonlocalConv}, define $q_{n,\delta} := \frac{ 1 }{ (1+\kappa_1 c_n/2) } - \frac{c_n}{\delta}$. Then we have for all $\bx$, $\by \in \Omega$,
    \begin{equation*}
        |\bx-\by| < \eta_{q_{n,\delta} \delta}^{\tau_n}(\bx) \qquad \Rightarrow \qquad |T_n(\bx)-T_n(\by)| < \eta_\delta^{\tau_n}(T_n(\bx)).
    \end{equation*}
\end{lemma}

\begin{proof}
    By the above lemmas,
    \begin{equation*}
        \begin{split}
            |T_n(\bx) - T_n(\by)| &\leq 2 \vnorm{T_n - Id}_{L^\infty(\Omega)} + |\bx-\by|  \\
            &\leq c_n \tau_n + q_{n,\delta} \delta \max \{ \eta(\bx), \tau_n \} \\
            &\leq (c_n+q_{n,\delta} \delta) \max \{ \eta(\bx), \tau_n \} \\
            &\leq (c_n+q_{n,\delta} \delta) (1 + \kappa_1 c_n/2) \delta \max \{ \eta(T_n(\bx)), \tau_n \} = \eta_\delta^{\tau_n}(T_n(\bx)).
        \end{split}
    \end{equation*}
\end{proof}

\begin{lemma}\label{lma:DiscToCont:UpperBound}
    With the assumptions of \Cref{thm:GraphToNonlocalConv}, define $Q_{n,\delta} := (1+\kappa_1 c_n/2) ( 1 + \frac{c_n}{\delta})$. Then we have for all $\bx$, $\by \in \Omega$,
    \begin{equation*}
        |T_n(\bx)-T_n(\by)| < \eta_{\delta}^{\tau_n}(T_n(\bx)) \qquad \Rightarrow \qquad |\bx-\by| < \eta_{Q_{n,\delta}\delta}^{\tau_n}(\bx).
    \end{equation*}
\end{lemma}

\begin{proof}
    By the above lemmas,
    \begin{equation*}
        \begin{split}
            |\bx-\by| &\leq 2 \vnorm{T_n - Id}_{L^\infty(\Omega)} + |T_n(\bx)-T_n(\by)|  \\
            &\leq  c_n \tau_n + \delta \max \{ \eta(T_n(\bx)), \tau_n \} \\
            &\leq (c_n + \delta) \max \{ \eta(T_n(\bx)), \tau_n \} \\
            &\leq (c_n + \delta) (1+\kappa_1 c_n/2) \max \{ \eta(\bx), \tau_n \} = \eta_{Q_{n,\delta} \delta}^{\tau_n}(\bx).
        \end{split}
    \end{equation*}
\end{proof}

\subsection{Discrete and truncated nonlocal energy comparison}

\begin{theorem}[Upper bound]\label{thm:Discrete:UpperBd}
    With all the assumptions of \Cref{thm:GraphToNonlocalConv}, and $Q_{n,\delta}$ as defined in \Cref{lma:DiscToCont:UpperBound},
    \begin{equation*}
        \cE_{n,\delta,\tau_n}(u \circ T_n^{-1}) \leq Q_{n,\delta}^{d+p} (1+\kappa_1 c_n/2)^{d+p+\beta}  |\Omega|^{-2} \cE_{Q_{n,\delta}\delta,\tau_n}(u), \qquad \forall u \in C^1(\overline{\Omega}).
    \end{equation*}
\end{theorem}

\begin{proof}
    By a change of variables (i.e. $\mu_n(U) = \mu(T_n^{-1}(U))$), we have
    \begin{equation}\label{eq:DiscreteEnergy:CoV}
        \begin{split}
            &\cE_{n,\delta,\tau_n}(u \circ T_n^{-1}) \\
            =& \int_{\Omega} \int_{\Omega} \rho \left( \frac{ |\bx-\by| }{ \eta_\delta^{\tau_n}(\bx) } \right) \frac{|u(T_n^{-1}(\bx)) -u(T_n^{-1}(\by))|^p}{\gamma^{\tau_n}(\bx)^{\beta} (\eta_\delta^{\tau_n}(\bx) )^{d+p}} \, \rmd \mu_n(\by) \, \rmd \mu_n(\bx) \\
            =& \frac{1}{|\Omega|^2} \int_{\Omega} \int_{\Omega} \rho \left( \frac{ |T_n(\bx)-T_n(\by)| }{ \eta_\delta^{\tau_n}(T_n(\bx))  } \right) \frac{|u(\bx) -u(\by)|^p}{\gamma^{\tau_n}(T_n(\bx))^\beta (\eta_\delta^{\tau_n}(T_n(\bx)) )^{d+p}} \, \rmd \by \, \rmd \bx.
        \end{split}
    \end{equation}
    
    \underline{Step 1}: Assume that $\rho(|\bx|) = \mathds{1}_{B(0,1)}(\bx)$. By \Cref{lma:DiscToCont:UpperBound}
    \begin{equation*}
        \rho \left( \frac{ |T_n(\bx)-T_n(\by)| }{ \eta_\delta^{\tau_n}(T_n(\bx)) } \right) \leq \rho \left( \frac{ |\bx-\by| }{ \eta_{Q_{n,\delta}\delta}^{\tau_n}(\bx) }  \right).
    \end{equation*}
    Therefore, along with \eqref{eq:CutoffHorizon:Comp2} and \eqref{eq:CutoffWeight:Comp2}
    \begin{equation*}
        \begin{split}
            &\cE_{n,\delta,\tau_n}(u \circ T_n^{-1}) \\
            \leq& \frac{ \left( 1 + \kappa_1 c_n/2\right)^{d+p+\beta} }{|\Omega|^{2}} \int_{\Omega} \int_{\Omega} \rho \left( \frac{ |\bx-\by| }{ \eta_{Q_{n,\delta}\delta}^{\tau_n}(\bx) } \right) \frac{|u(\bx) -u(\by)|^p}{ \gamma^{\tau_n}(\bx)^{\beta} (\eta_\delta^{\tau_n}(\bx))^{d+p}} \, \rmd \by \, \rmd \bx \\
            =& \frac{ Q_{n,\delta}^{d+p} \left( 1 + \kappa_1 c_n/2\right)^{d+p+\beta} }{|\Omega|^{2}} \int_{\Omega} \int_{\Omega} \rho \left( \frac{ |\bx-\by| }{ \eta_{Q_{n,\delta}\delta}^{\tau_n}(\bx) } \right) \frac{|u(\bx) -u(\by)|^p}{ \gamma^{\tau_n}(\bx)^{\beta} (\eta_{Q_{n,\delta}\delta}^{\tau_n}(\bx))^{d+p}} \, \rmd \by \, \rmd \bx.
        \end{split}
    \end{equation*}

    \underline{Step 2}: Assume that $\rho$ is a piecewise constant function satisfying \eqref{assump:VarProb:Kernel} sans the normalization condition. Then $\rho = \sum_{m = 1}^M \rho_m$ for some $M \in \bbN$ and for some $\rho_m$ as in Step 1. Denote the energies with kernel $\rho_m$ as $\cE_{n,\delta,\tau_n}^m$ and $\cE_{q_{n,\delta}\delta,\tau_n}^m$. Then by Step 1
    \begin{equation*}
        \begin{split}
            \cE_{n,\delta,\tau_n}(u\circ T_n^{-1}) 
            &= \sum_{m=1}^M \cE_{n,\delta,\tau_n}^m(u\circ T_n^{-1}) \\
            &\leq \frac{ Q_{n,\delta}^{d+p} \left( 1 + \kappa_1 c_n/2\right)^{d+p+\beta} }{|\Omega|^{2}} \sum_{m=1}^M \cE_{Q_{n,\delta}\delta,\tau_n}^m(u) \\
            &= \frac{ Q_{n,\delta}^{d+p} \left( 1 + \kappa_1 c_n/2\right)^{d+p+\beta} }{|\Omega|^{2}} \cE_{Q_{n,\delta}\delta,\tau_n}(u).
        \end{split}
    \end{equation*}

    \underline{Step 3}: Let $\rho$ satisfy \eqref{assump:VarProb:Kernel}. Then there exists an increasing sequence $\rho_m$ of piecewise constant functions, each of which satisfy the assumptions of Step 2, such that $\rho_m \nearrow \rho$ almost everywhere.
    Denote the energies with kernel $\rho_m$ as $\cE_{n,\delta,\tau_n}^m$ and $\cE_{q_{n,\delta}\delta,\tau_n}^m$. Then by Step 2 and the monotone convergence theorem
    \begin{equation*}
        \begin{split}
            \cE_{n,\delta,\tau_n}(u\circ T_n^{-1}) 
            &= \lim\limits_{m \to \infty} \cE_{n,\delta,\tau_n}^m(u\circ T_n^{-1}) \\
            &\leq \lim\limits_{m \to \infty} \frac{Q_{n,\delta}^{d+p} \left( 1 + \kappa_1 c_n/2\right)^{d+p+\beta}}{|\Omega|^2} \cE_{Q_{n,\delta}\delta,\tau_n}^m(u) \\
            &= \frac{ Q_{n,\delta}^{d+p} \left( 1 + \kappa_1 c_n/2\right)^{d+p+\beta} }{|\Omega|^2} \cE_{Q_{n,\delta}\delta,\tau_n}(u).
        \end{split}
    \end{equation*}

    Last, we note that $\cE_{n,\delta,\tau_n}(u\circ T_n^{-1})$ is finite for $u \in C^1(\overline{\Omega})$, since
    \begin{equation*}
        \cE_{Q_{n,\delta}\delta,\tau_n}(u) \leq C(\rho) \text{Lip}(u)^p \int_{\Omega} \frac{1}{\gamma^{\tau_n}(\bx)^\beta} \, \rmd \bx \leq C(\rho) \text{Lip}(u)^p \int_{\Omega} \frac{1}{\gamma(\bx)^\beta} \, \rmd \bx < \infty.
    \end{equation*}
\end{proof}

\begin{theorem}[Lower bound]\label{thm:Discrete:LowerBd}
    With all the assumptions of \Cref{thm:GraphToNonlocalConv}, and $q_{n,\delta}$ as defined in \Cref{lma:DiscToCont:LowerBound},
    \begin{equation*}
        \cE_{n,\delta,\tau_n}(u \circ T_n^{-1}) \geq \frac{q_{n,\delta}^{d+p}}{|\Omega|^2(1+\kappa_1 c_n/2)^{d+p+\beta}} \cE_{q_{n,\delta}\delta,\tau_n}(u), \qquad \forall u \in \mathfrak{W}^p[\delta](\Omega;\beta).
    \end{equation*}
\end{theorem}

\begin{proof}
    The proof is similar to that of \Cref{thm:Discrete:UpperBd}.
    Recall the change of variables \eqref{eq:DiscreteEnergy:CoV}.
    
    \underline{Step 1}: Assume that $\rho(|\bx|) = y_0 \mathds{1}_{B(0,x_0)}(\bx)$ for some $x_0 > 0$, $y_0 > 0$.
    By \Cref{lma:DiscToCont:LowerBound} we have
    \begin{equation*}
        \rho \left( \frac{ |\bx-\by| }{ \eta_{q_{n,\delta}\delta}^{\tau_n}(\bx) } \right) \leq \rho \left( \frac{ |T_n(\bx)-T_n(\by)| }{ \eta_\delta^{\tau_n}(T_n(\bx)) } \right).
    \end{equation*}
    Therefore, along with \eqref{eq:CutoffHorizon:Comp1} and \eqref{eq:CutoffWeight:Comp1}
    \begin{equation*}
        \begin{split}
            &\cE_{n,\delta,\tau_n}(u \circ T_n^{-1}) \\
            \geq& \frac{1}{|\Omega|^2(1+\kappa_1 c_n/2)^{d+p+\beta}}  \int_{\Omega} \int_{\Omega} \rho \left( \frac{ |\bx-\by| }{ \eta_{q_{n,\delta}\delta}^{\tau_n}(\bx) } \right) \frac{|u(\bx) -u(\by)|^p}{\gamma^{\tau_n}(T_n(\bx))^{\beta} \eta_{\delta}^{\tau_n}(\bx)^{d+p}} \, \rmd \by \, \rmd \bx \\
            =& \frac{ q_{n,\delta}^{d+p} }{|\Omega|^2(1+\kappa_1 c_n/2)^{d+p+\beta} } \int_{\Omega} \int_{\Omega} \rho \left( \frac{ |\bx-\by| }{ \eta_{q_{n,\delta}\delta}^{\tau_n}(\bx) } \right) \frac{|u(\bx) -u(\by)|^p}{\gamma^{\tau_n}(T_n(\bx))^{\beta} \eta_{q_{n,\delta}\delta}^{\tau_n}(\bx)^{d+p}} \, \rmd \by \, \rmd \bx.
        \end{split}
    \end{equation*}
    The theorem is then established in two additional steps, identical to those in the proof of \Cref{thm:Discrete:UpperBd}.
\end{proof}

\begin{proof}[Proof of \Cref{thm:GraphToNonlocalConv}]
    First, since $Q_{n,\delta} \delta \to \delta$ as $n \to \infty$, we note the following as a consequence of \Cref{thm:Discrete:UpperBd} and \eqref{eq:TruncNonlocal:Finite}: there exists $n_0 \in \bbN$ depending on $\underline{\delta}_0$ and $R$ (i.e. on $\Gamma$) such that for all $n \geq n_0$
    \begin{equation*}
        \cE_{n,\delta,\tau_n}(v \circ T_n^{-1}) \leq \frac{C(d,p,\beta,\rho,\kappa,\Omega)}{\delta^p} \vnorm{v}_{\mathfrak{W}^p[\delta](\Omega;\beta)}^p, \qquad \forall v \in C^1(\overline{\Omega}).
    \end{equation*}
    Therefore, for general $u \in \mathfrak{W}^p[\delta](\Omega;\beta)$ we take a sequence $\{u_m\} \subset C^\infty(\overline{\Omega})$ converging to $u$ in $\mathfrak{W}^p[\delta](\Omega;\beta)$, and apply Fatou's lemma and the change of variables \eqref{eq:DiscreteEnergy:CoV} to get
    \begin{equation*}
        \begin{split}
             &\int_{\Omega} \int_{\Omega} \rho \left( \frac{ |T_n(\bx)-T_n(\by)| }{ \eta_\delta^{\tau_n}(T_n(\bx))  } \right) \frac{|u(\bx) -u(\by)|^p}{\gamma^{\tau_n}(T_n(\bx))^\beta (\eta_\delta^{\tau_n}(T_n(\bx)) )^{d+p}} \, \rmd \by \, \rmd \bx \\
             &\leq \liminf_{m \to \infty} \cE_{n,\delta,\tau_n}(u_m \circ T_n^{-1}) \\
             &\leq \frac{C}{\delta^p} \liminf_{m \to \infty} \vnorm{u_m}_{\mathfrak{W}^p[\delta](\Omega;\beta)}^p = \frac{C}{\delta^p} \liminf_{m \to \infty} \vnorm{u}_{\mathfrak{W}^p[\delta](\Omega;\beta)}^p.
        \end{split}
    \end{equation*}
    Thus \eqref{eq:Discrete:Finite} is established, where the discrete energy evaluated at functions $u \in \mathfrak{W}^p[\delta](\Omega;\beta)$ is defined via the coordinate change \eqref{eq:DiscreteEnergy:CoV}:
    \begin{equation*}
        \cE_{n,\delta,\tau_n}(u \circ T_n^{-1}) := \frac{1}{|\Omega|^2} \int_{\Omega} \int_{\Omega} \rho \left( \frac{ |T_n(\bx)-T_n(\by)| }{ \eta_\delta^{\tau_n}(T_n(\bx))  } \right) \frac{|u(\bx) -u(\by)|^p}{\gamma^{\tau_n}(T_n(\bx))^\beta (\eta_\delta^{\tau_n}(T_n(\bx)) )^{d+p}} \, \rmd \by \, \rmd \bx.
    \end{equation*}

    Finally, since
    \begin{equation*}
        \lim\limits_{n \to \infty} q_{n,\delta} = \lim\limits_{n \to \infty} Q_{n,\delta} = \lim\limits_{n \to \infty} \frac{q_{n,\delta}^{d+p}}{(1+\kappa_1 c_n/2)^{d+p+\beta}} = \lim\limits_{n \to \infty} Q_{n,\delta}^{d+p} (1+\kappa_1 c_n/2)^{d+p+\beta} = 1,
    \end{equation*}
    and since \eqref{eq:TruncNonlocal:Limit} implies
    \begin{equation*}
        \lim\limits_{n \to \infty} \cE_{q_{n,\delta}\delta,\tau_n}(u) = \lim\limits_{n \to \infty} \cE_{Q_{n,\delta}\delta,\tau_n}(u) = \cE_\delta(u),
    \end{equation*}
    the convergence result $\lim\limits_{n \to \infty} \cE_{n,\delta,\tau_n}(u \circ T_n^{-1}) = |\Omega|^{-2} \cE_\delta(u)$ follows from \Cref{thm:Discrete:UpperBd}, \Cref{thm:Discrete:LowerBd}, and the squeeze theorem.
\end{proof}

	\section{Conclusion}
In this work, nonlocal continuum operators and related nonlocal problems are proposed and analyzed to model the semi-supervised learning. They may be viewed as bridges linking the discrete graph-based model with the local PDE model, thus offering new theoretical tools to understand the learning of large data sets in various scaling limits.
The well-posed nonlocal variational framework also provides  opportunities to design different sampling algorithms and discretization strategies \cite{du2024asymptotically}.

    Our setting can treat the case in which the value of $\ell$, the dimension of the labeled data set $\Gamma$, can change from one connected component of $\Gamma$ to the other. This provides a framework to account for the spatial heterogeneity in the data source, which could be of interest in real-world applications.
        Treating more general sets (e.g., $\Gamma \subset \bbR^3$ is the union of the plane $[-1,1]^2 \times \{0\}$ with the  perpendicular line segment $\{0\}^2 \times [-1,1]$) is outside of the scope of this paper, and will be left for future investigations.

        We remark that our analysis remains true for energies $\cE_\delta$ with weight function $\eta(\bx)$ replaced with a generalized distance function, say $\bar{\lambda}(\bx)$, that is comparable to $\eta(\bx)$. 
        Even more generally, similar results hold for $\eta$ replaced with powers of $\bar{\lambda}(\bx)$;
        when $\gamma \equiv 1$, such analysis for general forms of heterogeneous localization has been undertaken explicitly in \cite{Scott2023Nonlocal,Scott2023Nonlocala}. 
    
    Other extensions, such as the $p\to \infty$ limit, that may lead to interesting applications. Meanwhile, while various limits have been established in this work, further investigations are needed to provide estimates on the rate of the convergence and the computational complexity involved. Answers to such questions will depend on the rate of localization, i.e. the power of $\bar{\lambda}$ taken in the heterogeneous localization $\eta$.

\textbf{Acknowledgments.}
    This work is supported in part by the National Science Foundation DMS-2309245 and DMS-1937254. 
    The authors thank Jeff Calder for inspiring and helpful discussions.

\appendix

\section{Nonlocal function space}

    \begin{proof}[Proof of \Cref{thm:InvariantHorizon}]
	The second inequality is trivial, so the proof is devoted to the first inequality.
	Let $n \in \bbN$.
	To begin, we apply the triangle inequality to the telescoping sum for $\bx \in \Omega$ and $\bs \in B(0, \delta_2 \eta(\bx))$:
	\begin{equation*}
		|u(\bx+\bs) - u(\bx)| \leq \sum_{i = 1}^n \left| u \left( \bx + \frac{i}{n} \bs \right) - u \left( \bx + \frac{i-1}{n} \bs \right)\right|.
	\end{equation*}
Setting $\bx_i := \bx + \frac{i-1}{n} \bs$ and using H\"older's inequality, we get 
	\begin{equation*}
		\begin{split}
			[u]_{ \mathfrak{W}^{p}[\delta_2](\Omega;\beta) }^p \leq \bar{c} n^{p-1} \sum_{i = 1}^n \int_{ \Omega } \int_{B(0,\delta_2 \eta(\bx) ) } \frac{ |u(\bx_i+\frac{1}{n}\bs)-u(\bx_i)|^p }{ |\delta_2 \eta(\bx)|^{d+p} \gamma(\bx)^\beta } \, \rmd \bs \, \rmd \bx.
		\end{split}
	\end{equation*}
	where $\bar{c} =  \frac{ \overline{C}_{d,p}(d+p) }{ \sigma(\bbS^{d-1}) }$ is the normalizing constant. Now, since the distance function $\eta$ is Lipschitz,
        $
            |\eta(\bx_i) - \eta(\bx)| \leq |\bx_i-\bx| \leq |\bs| \leq \delta_2 \eta(\bx)
        $,
	and rearranging this inequality and using the assumption \eqref{assump:Horizon} gives
        \begin{equation}\label{eq:VaryingLambda:SeminormComparison:Pf2}
		\frac{3}{4} \eta(\bx_i)
        \leq \frac{\eta(\bx_i)}{1+\delta_2 } \leq \eta(\bx) \leq \frac{\eta(\bx_i)}{1 - \delta_2 }
        \leq \frac{3}{2} \eta(\bx_i)
	\end{equation}
	for all $\bx \in \Omega$. Similarly, using the properties described in \eqref{assump:weight}
        \begin{equation}\label{eq:VaryingLambda:SeminormComparison:Pf3}
            |\gamma(\bx_i) - \gamma(\bx)| \leq \kappa_1 |\bs| \leq \kappa_1 \delta_2 \eta(\bx) \leq \kappa_1 \delta_2 
 \dist(\bx,\Gamma) \leq \kappa_0 \kappa_1 \delta_2 
\gamma(\bx).
        \end{equation}
 
	Using \eqref{eq:VaryingLambda:SeminormComparison:Pf2} and \eqref{eq:VaryingLambda:SeminormComparison:Pf3}, we come to
    \begin{equation*}\label{eq:VaryingLambda:SeminormComparison:Pf4}
		\begin{split}
		&[u]_{ \mathfrak{W}^{p}[\delta_2](\Omega;\beta) }^p \\
        \leq& \bar{c} n^{p-1} (1+\delta_2)^{d+p} \left( \frac{1+\kappa_0 \kappa_1 \delta_2}{1-\kappa_0 \kappa_1 \delta_2} \right)^{|\beta|} \sum_{i = 1}^n \int_{ \Omega } \int_{B(0,\frac{\delta_2}{1-\delta_2}  \eta(\bx_i)) } \frac{ |u(\bx_i+\frac{1}{n}\bs)-u(\bx_i)|^p }{ \gamma(\bx_i)^{\beta} |\delta_2 \eta(\bx_i)|^{d+p} } \, \rmd \bs \, \rmd \bx.
		\end{split}
	\end{equation*}
	Clearly $\bx_i \in \Omega$, so we can perform a change of variables in the outer integral; letting $\by = \bx_i = \bx + \frac{i-1}{n} \bs$ 
    we get
	\begin{equation*}
		\begin{split}
			&[u]_{ \mathfrak{W}^{p}[\delta_2](\Omega;\beta) }^p \\
            \leq& \bar{c} n^{p-1} (1+\delta_2)^{d+p} \left( \frac{1+\kappa_0 \kappa_1 \delta_2}{1-\kappa_0 \kappa_1 \delta_2} \right)^{|\beta|} \sum_{i = 1}^n \int_{ \Omega } \int_{B(0,\frac{\delta_2}{1-\delta_2}  \eta(\by)) } \frac{ |u(\by+\frac{1}{n}\bs)-u(\by)|^p }{ \gamma(\by)^{\beta} |\delta_2 \eta(\by)|^{d+p} } \, \rmd \bs \, \rmd \by \\
			=& \bar{c} n^{p} (1+\delta_2)^{d+p} \left( \frac{1+\kappa_0 \kappa_1 \delta_2}{1-\kappa_0 \kappa_1 \delta_2} \right)^{|\beta|} \int_{ \Omega } \int_{B(0,\frac{\delta_2}{1-\delta_2}  \eta(\by)) } \frac{ |u(\by+\frac{1}{n}\bs)-u(\by)|^p }{ \gamma(\by)^{\beta} |\delta_2 \eta(\by)|^{d+p} } \, \rmd \bs \, \rmd \by .
		\end{split}
	\end{equation*}
	Now perform a change of variables in the inner integral by $\bz = \frac{\bs}{n}$ to obtain
	\begin{equation*}
		\begin{split}
			&[u]_{ \mathfrak{W}^{p}[\delta_2](\Omega;\beta) }^p \\
			\leq& \bar{c} n^{d+p} (1+\delta_2)^{d+p} \left( \frac{1+\kappa_0 \kappa_1 \delta_2}{1-\kappa_0 \kappa_1 \delta_2} \right)^{|\beta|} \int_{ \Omega } \int_{B(0,\frac{\delta_2}{1-\delta_2} \frac{ \eta(\by) }{n}) } \frac{ |u(\by+\bz)-u(\by)|^p }{ \gamma(\by)^{\beta} |\delta_2 \eta(\by)|^{d+p} } \, \rmd \bz \, \rmd \by \\
			=& \bar{c} \left( \frac{n \delta_1 (1+\delta_2)}{\delta_2} \right)^{d+p} \left( \frac{1+\kappa_0 \kappa_1 \delta_2}{1-\kappa_0 \kappa_1 \delta_2} \right)^{|\beta|} \int_{ \Omega } \int_{B(0,\frac{\delta_2}{1-\delta_2} \frac{ \eta(\by) }{n}) } \frac{ |u(\by+\bz)-u(\by)|^p }{ \gamma(\by)^{\beta} |\delta_1 \eta(\by)|^{d+p} } \, \rmd \bz \, \rmd \by.
		\end{split}
	\end{equation*}
	By taking $n \in \bbN$ such that	\begin{equation*}
		\frac{\delta_2}{\delta_1 ( 1 - \delta_2)} < n < \frac{2\delta_2}{\delta_1 ( 1 - \delta_2)},
	\end{equation*}
        we have
	\begin{equation*}
		\begin{split}
			&[u]_{ \mathfrak{W}^{p}[\delta_2](\Omega;\beta) }^p 
			\leq \left( \frac{2(1+\delta_2)}{1-\delta_2} \right)^{d+p}  \left( \frac{1+\kappa_0 \kappa_1 \delta_2}{1-\kappa_0 \kappa_1 \delta_2} \right)^{|\beta|} [u]_{\mathfrak{W}^p[\delta_1](\Omega;\beta)}^p,
		\end{split}
	\end{equation*}
	as desired.
\end{proof}

\begin{proof}[Proof of \Cref{thm:EnergySpaceIndepOfKernel}]
    First, by \eqref{assump:VarProb:Kernel} we have
    \begin{equation*}
        C(\rho)^{-1} \mathds{1}_{ B(0,c_\rho) }(\bx) \leq \rho(|\bx|) \leq C(\rho) \mathds{1}_{ B(0,1) }(\bx)
    \end{equation*}
    for $C(\rho) > 1$. Next, by \eqref{assump:Localization}, 
    \begin{gather*}
        \frac{ \eta(\bx) }{\kappa_0} \leq \lambda(\bx) \leq \kappa_0 \eta(\bx).
    \end{gather*}
   Therefore,
    \begin{equation*}
    \begin{split}
        &\frac{1}{C(\rho)} \mathds{1}_{ \{ |\by-\bx| < \frac{c_\rho}{\kappa_0} \delta \eta(\bx)  \} } \frac{|u(\bx)-u(\by)|^p}{|\kappa_0 \delta \eta(\bx)|^{d+p} } \\
        &\leq \rho \left( \frac {|\by-\bx|}{\delta \lambda(\bx) } \right) \frac{|u(\bx)-u(\by)|^p}{|\delta \lambda(\bx)|^{d+p} }  \\
        &\leq C(\rho) \kappa_0^{d+p} \mathds{1}_{ \{ |\by-\bx| < \kappa_0 \delta \eta(\bx) \} } \frac{|u(\bx)-u(\by)|^p}{|\delta \eta(\bx)|^{d+p} }.
    \end{split}
    \end{equation*}
   The conclusion then follows -- after multiplying by $\gamma(\bx)^{-\beta}$ and integrating -- from the assumptions on $\delta$ and then using the independence of the bulk horizon $\delta$ as in \Cref{thm:InvariantHorizon}.
\end{proof}

    \section{Boundary-localized convolutions}\label{sec:Apdx:Conv}

        \begin{lemma}\label{lma:CoordChange2}
For a fixed $\bz \in B(0,1)$ and generalized distance $\lambda$ satisfying \eqref{assump:Localization},
    define the function $\bszeta_\bz^\veps : \Omega \to \bbR^d$ by 
    \begin{equation*}
    \bszeta_\bz^\veps(\bx) := \bx + \lambda_\veps(\bx) \bz, \quad
    \forall \veps \in (0,\underline{\delta}_0).
    \end{equation*}
    Then the following hold
    for all $\bx$, $\by \in \Omega$, for all $\delta \in (0,\underline{\delta}_0)$, and for all $\veps \in (0,\underline{\delta}_0)$:
    \begin{equation}\label{eq:bszetaproperties}
        \begin{gathered}
            \det \grad \bszeta_\bz^\veps(\bx) = 1 + \grad \lambda_\veps(\bx) \cdot \bz > 1-\kappa_1 \veps > \frac{2}{3}, \\
            \bszeta_\bz^\veps(\bx) \in \Omega \text{ and } 0 < (1- \kappa \veps) \eta(\bx) \leq \eta(\bszeta_\bz^\veps(\bx)) \leq (1+\kappa \veps) \eta(\bx) , \\
            0 < (1-\kappa \veps) \gamma(\bx) \leq \gamma(\bszeta_\bz^\veps(\bx)) \leq (1+\kappa \veps) \gamma(\bx), \text{ and } \\
            0 < (1- \kappa \veps) |\bx-\by| \leq | \bszeta_\bz^\veps(\bx) - \bszeta_\bz^\veps(\by)| \leq (1+\kappa \veps) |\bx-\by|.
        \end{gathered}
    \end{equation} 
\end{lemma}

\begin{proof}
    The positive lower bound on $\det \grad \bszeta_\bz^\veps$ follows from the properties of $\lambda$ and the assumption on $\underline{\delta}_0$. 
    The second line in \eqref{eq:bszetaproperties} follows from the Lipschitz continuity of the distance function $\eta$ and the bound $\lambda(\bx) \leq \kappa_0 \eta(\bx)$ implied by \eqref{assump:Localization}:
    \begin{equation*}
        |\eta(\bszeta_\bz^\veps(\bx))-\eta(\bx)| \leq \lambda_\veps(\bx)|\bz| \leq \kappa_0 \veps \eta(\bx) \leq \kappa \veps \eta(\bx).
    \end{equation*}
    The third line is shown similarly by using \eqref{assump:weight} and the fact that $\eta(\bx) = \dist(\bx,\p \Omega) \leq \dist(\bx,\Gamma)$;
    \begin{equation*}
        |\gamma(\bszeta_\bz^\veps(\bx))-\gamma(\bx)| \leq \kappa_1 \lambda_\veps(\bx)|\bz| 
        \leq \kappa_1 \kappa_0 \veps \eta(\bx) 
        \leq \kappa_1 \kappa_0 \veps d_{\Gamma}(\bx)
        \leq \kappa_1 \kappa_0^2 \veps \gamma(\bx)
        = \kappa \veps \gamma(\bx).
    \end{equation*}
    The fifth line of \eqref{eq:bszetaproperties} follows from the estimate
    \begin{equation*}
        \big| | \bszeta_\bz^\veps(\bx) - \bszeta_\bz^\veps(\by)| - |\bx-\by| \big| \leq |\lambda_\veps(\bx) - \lambda_\veps(\by)| |\bz| \leq \kappa_1 \veps |\bx-\by| \leq \kappa \veps |\bx-\by|.
    \end{equation*}
\end{proof}

Similar estimates hold when $\lambda$ is replaced by the distance function $\eta$; for instance,
\begin{equation}\label{eq:Det:DistanceFxn}
    \det \grad (\bx + \eta_\veps(\bx)\bz) = 1 + \grad \eta_\veps(\bx) \cdot \bz > 1 - \veps.
\end{equation} 

Using \eqref{lma:CoordChange2} and elementary algebraic estimates, we obtain the inequalities for $\beta \in \bbR$:
\begin{equation}\label{eq:CoordChange2:Conseq}
    \frac{1+\kappa \veps}{1-\kappa \veps} \frac{|\by-\bx|}{\eta_\delta(\bx)} \leq \frac{ |\bszeta_{\bz}^{\veps}(\by) - \bszeta_{\bz}^{\veps}(\bx)| }{\eta_\delta(\bszeta_{\bz}^{\veps}(\bx))} \quad \text{ and } 
    \quad \gamma( \bszeta_\bz^\veps(\bx))^{\beta} \leq \left(\frac{1+\kappa \veps}{1-\kappa \veps} \right)^{|\beta|} \gamma(\bx)^{\beta}.
\end{equation}

\begin{proof}[Proof of \Cref{thm:KnownConvResults:Nonlocal}, \eqref{eq:Intro:ConvEst}]
    The estimate $\vnorm{K_\delta u - u}_{L^p(\Omega;\beta+p)} \leq C \vnorm{(K_\delta u - u)\eta^{-1}}_{L^p(\Omega;\beta)}$ follows from applying the inequality $\eta(\bx) \leq \dist(\bx,\Gamma) \leq \kappa_0 \gamma(\bx)$, valid for all $\bx \in \Omega$, to the integrands. 
    Next, Jensen's inequality applied with the measure $\psi_\delta(\bx,\by) \, \rmd \by$ gives
    \begin{equation*}
        \begin{split}
            \vnorm{(K_\delta u - u)\eta^{-1}}_{L^p(\Omega;\beta)}^p &= \int_{\Omega} \frac{1}{\gamma(\bx)^\beta} \left| \delta \int_{\Omega} \psi_\delta(\bx,\by) \frac{(u(\bx) - u(\by))}{\delta \eta(\bx)} \, \rmd \by \right|^p \, \rmd \bx \\
            &\leq \delta^p \int_{\Omega} \frac{1}{\gamma(\bx)^\beta} \int_{\Omega} \psi_\delta(\bx,\by) \frac{|u(\bx) - u(\by)|^p}{\eta_\delta(\bx)^p} \, \rmd \by \, \rmd \bx.
        \end{split}
    \end{equation*}
    After multiplication by a suitable constant, the renormalized kernel satisfies \eqref{assump:VarProb:Kernel}. Then the estimate follows from the kernel equivalence given in \Cref{thm:EnergySpaceIndepOfKernel}.
\end{proof}

\begin{proof}[Proof of \Cref{thm:KnownConvResults:Nonlocal}, \eqref{eq:Intro:ConvEst:Deriv}]
    Since $\int_{\Omega} \psi_\delta(\bx,\by) \, \rmd \by = 1$, its gradient in $\bx$ vanishes, and so
	\begin{equation*}
		\begin{split}
			\grad K_{\delta}u (\bx) &= \intdm{\Omega}{\grad_{\bx} \psi_{\delta}(\bx,\by) u(\by)}{\by}  
			= \intdm{\Omega}{\grad_{\bx} \psi_{\delta}(\bx,\by) (u(\by)-u(\bx))}{\by}.
		\end{split}
	\end{equation*}
    By H\"older's inequality 
	\begin{equation}\label{eq:Intro:ConvEst:Deriv:Pf1}
		\begin{split}
			&\Vnorm{ \grad K_{\delta} u }_{L^p(\Omega;\beta)}^p \\
			&\leq \int_{\Omega}  \left( \eta_\delta(\bx)  \intdm{\Omega}{ |\grad_{\bx} \psi_{\delta}(\bx,\bz)|}{\bz}
   \right)^{p-1}
   \intdm{\Omega}{\frac{|\grad_{\bx} \psi_{\delta}(\bx,\by)|}{\gamma(\bx)^\beta \eta_\delta(\bx)^{p-1}} |u(\by)-u(\bx)|^p}{\by} \rmd \bx.
		\end{split}
	\end{equation}
        Now,
        
        \begin{equation*}
		\begin{split}
        \grad_{\bx} {\psi}_{\delta}(\bx,\by) &= (\psi')_\delta(\bx,\by) \left( \frac{\bx-\by}{|\bx-\by|} \frac{1}{\lambda_\delta(\bx)} - \frac{|\bx-\by|}{\lambda_\delta(\bx)} \frac{\grad \lambda_\delta(\bx)}{\lambda_\delta(\bx)} \right) \\
		& \quad - \psi_{\delta}(\bx,\by)  \frac{ \grad \lambda_\delta(\bx) }{\lambda_\delta(\bx)} d ,
		\end{split}
	\end{equation*}
        so therefore using the support of $\psi$ and the fact that $|\grad \lambda_\delta| \leq 1/3$
        \begin{equation*}
            |\grad _{\bx} \psi_{\delta}(\bx,\by)| \leq \frac{C}{\eta_\delta(\bx) } \Big( \psi_{\delta}(\bx,\by) 
		+ (|\psi'|)_{\delta}(\bx,\by) \Big),  \quad \forall \bx, \by \in \Omega.
        \end{equation*}
        We use this estimate in \eqref{eq:Intro:ConvEst:Deriv:Pf1}; since $\int_{\Omega} (|\psi'|)_\delta(\bx,\by) \, \rmd \by \leq C$ we have
        \begin{equation*}
		\begin{split}
		  \Vnorm{ \grad K_{\delta} u }_{L^p(\Omega;\beta)}^p 
		  \leq \int_{\Omega}
            \intdm{\Omega}{\frac{\psi_{\delta}(\bx,\by) 
		+ (|\psi'|)_{\delta}(\bx,\by)}{\gamma(\bx)^\beta \eta_\delta(\bx)^{p}} |u(\by)-u(\bx)|^p}{\by} \rmd \bx.
		\end{split}
	\end{equation*}
        Then the result follows from the kernel equivalence as in \Cref{thm:EnergySpaceIndepOfKernel}.
\end{proof}

    \begin{proof}[Proof of \Cref{thm:KnownConvResults:Nonlocal}, \eqref{eq:Comp:ConvolutionEstimate}]
    Let $u \in \mathfrak{W}^{p}[\delta](\Omega;\beta)$. For the $\lambda$ appearing in the definition of $K_\delta u$, define the function $\bszeta_\bz^\veps(\bx) = \bx + \lambda_\veps(\bx) \bz$ as in \Cref{lma:CoordChange2}.
    Now define the function $\theta(\veps) := \frac{1-\kappa \veps}{1+\kappa \veps}$ for $\veps \in (0,\underline{\delta}_0)$, where $\kappa$ is defined in \eqref{assump:Horizon}.
    Then by Jensen's inequality we have the following estimate for the quantity $I$:
    \begin{equation*}
        \begin{split}
        I &:= \int_{\Omega \cap \{\eta(\bx) < \theta(\veps) r \} } \int_\Omega \frac{1}{\gamma(\bx)^\beta} \rho \left( \frac{|\bx-\by|}{\eta_{\theta(\veps)\delta}(\bx)} \right) \frac{ |K_\veps u(\bx) - K_\veps u(\by)|^p }{ \eta_{\theta(\veps)\delta}(\bx)^{d+p} } \, \rmd \by \, \rmd \bx \\
        &\leq \int_{B(0,1)} \psi(|\bz|) \int_{\Omega \cap \{\eta(\bx) < \theta(\veps) r \} } \int_\Omega \frac{1}{\gamma(\bx)^\beta} \rho \left( \frac{|\bx-\by|}{\eta_{\theta(\veps)\delta}(\bx)} \right) \frac{ |u(\bszeta_\bz^\veps(\bx))-u(\bszeta_\bz^\veps(\by))|^p }{  \eta_{\theta(\veps)\delta}(\bx)^{d+p}  } \, \rmd \by \, \rmd \bx \, \rmd \bz.
        \end{split}
    \end{equation*}
    Now we use the inequalities \eqref{eq:CoordChange2:Conseq}.
    Since $\rho$ is nonincreasing, we get
    \begin{equation*}
    \begin{split}
        I
        &\leq
        \frac{ (1+\kappa \veps)^{d+p} }{ \theta(\veps)^{d+p+|\beta|} (1-\kappa_1 \veps)^{2} }
        \int_{B(0,1)} \psi \left( |\bz| \right) \int_{\Omega \cap \{ \eta(\bszeta_\bz^\veps(\bx)) < r \} } \int_{\Omega} 
        \rho \left( \frac{ |\bszeta_{\bz}^{\veps}(\by) - \bszeta_{\bz}^{\veps}(\bx)| }{\eta_\delta(\bszeta_{\bz}^{\veps}(\bx))}  \right) 
        \\
        &\qquad \frac{ \left|  u( \bszeta_{\bz}^\veps(\bx) ) - u( \bszeta_{\bz}^\veps(\by)  ) \right|^p }{ \gamma(\bszeta_\bz^\veps(\bx))^\beta \eta_\delta(\bszeta_{\bz}^{\veps}(\bx))^{d+p}}  \det \grad \bszeta_{\bz}^{\veps}(\bx) \det \grad \bszeta_{\bz}^{\veps}(\by)  \, \rmd \by \, \rmd \bx \, \rmd \bz.
    \end{split}
    \end{equation*}
    We apply the change of variables $\bar{\by} = \bszeta_{\bz}^\veps(\by) $, $\bar{\bx} = \bszeta_{\bz}^\veps(\bx)$, and -- using that $\bszeta_\bz^\veps(\bx) \in \Omega$ for all $\bx \in \Omega$ -- obtain 
    \begin{equation*}
    \begin{split}
        I&\leq 
        \frac{ (1+\kappa \veps)^{d+p} }{ \theta(\veps)^{d+p+|\beta|} (1-\kappa_1 \veps)^{2} }
        \int_{ \Omega \cap \{ \eta(\bx) < r \} } \int_{ \Omega }
        \rho \left( \frac{ |\bar{\by} - \bar{\bx}| }{ \eta_\delta(\bar{\bx})}  \right) \frac{ \left|  u(\bar{\bx}) - u(\bar{\by}) \right|^p }{\gamma(\bx)^\beta \eta_\delta(\bar{\bx})^{d+p}} \rmd \bar{\by} \, \rmd \bar{\bx}.
    \end{split}
    \end{equation*}
    Then \eqref{eq:Comp:ConvolutionEstimate} is established by setting $\phi(\veps) := \frac{ (1+\kappa \veps)^{d+p} }{ \theta(\veps)^{d+p+|\beta|} (1-\kappa_1 \veps)^{2} }$.
    \end{proof}

\begin{proof}[Proof of \Cref{thm:KnownConvResults:Nonlocal}, item 3)]
    Recall the definition of $\theta(\veps)$ in \eqref{eq:Comp:ConvolutionEstimate}. First, using \Cref{thm:InvariantHorizon} we obtain
    \begin{equation*}
        [K_\veps u-u]_{ \mathfrak{W}^p[\delta](\Omega;\beta) }^p \leq C [K_\veps u - u]_{ \mathfrak{W}^p[\theta(\veps)\delta](\Omega;\beta) }^p,
    \end{equation*}
    where $C$ is independent of both $\delta$ and $\veps$.
    Applying \eqref{eq:Comp:ConvolutionEstimate} with kernel $\rho(x) = \frac{ (d+p)\overline{C}_{d,p} }{ \sigma(\bbS^{d-1}) } \mathds{1}_{(-1,1)}(x)$ and $r > 0$, we get
    \begin{equation*}
        \begin{split}
		&\int_{ \Omega \cap \{ \eta(\bx) < \theta(\veps) r \} } \int_{ \Omega } \mathds{1}_{ \{ |\by-\bx| < \eta_{\theta(\veps)\delta}(\bx) \} } \frac{|K_{\veps} u(\bx) - K_{\veps}u(\by) - ( u(\bx) - u(\by) )|^p}{ \gamma(\bx)^{\beta} \eta_{\theta(\veps)\delta}(\bx)^{d+p}  } \, \rmd \by \, \rmd \bx \\
        &\leq 2^{p-1}(1+\phi(\veps)) \int_{ \Omega \cap \{ \eta(\bx) < r \} } \int_{ \Omega } \mathds{1}_{ \{ |\by-\bx| < \eta_\delta(\bx) \} } \frac{| u(\bx) - u(\by)|^p}{ \gamma(\bx)^{\beta} \eta_\delta(\bx)^{d+p}  } \, \rmd \by \, \rmd \bx.
        \end{split}
    \end{equation*}
    Since $\phi(\veps)$ is bounded from above independent of $\veps$, we can apply continuity of the right-hand side integral to conclude that for any $\tau > 0$ there exists $r > 0$ such that
	\begin{equation*}
		\sup_{\veps \in (0,\underline{\delta}_0)} \int_{ \Omega \cap \{ \eta(\bx) < \theta(\veps) r \} } \int_{ B(\bx,\eta_{\theta(\veps) \delta}(\bx)) } 
        \frac{|K_{\veps} u(\bx) - K_{\veps}u(\by) - ( u(\bx) - u(\by) )|^p}{ \gamma(\bx)^{\beta} \eta_{\theta(\veps)\delta}(\bx)^{d+p}  } \, \rmd \by \, \rmd \bx < \tau.
	\end{equation*}
    Now, whenever $|\bx-\by| \leq \delta \theta(\veps) \eta_(\bx)$ and $\eta(\bx) \geq \theta(\veps) r$ we have $\eta_\delta(\bx) \geq \delta \theta(\veps) r$ and $\eta(\by) \geq \eta(\bx) -\kappa_1|\by-\bx| \geq (1-\kappa_1 \delta \theta(\veps)) r$. Further, \eqref{eq:Comp1} holds. So therefore
	\begin{equation*}
		\begin{split}
		&\int_{ \Omega } \int_{ \Omega  \cap \{ \eta(\bx) \geq \theta(\veps) r \} } \mathds{1}_{ \{ |\by-\bx| < \eta_{\theta(\veps)\delta}(\bx) \} } \frac{|K_{\veps} u(\bx) - K_{\veps}u(\by) - ( u(\bx) - u(\by) )|^p}{ \gamma(\bx)^{\beta} \eta_{\theta(\veps)\delta}(\bx)^{d+p}  } \, \rmd \by \, \rmd \bx \\
		&\leq \frac{C}{(\delta r)^{p} } \Vnorm{ K_{\veps} u - u }_{L^p(\Omega;\beta)}^p.
		\end{split}
	\end{equation*}
	Since $\Vnorm{ K_{\veps} u - u }_{L^p(\Omega;\beta)}^p \to 0$ as $\veps \to 0$, we conclude that
	\begin{equation*}
		\limsup\limits_{\veps \to 0} [K_{\veps} u - u]_{ \mathfrak{W}^p[\delta](\Omega;\beta)}^p < \tau + C(r) \limsup\limits_{\veps \to 0} \Vnorm{ K_{\veps} u - u }_{L^p(\Omega;\beta)}^p = \tau,
	\end{equation*}
	and so the convergence follows since follows since $\tau > 0$ is arbitrary.  
\end{proof}

		\section{Muckenhoupt weights}
		
		A nonnegative function $w \in L^1_{loc}(\bbR^d)$ is called an $A_p$-\textit{weight} if
		\begin{equation*}
			\begin{gathered}
				[w]_{A_p} := \sup_{\bx_0 \in \bbR^d, R > 0} A_p(w,B(\bx_0,R)) < \infty, \\
				\text{ where } A_p(w,B(\bx_0,R)) := \left( \fint_{B(\bx_0,R)} w(\bx) \, \rmd \bx \right) \left( \fint_{B(\bx_0,R)} w(\bx)^{\frac{-1}{p-1}} \, \rmd \bx \right)^{p-1}.
			\end{gathered}    
		\end{equation*}
		
		\begin{lemma}\label{lma:ApWeight}
        Let $d \geq 2$, and let $\ell \in \{0,1,2,\ldots,d-1\}$. 
        Set $\bx = (\bx',\bx'') \in \bbR^{\ell} \times \bbR^{d-\ell}$, and define $w(\bx) = |\bx''|^{-\beta}$.
		Then $w$ is an $A_p$ weight if and only if
        $(d-\ell)(1-p) < \beta < d - \ell$.
		\end{lemma}
		
		\begin{proof}
			First, suppose $\ell = 0$, so that $\bx'' = \bx$. Then it is well-known that
			\begin{equation*}
				w(\bx) = |\bx|^{-\beta} \text{ is an } A_p \text{ weight} \qquad \Leftrightarrow \qquad d(1-p) < \beta < d.
			\end{equation*}

				For the first implication, suppose that $[w]_{A_p} < \infty$. Then it holds that
				\begin{equation*}
					A_p(w,B(0,1)) = \left( \fint_{B(0,1)} w(\bx) \, \rmd \bx \right) \left( \fint_{B(0,1)} w(\bx)^{\frac{-1}{p-1}} \, \rmd \bx \right)^{p-1} \leq [w]_{A_p} < \infty.
				\end{equation*}
				The first integral is finite only if $\beta < d$, and the second integral is finite only if $\beta > d(1-p)$.
				
				For the reverse implication, we assume $d(1-p) < \beta < d$, and consider two different cases for estimating $A_p(w,B(\bx_0,R))$. First assume that $R < \frac{|\bx_0|}{2}$. Then for $\bx \in  B(\bx_0,R)$
				\begin{equation*}
					\begin{split}
						|\bx_0| \leq |\bx-\bx_0| + |\bx| \leq R + |\bx|, \qquad &\Rightarrow \qquad |\bx| > \frac{|\bx_0|}{2}, \\
						|\bx| \leq |\bx-\bx_0| + |\bx_0| \leq R + |\bx_0|, \qquad &\Rightarrow \qquad |\bx| \leq \frac{3|\bx_0|}{2}.
					\end{split}
				\end{equation*}
				Therefore 
				\begin{equation*}
					A_p(w,B(\bx_0,R)) \leq 3^{|\beta|}.
				\end{equation*}
				Now assume $R > \frac{|\bx_0|}{2}$. Then $B(\bx_0,R) \subset B(0,3R)$, and so
				\begin{equation*}
					A_p(w,B(\bx_0,R)) \leq C(d,p) A_p(w,B(0,3R).
				\end{equation*}
				But a change of coordinates gives $A_p(w,B(0,3R)) = A_p(w,B(0,1)) < \infty
				$, since $d(1-p) < \beta < d$.

			Now assume that $\ell \in \{1,\ldots,d-1\}$. Let $B(\bx_0,R) \subset \bbR^d$. 
			Then there exists $0 < \kappa < 1$ depending only on $d$ such that
			\begin{equation*}
				\cC(\bx_0,\kappa R) \subset B(\bx_0,R) \subset \cC(\bx_0,\kappa R),
			\end{equation*}
			where $\cC(\bx_0,R)$ is the cylindrical domain
			\begin{equation*}
				\cC(\bx_0,R) := \{ \bx = (\bx',\bx'') \in \bbR^d  \, : \, \vnorm{\bx'-\bx_0'}_{\infty} < R, \quad |\bx'-\bx_0''| \leq R \},
			\end{equation*}
			and where $\vnorm{\bx'}_\infty := \max_{1\leq i \leq \ell} |x_i|$. Moreover, $|B(\bx_0,R)| \approx_d \cC(\bx_0,R)$.
			It follows that 
			\begin{equation*}
				A_p(w,B(\bx_0,R)) \approx A_p(w,\cC(\bx_0,R))
			\end{equation*}
			where the constant of comparison depends only on $d$ and $p$. Moreover, $A_p(w,\cC(\bx_0,R)) = A_p(w,B_{d-\ell}(\bx_0'',R))$, where $B_{d-\ell}(\bx_0'',R))$ is the Euclidean ball in $d-\ell$ dimensions centered at $\bx_0''$ with radius $R$. This is exactly the form of the weight in the first case considered, with $d-\ell$ in place of $d$. Therefore, by applying the analysis of the first case we get that
			\begin{equation*}
				w \text{ is an } A_p \text{ weight} \qquad \Leftrightarrow \qquad (d-\ell)(1-p) < \beta < d - \ell.
			\end{equation*}
			
		\end{proof}


\bibliographystyle{plain}

\end{document}